\documentclass[11pt, a4paper]{article}
\usepackage[utf8]{inputenc}
\usepackage{amsmath,amssymb,amsfonts,dsfont}
\usepackage{amsthm}
\usepackage{a4wide}
\usepackage[usenames,dvipsnames]{xcolor}

\usepackage[unicode, bookmarks, colorlinks, breaklinks, 
	pdfauthor={Trien-Thuyen Dang},
]{hyperref}  
\usepackage[capitalize,nameinlink,noabbrev]{cleveref}

\newtheorem{theorem}{Theorem}
\newtheorem{proposition}[theorem]{Proposition}%
\newtheorem{remark}{Remark}%

\newtheorem{definition}{Definition}%

\newtheorem{lemma}[theorem]{Lemma}
\numberwithin{equation}{section}

\newcommand{\leb}[1]{\mathop \le \limits_{(#1)}}
\newcommand{\geb}[1]{\mathop \ge \limits_{(#1)}}
\newcommand{\lleb}[1]{\mathop < \limits_{(#1)}}
\newcommand{\ggeb}[1]{\mathop > \limits_{(#1)}}
\newcommand{\eqb}[1]{\mathop = \limits_{(~#1)}}

\newcommand{\hh}{\hspace*{.48in}}
\newcommand{\hk}{\hspace*{.12in}}

\def\beq{\begin{eqnarray*}}\def\eeq{\end{eqnarray*}}
\def\bq{\begin{equation}}\def\eq{\end{equation}}

\newcommand{\NN}{{\mathbb N}}

\newcommand{\RR}{{\mathbb R}}

\newcommand{\email}[1]{\href{mailto:#1}{\texttt{#1}}}

\newcommand{\cv}[1][]{%
\ifthenelse{\isempty{#1}}{\xrightarrow[\hphantom{~2~}]{}}{\xrightarrow[\hphantom{~2~}]{#1}}%
}
\newcommand{\wcv}[1][]{%
\ifthenelse{\isempty{#1}}{\xrightharpoonup[\hphantom{~2~}]{}}{\xrightharpoonup[\hphantom{~2~}]{#1}}%
}

\def\XXint#1#2#3{{\setbox0=\hbox{$#1{#2#3}{\int}$ }
\vcenter{\hbox{$#2#3$ }}\kern-.6\wd0}}

\allowdisplaybreaks

\begin{document}

\title{Green representations and global H\"{o}lder continuity  for solutions of elliptic equations}

\author{ Duong Minh Duc\footnote{
    Department of Mathematics and Computer Science, Vietnam
    National University Ho Chi Minh City--University of Science, 227
    Nguyen Van Cu Street, Phuong Cho Quan, Ho Chi Minh City, Vietnam
  (\email{dmduc@hcmus.edu.vn}).}}

\date{}

\maketitle

\begin{abstract}Let $N\in\NN$ and $u$ be a weak solution of equation $\displaystyle Lu\equiv -
\sum_{i,j=1}^{N}\frac{\partial}{\partial x_{j}}(\frac{\partial u}{\partial x_{i}}b^{ij})= f$ in $\Omega\subset \RR^{N}$. We obtain functions $G$  and  $H_{l}$  on $\Omega\times \Omega$ for every $l\in\{1,\cdots,N\}$ having following properties: if $f$ is in $L^{1}(\Omega)$, then  \\\hh
$\displaystyle \int_{\Omega}G(x,y)f(x)dx= u(y),$\\\hh
$\displaystyle\int_{\Omega}H_{l}(x,y)f(x)dx= -\frac{\
\partial u}{\partial x_{l}}(y)\hk a.e~y\in \Omega, \forall~l\in\{1,\cdots,N\}.$\\\hk 
Applying these results and using properties of Riesz potentials, we get the global H\"{o}lder continuity of $u$. $L$ may be not strictly nor uniformly elliptic and $u$ may  vanish on a part $A$ of the boundary and is free outside of $A$.  \end{abstract}





\maketitle
{{\bf Acknowledgments}.~~We would like to thank Professors James Eells,  Isabel M. C. Salavessa and Alberto Verjovski who have attracted us to the regularity for elliptic equations. We also would like to thank Dr. Thuyen Dang for helpful discussions.}

{\section{Introduction}\hk
 In this paper we study Green presentations for solutions and their gradients of elliptic equations and  obtain their global H\"{o}lder continuity by  combining these presentations with properties of Riesz operators.\\\hk
Let $\Omega$ be a bounded open subset of the euclidean space $\RR^{N}$ with $N\ge 2$, $\partial\Omega$ be the boundary $\Omega$, $A$ be a subset of $\partial\Omega$,  $t\in (\frac{2N^{2}+2N-2}{N^{2}+2N-1},2)$, $r=\frac{t(N+1)-2}{N-t}>2$, $\overline{r}\in (2,r)$,  $b^{ij}$ be a measurable function on $\Omega$ for every $i,j \in \{1,\cdots,N\}$, $B$ be $\{b^{ij}\}_{i,j}$,  $b$ and $\overline{b}$ be non-negative measurable functions on $\Omega$ such that $b^{ij}= b^{ji}$ for every $i,j \in \{1,\cdots,N\}$. We consider following conditions
\begin{equation} b|\xi|^{2}\le b^{ij}\xi_{i}\xi_{j}\le \overline{b}|\xi|^{2}\hh\forall~\xi=(\xi_{1},\cdots,\xi_{j})\in \RR^{N},
\label{c1}
\end{equation} 
 \begin{equation} b^{-1}\in L^{\frac{t}{2-t}}(\Omega),
\label{c3}
 \end{equation}
\begin{equation} b\in L^{\frac{\overline{r}}{\overline{r}-2}}(\Omega),
\label{c2}
 \end{equation}
 \begin{equation} \overline{b}\in L^{\frac{r}{r-2}}(\Omega),
\label{c6}
 \end{equation}
 \begin{equation} b^{-\frac{1}{2}}\overline{b}\in L^{\frac{2\overline{r}}{\overline{r}-2}}(\Omega),
\label{c4}
 \end{equation}\hk
  We put
 \begin{equation} L(u)= -\sum_{i,j=1}^{N}\frac{\partial}{\partial x_{j}}(\frac{\partial u}{\partial x_{i}}b^{ij}).
\label{c5}
 \end{equation}\hk
  Let $A$ be admissible with respect to $\Omega$ (see Definition \ref{d41}), $W_{B,A}(\Omega)$ be defined  as in Definition \ref{d44}. Then $W_{B,A}(\Omega)$  is  contained in the family of measurable function on $\Omega$ having first order generalized partial derivatives (see Lemma \ref{5z}).   If $B=\{\delta^{i}_{j}\}_{i,j}$ and $A = \partial\Omega$, then $W_{B,A}(\Omega)$ is the usual Sobolev'space $W^{1,2}_{0}(\Omega)$.\\\hk
   Let $f\in L^{1}(\Omega)$ and $u$ be in  $W_{B,A}(\Omega)$  such that $L(u)=f$ in weak sense, that is
\begin{equation}\int_{\Omega}\sum_{i,j=1}^{N}\frac{\partial u}{\partial x_{i}}\frac{\partial \varphi}{\partial x_{j}}b^{ij} dx = \int_{\Omega}f\varphi dx\hh\forall~\varphi \in W_{B,A}(\Omega).
 \label{eq}
 \end{equation} \hk
 Let $G$ be  a function on $\Omega\times \Omega$. If 
 \begin{equation}\int_{\Omega}G(x,y)f(x)dx= u(x)\hh a.e.~ x\in \Omega,
 \label{eqg}
 \end{equation}
 then $G$ is called the Green function of $L$ and \eqref{eqg} is called the Green presentation of $u$.\\\hk  
 If  $A=\partial\Omega$,  $\partial \Omega$ is $C^{1}$-smooth, $b$ and $\overline{b}$ are constant,  we have \eqref{c1}-\eqref{c4} and  $W_{B,A}(\Omega)$ is $W^{1,2}_{0}(\Omega)$ defined in  \cite[page 287]{BR}. In this case, if $\Omega$ is $B(a,s)$ in $\RR^{N}$, $N\in \NN$, $N\ge 3$, $\mu$ is measure  of bounded variation on $\Omega$ and 
\[u(x)=\int_{\Omega}G(x,y)d\mu,\]
 Littman, Stampachia and Weinberger have proved   $Lu=\mu$ in \cite[Theorem (6.1)]{LSW}.\\\hk
   If $N\ge 3$,  $\overline{b}$  is of class $A_{2}$ and 
\begin{equation} \frac{s}{h}[\frac{\overline{b}(B(x,s)}{\overline{b}(B(x,h)}]^{\theta}\le c[\frac{b(B(x,s)}{b(B(x,h)}]^{2},\hk 0<s<h, x\in \RR^{N},
\label{cha}
\end{equation}
 where $b(E)=\int_{E}bdx$ and $\overline{b}(E)=\int_{E}\overline{b}dx$, for some constant $c$ and $\theta >2$,
then Chanillo and Wheeden have proved: there is the Green  function $G$ defined on $\Omega\times\Omega = B(z,r)\times B(z,r)$ such that
\\\hh$\displaystyle u(y)=\int_{B(z,r)}f(x)G(x,y)dx \hk for~a.e.~y\in B(z,\frac{1}{2}r), \frac{f}{\overline{b}}\in L^{\frac{\delta}{\delta-1}}(B(z,r)). $\\
with some $\delta < \frac{\theta}{2}$ (see   \cite[Theorem 1.8]{CW}). The result in \cite{CW}  give us the local representation of solutions by Green function. \\\hk
  If $b$ and $\overline{b}$ are constant, in \cite{GW} Gr\"{u}ter and Widman have proved that there is the Green  $G$ on $\Omega\times \Omega$ such that $G(.,y)$ is in  $W^{1,2}(\Omega\setminus B(y,r))\cap W^{1,1}_{0}(\Omega)$ for every $y\in \Omega$ and $r$ in $(0,\infty)$ and
  \[\int_{\Omega}\sum_{i,j=1}^{N}G(x,y)Lu(x)dx=u(y)\hh\forall~u \in C^{\infty}_{c}(\Omega), y\in \Omega.\]\hk
  On other hand, if $b$ and $\overline{b}$ are constant and $\nabla b^{ij}$ is in $L^{2}(\Omega)$ for every $i,j \in\{1,\cdots,N\}$, Ladyzenskaja and Ural'ceva have proved the global H\"{o}lder continuity of solutions  in \cite[p.203]{LU}. If  coefficients $b^{ij}$ have the  VMO regularity and $\RR^{N}\setminus \Omega$  satisfies the uniform m-thickness condition, Byun,  Palagachev and Shin obtained this global H\"{o}lder continuity   in \cite{BPS}. If  these  coefficients satisfy the Dini mean oscillation in $\Omega$ and $ \partial\Omega$ is $C^{1,Dini}$, Dong, Escauriaza and  Kim  have proved that  the solution $u$ is in $C^{1}(\overline{\Omega})$ and $C^{2}(\overline{\Omega})$  in \cite{DEK}. \\\hk
 If $b^{ij}$ are only measurable and the equation is uniformly elliptic, Trudinger has proved the local H\"{o}lder continuity on $\partial\Omega$ of $u$ in \cite{TR}.\\\hk
 Our main results as follows.
 \begin{theorem}~~Let  $A$ be  admissible with respect to $\Omega$ and $\zeta \in (0,1)$. Then there are  $t(\zeta)\in (\frac{2N^{2}+2N-2}{N^{2}+2N-1},2)$ and $\overline{r}(\zeta)\in (2,\frac{t(\zeta)(N+1)-2}{N-t(\zeta)})$ such that for every   $t\in (t(\zeta),2)$, $r=\frac{t(N+1)-2}{N-t}$ and $\overline{r}\in (2,\overline{r}(\zeta))$ satisfying   \eqref{c1}-\eqref{c4},  there is a   function $G$ on $\Omega\times \Omega$ having following properties:  If  $u\in W_{B,A}(\Omega)$ and $f\in L^{1}(\Omega)$ are as in \eqref{eq}, then 
\[\int_{\Omega}G(x,y)f(x)dx= u(x)\hh a.e.~ x\in \Omega\]
in two following cases\\\hk
$(i)$~~$N\in  \{2,\cdots,8\}$ and \eqref{c9} holds.\\\hk
$(ii)$~~$N\in  \{3,4,\cdots\}$, $A=\partial\Omega$, $b$ and $\overline{b}$ are constant.
\label{th1}
\end{theorem}\hk
 \begin{theorem}~~Let  $A$ be  admissible with respect to $\Omega$ and $\zeta \in (0,1)$. Then there are  $t(\zeta)\in (\frac{2N^{2}+2N-2}{N^{2}+2N-1},2)$ and $\overline{r}(\zeta)\in (2,\frac{t(\zeta)(N+1)-2}{N-t(\zeta)})$ such that for every   $t\in (t(\zeta),2)$, $r=\frac{t(N+1)-2}{N-t}$ and $\overline{r}\in (2,\overline{r}(\zeta))$ satisfying   \eqref{c1}-\eqref{c4},  there is  a function  $H_{l}$ on $\Omega\times \Omega$ for every $l\in\{1,\cdots,N\}$ having following properties:  If  $u\in W_{B,A}(\Omega)$ and $f\in L^{1}(\Omega)$ are as in \eqref{eq}, then 
\[\int_{\Omega}H_{l}(x,z)f(x)dx= -\frac{\
\partial u}{\partial x_{l}}(z)\hh a.e~z\in \Omega, \forall~l\in\{1,\cdots,N\},
\]
in two following cases\\\hk
$(i)$~~$N\in  \{2,\cdots,8\}$,  \eqref{c9} and \eqref{thuyen}  hold.\\\hk
$(ii)$~~$N\in  \{3,4,\cdots\}$, $A=\partial\Omega$, $b$ and $\overline{b}$ are constant, and $b^{ij}$ is Dini-continuous on $\Omega$ for every $i,j \in\{1,\cdots,N\}$ (see definition in \cite{GW}).
\label{th2}
\end{theorem}\hk
    By Remarks \ref{r4} and \ref{r4b}, we get global representation of $u$ and $\nabla u$ when $\overline{b}$ may be not in class $A_{2}$ and $L$ may be degenerate and non-uniform.  Furthermore the solution $u$ in our resuts may be only null in $A$ and free outside of $A$.\\\hk
Applying these theorems and using properties of Riesz operators, we obtain the global H\"{o}lder continuity of $u$ as follows.
\begin{theorem} Assume conditions in Theorems \ref{th1} and \ref{th2} hold, $\partial\Omega$ is minimally smooth $($see definition in  \cite[p.189]{ST}$)$ and one of following conditions\\\hk
$(i)$~~$N\in  \{2,\cdots,8\}$,  \eqref{c9} and \eqref{thuyen}  hold.\\\hk
$(ii)$~~$N\in  \{3,4,\cdots\}$, $A=\partial\Omega$, $b$ and $\overline{b}$ are constant, and $b^{ij}$ is Dini-continuous on $\Omega$ for every $i,j \in\{1,\cdots,N\}$.\\\hk
 We have\\\hk
$(a)$~~If  $f\in L^{\frac{N}{2-s}}(\Omega)$ with some $s$ be in $(0,1)$. Then there  are  $C_{3}\in (0,\infty)$ and $\tau\in (0,1)$  such that \\\hh
$|u(x)-u(y)|\le C_{3}|x-y|^{\tau}\hh\forall~x,y\in\Omega.$\\\hk
$(b)$ If $\theta\in (0,\frac{1}{4})$, $\zeta= \frac{1}{4}\theta$,$\epsilon\in (0,\frac{2\theta}{N})$ is as in Remark \ref{rep}, $b^{-1}\in L^{\frac{1}{\epsilon}}(\Omega)$ and there are non-negative measurable function $a_{0}$, $a_{1}$, $a_{1}$  on $\Omega$ such that $(a_{0}+ a_{1}+a_{2}^{2})\in L^{\frac{N}{2-2\theta}}(\Omega)$ and $|f|\le a_{2}|\nabla u|+ a_{1}u+a_{0}$. Then  $u$ is H\"{o}lder continuous up to the boundary of $\Omega$.
\label{th3}
\end{theorem}\hk
   In sections \ref{wi} and \ref{app}, we introduce and study the properties of functions spaces, which are used in  the proof of Theorems \ref{th1}-\ref{th3}. We shall use the techniques of Serrin  in \cite{SE}, but these technique are only applicable to local regularity. Therefore we have to modify auxiliary functions in \cite{SE} in section \ref{aux}. This section is long and has many computations, but we would like to make sure the validity of our results. In sections \ref{bo} we obtain the global boundedness of  solution of elliptic equations. These results  are essential for next sections. We get theorems \ref{th1} and \ref{th2}  in sections \ref{gr} and \ref{ho}. We prove Theorems \ref{th3} in section \ref{ho}.
 \section{Functions spaces}\label{wi}
\hk 
In this section, let $N\ge 2$, $\Omega$ be a bounded open in $\RR^{N}$,  $t\in (\frac{2N^{2}+2N-2}{N^{2}+2N-1},2)$, $r=\frac{t(N+1)-2}{N-t}$, $t^{\ast}=\frac{tN}{N-t}$, $b^{ij}$, $B$, $b$ be as in the introduction. Assume \eqref{c1} and \eqref{c3} hold.  We have
\[
 r= \frac{tN-(2-t)}{N-t}< t^{\ast},
\]
\[
r-2 = \frac{tN-(2-t) -2(N-t)}{N-t}= \frac{t(N+3) -2(N+1)}{N-t}
 \]
\[\ggeb{t>\frac{2N^{2}+2N-2}{N^{2}+2N-1}}\frac{2}{(N-t)(N^{2}+2N-1)}[(N^{2}+N-1)(N+3)-(N+1)(N^{2}+2N-1)]\]
\[=\frac{2(N^{2}+N-2)}{(N-t)(N^{2}+2N-1)}\ggeb{N\ge 2}0.\]\hk
Thus
\begin{equation}
  1< t<2 < r <t^{\ast}.
 \label{rt}
\end{equation}\\\hk
\begin{definition}Let $\overline{t}\in (\frac{2N^{2}+2N-2}{N^{2}+2N-1},2)$, $r=\frac{\overline{t}(N+1)-2}{N-\overline{t}}>2$, $\overline{r}\in (2,r)$. We put \\\hk
  $C^{1}(\overline{\Omega})$:~ the family of function $v|_{\Omega}$ with $v$ in  $C^{1}(U_{v})$, where $U_{v}$ is an open subset contained $\Omega$ in $\RR^{N}$, \\\hk
   $C^{1}(\Omega,A)$:~ the family of function $v|_{\Omega}$ with $v$ in  $C_{c}^{1}(\RR^{N})$ such that $v(x)=0$ for  every $x$ in $A$,\\\hk
     $\displaystyle|||v|||_{B} = \{\int_{\Omega}\sum_{i,j=1}^{N}\frac{\partial v}{\partial x_{i}}\frac{\partial v}{\partial x_{j}}b^{ij}dz\}^{\frac{1}{2}} \hh\forall~v \in C^{1}(\Omega,A),$\\\hk
    $\displaystyle|||v|||_{b} = \{\int_{\Omega}|\nabla  v|^{2}bdx\}^{\frac{1}{2}} \hh\forall~v \in C^{1}(\Omega,A),$\\\hk
      \label{space}
 \end{definition}\hk
  We have
  \begin{equation}
  |||w|||_{b} \leb{\eqref{c1}} |||w|||_{B}\hh\forall~w \in C^{1}(\Omega,A).
  \label{i1}
\end{equation}  
  \begin{definition} Let   $A\subset \partial\Omega$. We say ~$A$   is admissible with respect to $\Omega$, if there is a positive real number $C(r,\Omega,A)$ such that
  \begin{equation}\{\int_{\Omega} |v|^{r}dx\}^{\frac{1}{r}}\le C(r,\Omega,A)\{\int_{\Omega} |\nabla v|^{t}dx\}^{\frac{1}{t}} \hh\forall ~v \in C^{1}(\Omega,A).
\label{ineq}
\end{equation}   
  \label{d41}
 \end{definition}    
  Let $\Sigma_{N-1}$ be the unit sphere of $\RR^{N}$. We denote the canonical measure on $\Sigma_{N-1}$ by $\sigma$ (see  \cite[\S III.3]{CH}). We have a following example of admissible subsets.
  \begin{proposition}  $A$ is admissible with respect to $\Omega$ in following cases.\\\hk
$(i)$~~ $\Omega$ is convex, $A=\partial\Omega$ and $\displaystyle C(r,\Omega,A)= \frac{c(N,r,t)|\Omega|^{\frac{t^{\ast}-r}{rt^{\ast}}} }{\sigma(\Sigma_{N-1})}$ where $c(N,r,t)$ is a real number depending only from $N, r, t$.\\\hk
$(ii)$~~$U$ is a convex bounded open subset of $\RR^{N}$, $B'(z,2s)$ is a subset of $U$, $\Omega=U\setminus B'(z,2s)$, $A= \partial B(z,2s)$ and $\displaystyle C(r,\Omega,A)=\frac{c(N,r,t)|\Omega|^{\frac{t^{\ast}-r}{rt^{\ast}}} 2diam \Omega}{s\sigma(\Sigma_{N-1})}$.\\\hk
$(iii)$~~ $\Omega=B(z,ks)\setminus B'(z,2s)$, $A= \partial B(z,2s)$, $z\in \RR^{N}$, $k>2$  and $\displaystyle C(r,\Omega,A)=\frac{c(N,r,t)|\Omega|^{\frac{t^{\ast}-r}{rt^{\ast}}} k}{\sigma(\Sigma_{N-1})}$.\\\hk
$(iii)$~~$\Omega=(0,1)^{N}$ and $A =\partial\Omega \setminus (\{0\}\times[0,1]^{N-1})$.\\\hk
$(iv)$~~$\Omega$ is connected and of class $C^{1}$ and $\mu(A)>0$, where $\mu$ is the Riemannian measure on $\partial\Omega$ (see definition in \cite[pp.120-121]{CH}).
 \label{admissible}
  \end{proposition}
  \begin{proof}
 Let $u\in C^{1}(\Omega,A)$. Fix $x\in \Omega$, $\omega\in \Sigma_{N-1}$ and $r_{\omega}\in(0,\infty)$ such that $x'=x+r_{\omega}\omega\in  A$ and $x+sw\in \Omega$ for every $s\in [0,r_{x})$. We have
\begin{equation}|u(x)|=|u(x)-u(x+r_{\omega}\omega)|=|\int_{0}^{r_{\omega}}\hspace{-.03in}\frac{du(x+\xi\omega)}{d\xi}d\xi|
\label{az11}
\end{equation}
 \[=|\int_{0}^{r_{\omega}}\hspace{-.12in}\nabla u(x+\xi\omega).\omega d\xi|
\le  \int_{0}^{r_{\omega}}\hspace{-.15in}|\nabla u(x+\xi\omega)|d\xi=\int_{0}^{r_{\omega}}\hspace{-.15in}\frac{|\nabla u(x+\xi\omega)|}{|x+\xi\omega -x|^{N-1}}\xi^{N-1}d\xi.\]\hk
$(i)$~~ Since $\Omega$ is convex, we have
 \[\Omega\setminus \{x\}= \bigcup_{\omega\in\Sigma_{N-1}}\{x+sw: s\in (0,r_{\omega})\}.\]\hk
 Using the polar coordinate with the pole at $x$, we get 
\begin{equation}|u(x)|=\frac{1}{\sigma(\Sigma_{N-1})}\int_{\Sigma_{N-1}}|u(x)|d\omega \label{a11}
\end{equation}
\[ \leb{\eqref{az11}}\frac{1}{\sigma(\Sigma_{N-1})}\int_{\Sigma_{N-1}}\int_{0}^{r_{\omega}}\frac{|\nabla u(x+\xi\omega)|\xi^{N-1}}{|x+\xi\omega -x|^{N-1}}d\xi d\omega=\frac{1}{\sigma(\Sigma_{N-1})}\int_{\Omega}\hspace{-.03in}\frac{|\nabla u(y)|}{|y -x|^{N-1}}dy. \] \hk
By Hardy–Littlewood–Sobolev's theorem in \cite[p.35]{LP} and \eqref{a11}, there is a positive constant $c(N,t)$ such that
\[\{\int_{\Omega} |u|^{r}dx\}^{\frac{1}{r}}\leb{H\ddot{o}lder}|\Omega|^{\frac{t^{\ast}-r}{rt^{\ast}}} \{\int_{\Omega} |u|^{t^{\ast}}dx\}^{\frac{1}{t^{\ast}}}\hspace{-0.2in}\leb{H-L-S,\eqref{a11}} \frac{c(N,t)|\Omega|^{\frac{t^{\ast}-r}{rt^{\ast}}} }{\sigma(\Sigma_{N-1})}\{\int_{\Omega} |\nabla u|^{t}dx\}^{\frac{1}{t}}.
\]\hk
Thus  we get \eqref{ineq}.\\\hk
 $(ii)$~~Let $u\in C^{1}(\Omega,A)$ and $v\in C^{1}(\RR^{N})$ such that $u=v|_{\Omega}$ and $(\partial B((z,2s))\cap supp(v)=\emptyset$. We can (and shall) suppose $v(y)=0$ for every $y\in B'(z,2s)$ and consider $u$ as its extension to $U$ such that $u(y)= 0$ for every $y\in B'(z,2s)$.\\\hk
 Let $x$ be in $\Omega=U\setminus B'(z,2s)$. Denote by $S_{x}$ the family of $\omega\in \Sigma_{N-1}$ such that there is $r_{x,\omega}$ having following properties:  $x+r_{x,\omega}\omega\in \partial B(z,s)$ and $x+\xi\omega\in U\setminus B(z,s)$ for every $\xi$ in $[0,r_{x,\omega})$. Let $y\in \partial B(z,s)$ such that the line $\{x+t(y,x): t\in\RR\}$is tangent to  $B(z,s)$. We have 
 \[sin\widehat{yxz}=\frac{s}{|x-z|}\ge  \frac{s}{diam(U)}\in (0,1), \]\hk
 which implies $\displaystyle \widehat{yxz}> \arcsin \frac{s}{diam(U)}\ge \frac{s}{diam U}$.
Thus 
\begin{equation}
|S_{x}|\ge \frac{s}{diam U}\frac{1}{2}\sigma(\Sigma_{N-1}),
\label{diam}
\end{equation}
where $|S_{x}|$ is the surface measure  of $S_{x}$.\\\hk
Let $x$ be in $\Omega$ and $\omega\in \Sigma_{N-1}$. There is $\overline{r}_{x,\omega}$ in $(0,\infty)$ such that $x+\overline{r}_{x,\omega}\omega\in \partial U$ and $x+\xi\omega\in U$ for every $\xi$ in $[0,\overline{r}_{x,\omega})$. We have  $\overline{r}_{x,\omega}\le r_{x,\omega}$. As in the proof of $(i)$, we have
\[|u(x)|=|u(x)-u(x+r_{x,\omega}\omega)|=|\int_{0}^{r_{x,\omega}}\hspace{-.03in}\frac{du(x+\xi\omega)}{d\xi}d\xi|\le \int_{0}^{r_{x,\omega}}\hspace{-.15in}\frac{|\nabla u(x+\xi\omega)|}{|x+\xi\omega -x|^{N-1}}\xi^{N-1}d\xi.\]\hk
 Thus
 \[|u(x)|\le \frac{1}{|S_{x}|}\int_{S_{x}}\int_{0}^{r_{x,\omega}}\hspace{-.05in}\frac{|\nabla u(x+\xi\omega)|\xi^{N-1}}{|x+\xi\omega -x|^{N-1}}d\xi\le \frac{1}{|S_{x}|}\int_{\Sigma_{N-1}}\int_{0}^{\overline{r}_{x,\omega}}\hspace{-.05in}\frac{|\nabla u(x+\xi\omega)|\xi^{N-1}}{|x+\xi\omega -x|^{N-1}}d\xi\]
 \[= \hspace{-.03in}\frac{1}{|S_{x}|}\hspace{-.03in}\int_{U}\hspace{-.03in}\frac{|\nabla u(y)|dy}{|y -x|^{N-1}}\hspace{-.1in}\leb{\eqref{diam}}\hspace{-.03in} \frac{2diam U}{s\sigma(\Sigma_{N-1})}\hspace{-.03in}\int_{U}\hspace{-.03in}\frac{|\nabla u(y)|dy}{|y -x|^{N-1}}\hspace{-.05in}\eqb{u|_{B'(z,2r)}=0} \frac{2diam \Omega}{s\sigma(\Sigma_{N-1})}\hspace{-.03in}\int_{\Omega}\hspace{-.03in}\frac{|\nabla u(y)|dy}{|y -x|^{N-1}} .\]\hk
 Now arguing as in the proof of $(i)$, we get
 \[\{\int_{\Omega} |u|^{r}dx\}^{\frac{1}{r}}\le\frac{c(N,t)|\Omega|^{\frac{t^{\ast}-r}{rt^{\ast}}} 2diam \Omega}{s\sigma(\Sigma_{N-1})}\{\int_{\Omega} |\nabla u|^{t}dx\}^{\frac{1}{t}},
\]
which implies $(ii)$.\\\hk
   $(iii)$~~   Put $\overline{\Sigma}_{N-1} =\{(w_{1},\cdots,w_{N})\in \Sigma_{N-1}: w_{1}>0\}$. Let $x\in \Omega$ and $w= (w_{1},\cdots, w_{N})
    \in \overline{\Sigma}_{N-1}$. Then $w_{1}>0$. Put $r_{w}$ as in the proof of $(i)$. Since $x+r_{w}w= (x_{1}+r_{w}w_{1},\cdots, x_{N}+r_{w}w_{N}$, $x_{1}>0$ and $w_{1}>0$, we have $x_{1}+r_{w}w_{1}>0$ and  $x+r_{w}w\in \partial\Omega \setminus (\{0\}\times[0,1]^{N-1})=A$.\\\hk
    Put
 \[ X= \bigcup_{w\in \overline{\Sigma}_{N-1}} \{x+sw: s \in (0,r_{w})\}.\]  
  Then $X$ is open subset of $\Omega$. Thus we can use the polar coordinate with the pole at $x$ and get   
    \begin{equation}|u(x)|=\frac{1}{\sigma(\overline{\Sigma}_{N-1})}\int_{\overline{\Sigma}_{N-1}}|u(x)|d\omega\leb{\eqref{az11}}\frac{1}{\sigma(\overline{\Sigma}_{N-1})}\int_{\overline{\Sigma}_{N-1}}\int_{0}^{r_{\omega}}\frac{|\nabla u(x+\xi\omega)|\xi^{N-1}}{|x+\xi\omega -x|^{N-1}}d\xi d\omega \label{a11b}
\end{equation}
\[ =\frac{1}{\sigma(\overline{\Sigma}_{N-1})}\int_{X}\hspace{-.03in}\frac{|\nabla u(y)|}{|y -x|^{N-1}}dy   \le\frac{1}{\sigma(\overline{\Sigma}_{N-1})}\int_{\Omega}\hspace{-.03in}\frac{|\nabla u(y)|}{|y -x|^{N-1}}dy. \] \hk
 Arguing as in the proof of $(i)$, we get \eqref{ineq}.\\\hk
 $(iv)$~~ We shall prove
  \begin{equation}\{\int_{\Omega} |v|^{r}dx\}^{\frac{1}{r}}\le C(r,\Omega,A)\{\int_{\Omega} |\nabla v|^{t}dx\}^{\frac{1}{t}}\hh\forall ~v \in W^{1,t}(\Omega), Bv=0~on ~A,
\label{ineqb}
\end{equation}   
where $W^{1,t}(\Omega)$ is the usual Sobolev space and $Bv$ is the trace of $v$ on $\partial\Omega$. \\\hk
 Assume by contradiction that there is a sequence $\{v_{n}\}$ in $W^{1,t}(\Omega)$ such that $Bv_{n}$ = 0 on $A$, $||v_{n}||_{L^{r}(\Omega)}=1$ and $||\nabla v_{n}||_{L^{t}(\Omega)}\le \frac{1}{n}$ for every $n\in \NN$.\\\hk
Note that
\begin{equation}
\frac{1}{t}-\frac{1}{N-1}\frac{t-1}{t}=\frac{N-t}{(N-1)t}\lleb{t>1} \frac{1}{t}.
\label{ineqc}
\end{equation}\hk
 By \eqref{rt}, $\{v_{n}\}$ is bounded in $W^{1,t}(\Omega)$. By Theorems 3.18 and 9.1 in \cite{BR}, Theorems 6.1 and 6.2 in \cite{NE} and \eqref{ineqc}, there are $v\in W^{1,t}(\Omega)$, and a subsequence $\{v_{n_{k}}\}$ of $\{v_{n}\}$ such that $\{\nabla v_{n_{k}}\}$ converges weakly to $\nabla v$ in $L^{t}(\Omega)$ and   $\lim_{k\to\infty}||v_{n_{k}}-v||_{L^{r}(\Omega)}=\lim_{k\to\infty}||Bv_{n_{k}}-Bv||_{L^{t}(\partial \Omega)}=0$. By Hahn-Banach's theorem, $||\nabla v||_{L^{t}(\Omega)}\le \limsup_{k\to\infty}||\nabla v_{n_{k}}||_{L^{t}(\Omega)}=0$. Thus  $\nabla v=0$ and $||v||_{L^{r}(\Omega)}=1$.\\\hk
  Since $\lim_{k\to\infty}||Bv_{n_{k}}-Bv||_{L^{t}(\partial \Omega)}=0$, by Theorem 4.9 in \cite[p.94]{BR},  $Bv=0$ on $A$. By Theorem 3.1.3 in \cite[p.64]{MO}, $v=0$ , which contradicts to $||v||_{L^{r}(\Omega)}=1$. Thus we get \eqref{ineqb} and \eqref{ineq}.
  \end{proof} \hk
 We have following properties of $|||.|||_{b}$ and $|||.|||_{B}$.
 \begin{lemma} Let $A$ be admissible with respect to $\Omega$. Assume \eqref{c1} and \eqref{c3} hold. Then $|||.|||_{b}$ and $|||.|||_{B}$ are  norms of $C^{1}(\Omega,A)$. 
 \label{4l2}
  \end{lemma}
    \begin{proof}
 Let $v$ be in $C^{1}(\Omega,A)\setminus \{0\}$. It sufficient to prove $|||v|||_{b}>0$ and $|||v|||_{B}>0$. We have
\begin{equation}0\lleb{v\not=0}\{\int_{\Omega}|v|^{r}dx\}^{\frac{1}{r}}\leb{\eqref{ineq}} C(r,\Omega,A)\{\int_{\Omega} |\nabla v|^{t}dx\}^{\frac{1}{t}} 
 \label{a12}
\end{equation}
\[\leb{H\ddot{o}lder}C(r,\Omega,A)\{\int_{\Omega} b^{-\frac{t}{2-t}}dx\}^{\frac{2-t}{2t}}\{\int_{\Omega} |\nabla v|^{2}bdx\}^{\frac{1}{2}}  \]
\[\leb{\eqref{c1},\eqref{c3}}C(r,\Omega,A)\{\int_{\Omega} b^{-\frac{t}{2-t}}dx\}^{\frac{2-t}{2t}}||| v|||_{B}. 
\]\hk
 Thus we get the lemma.
 \end{proof}
 \begin{definition} Assume \eqref{c1} and \eqref{c3} hold. We denote the completions of  $(C^{1}(\Omega,A),|||.|||_{b})$ and $(C^{1}(\Omega,A),|||.|||_{B})$ by $(W_{b,A}(\Omega),|||.|||_{b,A})$  and  $(W_{B,A}(\Omega),|||.|||_{B,A})$  respectively. 
   \label{d44}
 \end{definition}
By \eqref{i1}, we have
    \begin{equation}
  W_{B,A}(\Omega) \subset W_{b,A}(\Omega),
  \label{i2}
\end{equation}
\begin{equation}
  |||w|||_{b}\le |||w|||_{B}\hh\forall~w\in W_{B,A}(\Omega).
  \label{i2b}
\end{equation}\hk
 Since we shall use $W_{b,A}(\Omega)$ and$W_{B,A}(\Omega)$  to study elliptic equations, we need relations between them with functions having generalized derivatives. We have the following results.
   \begin{lemma} Let $A$ be  admissible with respect to $\Omega$. Assume \eqref{c1} and \eqref{c3} hold. Then $W_{b,A}(\Omega)$ is  contained in  the family of measurable functions on $\Omega$ having first order generalized partial derivatives and
   \begin{equation}|||v|||_{b}= \{\int_{\Omega}|\nabla v|^{2}b dx\}^{\frac{1}{2}}\hh\forall~v\in W_{b,A}(\Omega),   
   \label{i3}
\end{equation}
\begin{equation}|||v|||_{B}\ge \{\int_{\Omega}\sum_{i,j=1}^{N}\frac{\partial v}{\partial x_{i}}\frac{\partial v}{\partial x_{j}}b^{ij}  dx\}^{\frac{1}{2}}\hh\forall~v\in W_{B,A}(\Omega).
\label{i3b}
\end{equation}
\label{5z}
  \end{lemma}
  \begin{proof} We denote $L_{b}^{2}(\Omega)$ the norm space of all measurable function $w$ on $\Omega$ such that
  \[|||w|||_{b}\equiv \{\int_{\Omega}|w|^{2}bdx\}^{\frac{1}{2}}<\infty.\]\hk  
  Since $2>t >\frac{2N^{2}+2N-2}{N^{2}+2N-1}\ge 1$, we have $\frac{t}{2-t}\ge 1$ and
  \begin{equation} L^{\frac{t}{2-t}}(\Omega) \subset L^{1}(\Omega).  
   \label{a13}
\end{equation}
  Let $\{u_{m}\}$ be in $C^{1}(\Omega,A)$ such that $\{\nabla u_{m}\}$ a Cauchy sequence in $(C^{1}(\Omega,A),|||.|||_{B})$.  By \eqref{i2b} and \eqref{a12}, $\{\nabla u_{m}\}$ a Cauchy sequence in $L^{2}_{b}(\Omega)$ and $\{u_{m}\}$ is a Cauchy sequence  in $L^{r}(\Omega)$. Because $L^{r}(\Omega)$ and $L^{2}_{b}(\Omega)$ are Banach spaces,  $\{u_{m}\}$  converges to $u$ in $L^{r}(\Omega)$ and  $\displaystyle\{\frac{\partial u_{m}}{\partial x_{j}}\}$  converges to $v_{j}$ in $L^{2}_{b}(\Omega)$ for every $j$ in $\{1,\cdots,N\}$.\\\hk
  We have
  \[\lim_{m\to \infty}\int_{\Omega}|u_{m}-u|dx \eqb{r\ge 1}0,\]
  \[\lim_{m\to \infty}\int_{\Omega}\hspace{-.05in}|\frac{\partial u_{m}}{\partial x_{i}}-v_{j}|dx\hspace{-.1in}\leb{H\ddot{o}lder} \lim_{m\to \infty}\{\int_{\Omega}\hspace{-.05in}b^{-1}dx\}^{\frac{1}{2}} \{\int_{\Omega}|\frac{\partial u_{m}}{\partial x_{i}}-v_{j}|^{2}b d x\}^{\frac{1}{2}}\hspace{-.15in} \eqb{\eqref{c3},\eqref{a13}}\hspace{-.1in}0\]
  for every $ j=1,\cdots,N$. \\\hk  
  Thus
  \[\int_{\Omega}v_{i}\varphi dx= \lim_{n\to\infty} \int_{\Omega}\frac{\partial u_{n}}{\partial x_{i}}\varphi dx=-\lim_{n\to\infty} \int_{\Omega}u_{n}\frac{\partial \varphi}{\partial x_{i}} dx\]
  \[= -\int_{\Omega}u\frac{\partial \varphi}{\partial x_{i}} dx\hh\forall~\varphi\in C^{1}_{c}(\Omega).  \]\hk  
  Therefore $v_{i}$ is the generalized partial derivative of $u$ along direction $x_{i}$.  Hence $u$ has  first order generalized partial derivatives. On other hand
   \[|||u|||_{b}=\lim_{n\to\infty}\{\int_{\Omega} |\nabla u_{n}|^{2}bdx\}^{\frac{1}{2}}=\{\int_{\Omega}|( v_{1},\cdots,v_{N})|^{2}b dx\}^{\frac{1}{2}}=\{\int_{\Omega}|\nabla u|^{2}b dx\}^{\frac{1}{2}}. \]\hk
   Thus we get \eqref{i3}. By Theorem 4.9 in \cite[p.94]{BR}, we can suppose  $\displaystyle\{\frac{\partial u_{m}}{\partial x_{j}}\}$  converges  to $\displaystyle \frac{\partial u}{\partial x_{j}}$ a.e. on $\Omega$ for every $j$ in $\{1,\cdots,N\}$. Then
   \[ \{\int_{\Omega}\sum_{i,j=1}^{N}\frac{\partial u}{\partial x_{i}}\frac{\partial u}{\partial x_{j}}b^{ij}  dx\}^{\frac{1}{2}}\hspace{-.15in}\leb{Fatou}\hspace{-.08in}\liminf_{n\to\infty}\{\int_{\Omega}\sum_{i,j=1}^{N}\frac{\partial u_{n}}{\partial x_{i}}\frac{\partial u_{n}}{\partial x_{j}}b^{ij}  dx\}^{\frac{1}{2}}
=\lim_{n\to\infty}|||u_{n}|||_{B}=|||u|||_{B}. \]\hk
 Thus we get \eqref{i3b}.
 \end{proof} 
 \begin{lemma} Let $A$ be  admissible with respect to $\Omega$. Assume \eqref{c1} and \eqref{c3} hold. Then 
 \begin{equation}||w||_{L^{r}(\Omega)}\le C(r,\Omega,A)|||w|||_{b}\le C(r,\Omega,A)||| w|||_{B}
\hh\forall~ w\in W_{b,,A}(\Omega).
\label{cb}
\end{equation}
  \label{4l2b}
  \end{lemma}
  \begin{proof}
 Let $u$ be in $W_{b ,q,A}(\Omega)$ and $\{u_{n}\}_{n}$ be a sequence in $C^{1}(\Omega,A)$ such that $\{u_{n}\}_{n}$ converges to $u$ in $W_{b ,q,A}(\Omega)$. By \eqref{ineq}, $\{u_{n}\}_{n}$ converges to $u$ in $L^{r}(\Omega)$. Thus
 \[||u||_{L^{r}(\Omega)}\hspace{-.04in}= \hspace{-.04in}\lim_{n\to \infty}||u_{n}||_{L^{r}(\Omega)}\hspace{-.08in}\leb{\eqref{ineq}} \hspace{-.08in} C(r,\Omega,A)\lim_{n\to \infty}||| u_{n}|||_{b}\]
 \[=C(r,\Omega,A)||| u|||_{b}\leb{\eqref{i2b}}C(r,\Omega,A)||| u|||_{B}\]
and we get the lemma.
 \end{proof} 
   \section{Superpositions  in $W_{B,A}(\Omega)$.}\label{app}\hk
We need  following results to study the regularity of solutions of elliptic equations.
\begin{lemma} Let    $A$ be admissible with respect to $\Omega$,  $\alpha \in [-\infty,\infty)$, $\phi\in C^{1}((\alpha,\infty))$, $v$ be in $W_{B,A}(\Omega)$, $\{v_{n}\}$ be a sequence in $C^{1}(\Omega,A)$ such that  $\cup_{n=1}^{\infty} v_{n}(\Omega)\cup v(\Omega) \subset (\alpha,\infty)$. Assume \eqref{c1} and \eqref{c3} hold,\\\hk
  $(i)$~$\phi\in C^{1}((\alpha,\infty))$, $\phi'$ is  uniformly continuous on $(\alpha,\infty)$ and 
  $||\phi'||_{L^{\infty}(\RR)}\le M$  with  a positive real number M and\\\hk
   $(ii)$~~$\lim_{n\to\infty}|||v_{n}-v|||_{B}=0$.\\\hk
  Then $\{v_{n}\}$ has a subsequence $\{v_{n_{k}}\}$ such that
  \begin{equation}
  \lim_{k\to\infty}|||\phi\circ v_{n_{k}}-\phi\circ v|||_{B}=0,
   \label{2l51}
   \end{equation}\hk
  Furthermore, $\phi\circ v \in W_{B,A}(\Omega)\subset W_{b,A}(\Omega)$, if $\phi(0)=0$.
     \label{2w1ba}
  \end{lemma}
\begin{proof} We prove the lemma by two steps.\\\hk
{\bf Step 1}~. Since $\{v_{n}\}$ converges in $W_{B,A}(\Omega)$, there is a positive real number $M_{1}$ such that
\begin{equation}
  |||v_{n}|||_{B}\le M_{1}\hh \forall~n\in\NN.
  \label{47z0}
\end{equation}\hk
  We have
\[\lim_{n\to \infty}\int_{\Omega}|v_n{}-v|dx \eqb{1<r,\eqref{cb}} 0.\]\hk 
   By Theorem 4.9 in \cite[p.94]{BR}, there is a subsequence $\{v_{n_{k}}\}$ of $\{v_{n}\}$ such that $\{v_{n_{k}}\}$ converges  to $v$ a.e. on $\Omega$. \\\hk
   Let $\epsilon \in(0,1)$ and $\delta \in (0,\epsilon)$. Since $\Omega$ is bounded, by Egorov's Theorem in \cite[p.115]{BR}, $\Omega$ has a subset $A_{\delta}$ such that $|\Omega\setminus A_{\delta}|\le \delta$ and  $\{v_{n_{k}}\}$ converges uniformly to $v$ on $A_{\delta}$. Because  $\phi'$ is uniformly continuous  on $\RR$, there is an integer $k_{\epsilon}$ such that
 \[
  |\phi'(v(x))-\phi'(v_{n_{k}}(x))|\le \epsilon \hh\forall~x \in A_{\delta}, k\ge k_{\epsilon},
  \] 
\[
  ||| v -  v_{k}|||_{B}\le \epsilon\hh\forall~ k\ge k_{\epsilon}.
  \]\hk 
It implies
 \begin{equation}
  |\phi'(v_{n_{k'}}(x))-\phi'(v_{n_{k}}(x))|\le 2\epsilon \hh\forall~x \in A_{\delta}, k', k\ge k_{\epsilon},
  \label{47z2}
\end{equation} 
\begin{equation}
  ||| v_{n_{k'}} -  v_{n_{k}}|||_{B}\le 2\epsilon\hh\forall~ k',k\ge k_{\epsilon}.
  \label{247z3}
\end{equation}\hk 
 By \eqref{i3b},  $\displaystyle
 \sum_{ij=1}^{N}\frac{\partial v_{n_{k_{\epsilon}}}}{\partial x_{i}}\frac{\partial v_{n_{k_{\epsilon}}}}{\partial x_{j}}b^{ij}$ is non-negative and integrable on $\Omega$. Therefore we can choose  $\delta$ such that 
 \begin{equation}\int_{\Omega\setminus A_{\delta}}\sum_{ij=1}^{N}\frac{\partial v_{n_{k_{\epsilon}}}}{\partial x_{i}}\frac{\partial v_{n_{k_{\epsilon}}}}{\partial x_{j}}b^{ij}dx \le \epsilon^{2}.
  \label{247z4}
\end{equation}\hk 
  Thus
\[ 
 \{\int_{\Omega\setminus A_{\delta}}\sum_{ij=1}^{N}\frac{\partial v_{n_{k}}}{\partial x_{i}}\frac{\partial v_{n_{k}}}{\partial x_{j}}b^{ij}dx\}^{\frac{1}{2}} \leb{Minkowski}\{\int_{\Omega\setminus A_{\delta}}\sum_{ij=1}^{N}\frac{\partial v_{n_{k_{\epsilon}}}}{\partial x_{i}}\frac{\partial v_{n_{k_{\epsilon}}}}{\partial x_{j}}b^{ij}dx\}^{\frac{1}{2}}\]
 \[+ \{\int_{\Omega\setminus A_{\delta}}\sum_{ij=1}^{N}\frac{\partial v_{n_{k}}-\partial v_{n_{k_{\epsilon}}}}{\partial x_{i}}\frac{\partial v_{n_{k}}-\partial v_{n_{k_{\epsilon}}}}{\partial x_{j}}b^{ij}dx\}^{\frac{1}{2}} \leb{\eqref{247z3},\eqref{247z4}} 3\epsilon\hh\forall~ k\ge k_{\epsilon}
 \]
  or
 \begin{equation}
 \int_{\Omega\setminus A_{\delta}}\sum_{ij=1}^{N}\frac{\partial v_{n_{k}}}{\partial x_{i}}\frac{\partial v_{n_{k}}}{\partial x_{j}}b^{ij}dx\le 9\epsilon^{2} \hh\forall~ k\ge k_{\epsilon}. 
  \label{247z3b}
\end{equation}\hk  
Fix $x$ in $\Omega$, put
 \[<y,z> = \sum_{i,j=1}^{N}y_{i}z_{j}b^{ij}(x)\hk\forall~x=(x_{1},\cdots,x_{N}, z=(z_{1},\cdots,z_{N})\in \RR^{N}.\]\hk
 Then $<.,.>$ is scalar product on $\RR^{N}$. Therefore
\[<y+z,y+z> \le (<y,y>^{\frac{1}{2}} +<z,z>^{\frac{1}{2}} )^{2}\le 2(<y,y>+<z,z>)
\]
\hh\hh$\forall~x=(x_{1},\cdots,x_{N}), z=(z_{1},\cdots,z_{N})\in \RR^{N}.$\\\hk
     Put $\displaystyle h_{i}=(\phi'(v_{n_{k'}})-\phi'(v_{n_{k}}))\frac{\partial v_{n_{k'}}}{\partial x_{i}} - \phi'(v_{n_{k}})\frac{\partial (v_{n_{k'}}-v_{n_{k}})}{\partial x_{i}}$ for every $i$ in $\{1,\cdots,N\}$ and $h=(h_{1},\cdots,h_{N})$. We have    
 \begin{equation}|||\phi\circ v_{n_{k'}} - \phi\circ v_{n_{k}}|||_{B} ^{2}
= \int_{\Omega}\frac{\partial (\phi\circ v_{n_{k'}} - \phi\circ v_{n_{k}})}{\partial x_{i}}\frac{\partial (\phi\circ v_{n_{k'}} - \phi\circ v_{n_{k}})}{\partial x_{j}}b^{ij} dx
\label{b1}
\end{equation}   
\[=\int_{\Omega}[\phi'(v_{n_{k'}})\frac{\partial  v_{n_{k'}} }{\partial x_{i}} - \phi'(v_{n})\frac{\partial v_{n_{k}}}{\partial x_{i}}][\phi'(v_{n_{k'}})\frac{\partial  v_{n_{k'}} }{\partial x_{j}} - \phi'(v_{n})\frac{\partial v_{n_{k}}}{\partial x_{j}}]b^{ij} dx \hspace{-.02in}
=\hspace{-.05in}\int_{\Omega}\hspace{-.08in}<h,h>dx
\]   
\[\le \hspace{-.05in}2\{\hspace{-.05in}\int_{\Omega}\hspace{-.06in}[(\phi'(v_{n_{k'}})-\phi'(\hspace{-.02in}v_{n_{k}}\hspace{-.02in})\hspace{-.02in})^{2}\frac{\partial  v_{n_{k'}} }{\partial x_{i}}\frac{\partial  v_{n_{k'}} }{\partial x_{j}}b^{ij} dx+ \hspace{-.05in}\int_{\Omega}\hspace{-.08in}\phi'(\hspace{-.02in}v_{n}\hspace{-.02in})^{2}\frac{\partial (\hspace{-.02in}v_{n_{k'}}-v_{n_{k}}\hspace{-.02in}) }{\partial x_{i}} \frac{\partial (\hspace{-.02in}v_{n_{k'}}-v_{n_{k}}\hspace{-.02in}) }{\partial x_{j}}b^{ij} dx\}
\]   
\[\leb{\eqref{247z3}} 2\int_{A_{\epsilon}}[(\phi'(v_{n_{k'}})-\phi'(v_{n_{k}}))^{2}\frac{\partial  v_{n_{k'}} }{\partial x_{i}}\frac{\partial  v_{n_{k'}} }{\partial x_{j}}b^{ij} dx\]
\[+2\int_{\Omega\setminus A_{\epsilon
}}[(\phi'(v_{n_{k'}})-\phi'(v_{n_{k}}))^{2}\frac{\partial  v_{n_{k'}} }{\partial x_{i}}\frac{\partial  v_{n_{k'}} }{\partial x_{j}}b^{ij} dx+ 4M^{2}\epsilon^{2}
\]    
\[\leb{\eqref{47z2}} 2\epsilon^{2}\int_{\Omega}\frac{\partial  v_{n_{k'}} }{\partial x_{i}}\frac{\partial  v_{n_{k'}} }{\partial x_{j}}b^{ij} dx+8M^{2}\int_{\Omega\setminus A_{\epsilon}}\frac{\partial  v_{n_{k'}} }{\partial x_{i}}\frac{\partial  v_{n_{k'}} }{\partial x_{j}}b^{ij} dx+ 4M^{2}\epsilon^{2}
\]   
\[\leb{\eqref{47z0},\eqref{247z4}} 2\epsilon^{2}\{M_{1}^{2}+ 36M^{2}+ 2M^{2}\}\hh\forall~k',k\ge k_{\epsilon}.
\] \hk
 Therefore $\{\phi\circ v_{n_{k}}\}_{k}$ is a Cauchy sequence in $W_{B,A}(\Omega)$. Thus it converges to $w$ in  $W_{B,A}(\Omega)$. \\\hk
 {\bf Step 2}~.    By \eqref{i1}, $\{\phi\circ v_{n_{k}}\}_{k}$ converges to $w$ in  $W_{b,A}(\Omega)$. We shall prove $w= \phi\circ v$. It is sufficient to  prove that $\{v_{n_{k}}\}_{k}$ has a subsequence $\{v_{n_{k_{l}}}\}_{l}$ such that $\{\phi\circ v_{n_{k_{l}}}\}_{l}$ converges to $\phi\circ v$ in $W_{b,A}(\Omega)$. \\\hk
   Let $\epsilon \in(0,\infty)$. By \eqref{i2b} and Lemma \ref{5z},  $|\nabla v|^{2}b $  is in $L^{1}(\Omega)$. Then  there is a positive real number $\delta$ such that
  \begin{equation}
  \int_{E}|\nabla v|^{2}b dx<\epsilon\hh\forall~E\subset \Omega~ with ~|E|\le \delta.
  \label{47z1}
\end{equation} \hk  
By $(ii)$ and Lemma \ref{4l2b}, we have
\[\lim_{k\to \infty}\int_{\Omega}|v_{n_{k}}-v|dx \eqb{1<r,\eqref{cb}} 0.\]\hk 
By Theorem 4.9 in \cite[p.94]{BR}, there is a subsequence $\{v_{n_{k_{l}}}\}$ of $\{v_{n_{k}}\}$ such that $\{v_{n_{k_{l}}}\}$ converges  to $v$ a.e. on $\Omega$. \\\hk
     Since $\Omega$ is bounded, by Egorov's Theorem in \cite[p.115]{BR}, $\Omega$ has a subset $E_{\delta}$ such that $|\Omega\setminus E_{\delta}|\le \delta$ and  $\{v_{n_{k}}\}$ converges uniformly to $v$ on $E_{\delta}$. Because  $\phi'$ is uniformly continuous  on $\RR$, there is an integer $l_{\epsilon}$ such that
 \begin{equation}
  |\phi'(v(x))-\phi'(v_{n_{k_{l}}}(x))|\le \epsilon \hh\forall~x \in E_{\delta}, l\ge l_{\epsilon},
  \label{47z2c}
\end{equation} 
\begin{equation}
  ||\nabla v -  \nabla v_{n_{k_{l}}}||_{L^{2}_{b }(\Omega)}\leb{\eqref{i2b},(ii)} \epsilon\hh\forall~ l\ge l_{\epsilon}.
  \label{47z3c}
\end{equation}\hk 
     Thus     
   \[||\nabla(\phi\circ v - \phi\circ v_{n_{k_{l}}})||_{L^{2}_{b }(\Omega)} \le ||\phi'(v)\nabla v -  \phi'(v_{n_{k_{l}}})\nabla v_{n_{k_{l}}}||_{L^{2}_{b }(\Omega)} 
 \]  
 \[ \leb{Minkowski} ||(\phi'(v)-\phi'(v_{n_{k_{l}}}))\nabla v||_{L^{2}_{b }(\Omega)} + ||\phi'(v_{n_{k_{l}}})(\nabla v -  \nabla v_{n_{k_{l}}})||_{L^{2}_{b }(\Omega)} 
 \]  
\[\leb{Minkowski,\eqref{47z3c}} ||(\phi'(v)-\phi'(v_{n_{k_{l}}}))\nabla v||_{L^{2}_{b }(E_{\delta})}+||(\phi'(v)-\phi'(v_{n_{k_{l}}}))\nabla v||_{L^{2}_{b }(\Omega\setminus E_{\delta})}+\epsilon M
\]   
 \[\leb{\eqref{47z2c}} \epsilon||\nabla v||_{L^{2}_{b }(\Omega)}+2M||\nabla v||_{L^{2}_{b }(\Omega\setminus E_{\delta})}+\epsilon M\leb{\eqref{47z1}} \epsilon(||\nabla v||_{L^{2}_{b }(\Omega)}+3M)\hh\forall~l
 \ge l_{\epsilon},\]
which implies $\{v_{n_{k_{l}}}\}_{l}$ converges to $\phi\circ v$ in $W_{b,A}(\Omega)$. \\\hk
 Since $\{v_{n_{k_{l}}}\}_{l}$ converges to $w$ in $W_{b,A}(\Omega)$, $w= \phi\circ v$. By Step 1, we get \eqref{2l51}.\\\hk
 {\bf Step 3}~.  When  $\phi(0)=0
 $,  we see that $\{\phi\circ v_{n_{k}}\}$ is contained in $C^{1}(\Omega,A)$. Therefore $\phi\circ v$ is in $W_{B,A}(\Omega)$.
\end{proof}
\begin{lemma} Let   $A$ be  admissible with respect to $\Omega$,  $v$ be in $W_{B,A}(\Omega)$,  $v^{+}\equiv \max\{0,v\}$,  $v^{-}\equiv \max\{0,-v\}$.  Assume \eqref{c1}, \eqref{c3} and \eqref{c6} hold. Then \\\hk
$(i)$~~$v^{+}$, $v^{-}$ and $|v|$ are in $W_{B,A}(\Omega)$.\\\hk
$(ii)$~~If $v$ is non-negative on $\Omega$, there is a sequence of non-negative functions $\{w_{n}\}$ in $C^{1}(\Omega,A)$ such that $\lim_{n\to\infty}|||w_{n}-v|||_{B}=0$.\\\hk
$(iii)$~~$\eta v$ is in $W_{B,A}(\Omega)$, if $\eta\in C^{1}(\RR^{N})$.
 \label{w1bb}
\end{lemma}
\begin{proof} $(i)$~Let $w_{1}$ be $v^{+}$, $w_{2}$ be $v^{-}$, $w_{3}$ be $|v|$, and $\{v_{n}\}$ be a sequence in $C^{1}(\Omega,A)$ such that $\{v_{n}\}$ converges to $v$ in $W_{B,A}(\Omega)$.\\\hk
Let $\beta\in (1,2)$. Put 
 \begin{equation}
\Phi_{\beta} (\xi) = \left\{ {\begin{array}{*{20}l}
 0 & {} & {\forall~ \xi \in (-\infty,0],}  \vspace{.1in}\\
  \xi^{\beta}& {} & {\forall~ \xi \in (0,1),}  \vspace{.1in}\\
     \beta(\xi-1)+1 & {} & {\forall~ \xi \in [1 ,\infty).}  \\
\end{array}} \right.
\label{z7}
\end{equation}\hk
We have
\begin{equation}
\Phi_{\beta}' (\xi) = \left\{ {\begin{array}{*{20}l}
 0 & {} & {\forall~ \xi \in (-\infty,0],}  \vspace{.1in}\\
 \beta \xi^{\beta-1} & {} & {\forall~ \xi \in (0,1),}  \vspace{.1in}\\
    \beta  & {} & {\forall~ \xi \in [1 ,\infty).}  \\
\end{array}} \right.
 \label{z780bz} 
  \end{equation}\hk
Thus $\Phi_{\beta}$ and $\Phi_{\beta}'$ are monotone increasing on $\RR$, $\Phi_{\beta}$ is in $C^{1}(\RR)$, $\Phi_{\beta}'$ is in $L^{\infty}(\RR)$ and uniformly continuous on $\RR$. By Lemma \ref{2w1ba},  $\Phi_{\beta}\circ v$ is in $W_{B,A}(\Omega)$ for every $\beta$ in $(1,2)$.\\\hk
By \eqref{z780bz}, we get
\begin{equation}
\Phi_{\beta} (\xi) \le \Phi_{\beta} (\zeta)\hh if~~ 0\le \xi\le \zeta<\infty,
\label{hp0d}
\end{equation}
\begin{equation}
0\le \Phi_{\beta}' (\xi) \le \beta~and~\lim_{\beta\to 1}\Phi_{\beta}' (\xi)=1\hh \forall~ \xi\in (0,\infty).
\label{hp0db}
\end{equation}\hk
  Since $v$ has first order gegralized partial derivatives by Lemma \ref{5z}, by Lemma 7.6 in \cite[p.152]{GT}, we have
\begin{equation}
\frac{\partial v^{+}}{\partial x_{j}}(x) = \left\{ {\begin{array}{*{20}l}
 0 & {} & {if~ v(x) \le 0,}  \vspace{.1in}\\
      \displaystyle \frac{\partial v}{\partial x_{j}}(x) & {} & {if~ v(x) > 0,}  \\
\end{array}} \right.
\label{5z1}
\end{equation}
\begin{equation}
\frac{\partial v^{-}}{\partial x_{j}}(x) = \left\{ {\begin{array}{*{20}l}
 0 & {} & {if~ v(x) \ge 0,}  \vspace{.1in}\\
      \displaystyle -\frac{\partial v}{\partial x_{j}}(x) & {} & {if~ v(x) < 0,}  \\
\end{array}} \right.
\label{5z2}
\end{equation}
\begin{equation}
\frac{\partial |v|}{\partial x_{j}}(x) = \left\{ {\begin{array}{*{20}l}
 \displaystyle \frac{\partial v}{\partial x_{j}}(x) & {} & {if~ v(x) > 0,}  \\
 0 & {} & {if~ v(x)= 0,}  \vspace{.1in}\\
      \displaystyle -\frac{\partial v}{\partial x_{j}}(x) & {} & {if~ v(x) < 0.}  \\
\end{array}} \right.
\label{5z3}
\end{equation}\hk
Put $\Omega^{+}=\{x\in\Omega : v(x)>0\}$.
By \eqref{z7}, $\Phi_{1+\frac{1}{n}}\circ v=\Phi_{1+\frac{1}{n}}\circ v^{+}$. We have
\begin{equation}\int_{\Omega^{+}}\frac{\partial v^{+}}{\partial x_{i}}\frac{\partial v^{+}}{\partial x_{j}}b^{ij}dx\leb{\eqref{5z1}}\int_{\Omega}\frac{\partial v}{\partial x_{i}}\frac{\partial v}{\partial x_{j}}b^{ij}dx\lleb{\eqref{i3b}} \infty,\label{bzz}
\end{equation}
\[\lim_{k\to \infty}\int_{\Omega}\sum_{i,j=1}^{N}(\frac{\partial (\Phi_{1+\frac{1}{k}}\circ v)}{\partial x_{i}}-\frac{\partial v^{+}}{\partial x_{i}})(\frac{\partial (\Phi_{1+\frac{1}{k}}\circ v)}{\partial x_{j}}-\frac{\partial v^{+}}{\partial x_{j}})b^{ij} dx \]
\[
\eqb{\eqref{5z1}} \lim_{k\to \infty }\int_{\Omega^{+}}\sum_{i,j=1}^{N} |[\Phi_{1+\frac{1}{k}}'(v(x))-1]^{2}\frac{\partial v^{+}}{\partial x_{i}}\frac{\partial v^{+}}{\partial x_{j}}(x)b^{ij} dx\]
\[\eqb{\eqref{hp0db},\eqref{bzz},Lebesgue's~Dominated~ Convergence}\hspace{-.42in} 0\hh\forall~k\in\NN. \]\hk
It implies
\[\lim_{m,n\to\infty}|||\Phi_{1+\frac{1}{n}}\circ v - \Phi_{1+\frac{1}{m}}\circ v|||_{B}\]
\[\leb{Minkowski} \lim_{n\to\infty}\{\int_{\Omega}\sum_{i,j=1}^{N}(\frac{\partial (\Phi_{1+\frac{1}{n}}\circ v)}{\partial x_{i}}-\frac{\partial v^{+}}{\partial x_{i}})(\frac{\partial (\Phi_{1+\frac{1}{n}}\circ v)}{\partial x_{j}}-\frac{\partial v^{+}}{\partial x_{j}})b^{ij} dx\}^{\frac{1}{2}}\]
\[+ \lim_{m\to\infty}\{\int_{\Omega}\sum_{i,j=1}^{N}(\frac{\partial (\Phi_{1+\frac{1}{m}}\circ v)}{\partial x_{i}}-\frac{\partial v^{+}}{\partial x_{i}})(\frac{\partial (\Phi_{1+\frac{1}{m}}\circ v)}{\partial x_{j}}-\frac{\partial v^{+}}{\partial x_{j}})b^{ij} dx\}^{\frac{1}{2}}=0.
\]\hk
 Thus, $\{\Phi_{1+\frac{1}{n}}\circ v\}$ is a Cauchy sequence in $W_{B,A}(\Omega)$ and converges to $w$ in $W_{B,A}(\Omega)$. We shall prove $w=v^{+}$.  By \eqref{i2b}, $\{\Phi_{1+\frac{1}{n}}\circ v\}$  converges to $w$ in $W_{b,A}(\Omega)$. \\\hk
  We have
\[\{\int_{\Omega^{+}}|\nabla v|^{2}bdx\}^{\frac{1}{2}}\leb{\eqref{5z1}}\{\int_{\Omega}|\nabla v^{+}|^{2}bdx\}^{\frac{1}{2}}\lleb{\eqref{c1},\eqref{i3},\eqref{bzz}}\infty,\]
\[\lim_{n\to\infty}||\Phi_{1+\frac{1}{n}}\circ v-v^{+}||_{b}=\lim_{n\to \infty}\{\int_{\Omega}\sum_{i=1}^{N}|\frac{\partial (\Phi_{1+\frac{1}{n}}\circ v)}{\partial x_{i}}-\frac{\partial v^{+}}{\partial x_{i}}|^{2}b dx \}^{\frac{1}{2}}\]
\[
\eqb{\eqref{5z1}} \lim_{n\to \infty}\{ \int_{\Omega^{+}}\sum_{i,j=1}^{N} |[\Phi_{1+\frac{1}{n}}'(v(x))-1]^{2}|\frac{\partial v}{\partial x_{i}}|^{2}b dx\}^{\frac{1}{2}}\]
\[\eqb{\eqref{hp0db},~Lebesgue's~Dominated~ Convergence}\hspace{-.42in} 0. \]\hk
Thus $\{\Phi_{1+\frac{1}{n}}\circ v\}$  converges to $v^{+}$ in $W_{b,A}(\Omega)$ and $v^{+}=w$. Therefore  $\{\Phi_{1+\frac{1}{n}}\circ v\}$  converges to $v^{+}$ in $W_{B,A}(\Omega)$.\\\hk
 We get the remainder of $(i)$ by replacing $v$ by $-v$ in the above arguments and noting that $|v|=v^{+}-v^{-}$.\\\hk
 $(ii)$~~Let $\{v_{n}\}$ be a sequence in $C^{1}(A,\Omega)$ such that it converges to  $v$ in $W_{B,A}(\Omega)$. By Lemma \ref{2w1ba}, $\{\Phi_{1+\frac{1}{k}}\circ v_{m}\}_{m}$ converges to $\Phi_{1+\frac{1}{k}}\circ v$ in  $W_{b,A}(\Omega)$ for every $k$ in $\NN$. By $(i)$, $\{\Phi_{1+\frac{1}{k}}\circ v\}_{k}$ converges to $v^{+}$ in  $W_{B,A}(\Omega)$. It implies, for each  $n$ in $\NN$, there are integers  $k_{n}$ and $m_{n}$ such that
 \[|||\Phi_{1+\frac{1}{k_{n}}}\circ v- v^{+}|||_{B}<\frac{1}{2n}, \]
 \[|||\Phi_{1+\frac{1}{k_{n}}}\circ v_{m_{n}}-\Phi_{1+\frac{1}{k_{n}}}\circ v|||_{B}<\frac{1}{2n}. \]\hk
  It implies
 \[|||\Phi_{1+\frac{1}{k_{n}}}\circ v_{m_{n}}- v^{+}|||_{B}\le |||\Phi_{1+\frac{1}{k_{n}}}\circ v_{m_{n}}-\Phi_{1+\frac{1}{k_{n}}}\circ v|||_{B}+|||\Phi_{1+\frac{1}{k_{n}}}\circ v- v^{+}|||_{B}<\frac{1}{n}. \]\hk
 Thus $\{\Phi_{1+\frac{1}{k_{n}}}\circ v_{m_{n}}\}$ is in $C^{1}(A,\Omega)$ and converges to $v^{+}=v$ in $W_{B,A}(\Omega)$. Since $\Phi_{1+\frac{1}{k_{n}}}\circ v_{m_{n}}\ge 0$, we get $(ii)$ with $w_{n}=\Phi_{1+\frac{1}{k_{n}}}\circ v_{m_{n}}$ for every $n$ in $\NN$.\\\hk
   $(iii)$~~ There is a positive real number $C(r,\Omega,A)$ such that
 \begin{equation}
  \{\int_{\Omega}|u|^{2}\overline{b}dx\}^{\frac{1}{2}}\leb{H\ddot{o}lder,\eqref{c6}}  \{\int_{\Omega}\overline{b}^{\frac{r}{r-2}}dx\}^{\frac{r-2}{2r}}\{\int_{\Omega}|u|^{r}dx\}^{\frac{1}{r}}
  \label{b0}
  \end{equation}
  \[\leb{\eqref{cb}}\{\int_{\Omega}\overline{b}^{\frac{r}{r-2}}dx\}^{\frac{r-2}{2r}}C(r,\Omega,A)|||u|||_{B}\hh\forall~u\in W_{B,A}(\Omega).
  \]
  Since $\eta\in C^{1}(\RR^{N})$, there is a positive real number $M$ such that $|\eta|+|\nabla \eta|\le M $ on $\overline{\Omega}$. Let  $\{v_{n}\}$ be a sequence in $C^{1}(\Omega,A)$ such that $\{v_{n}\}$ converges to $v$ in $W_{B,A}(\Omega)$. Then $\eta v_{n}$ is in  $C^{1}(\Omega,A)$ for every $n$ in $\NN$. Arguing as in \eqref{b1}, we have
  \[|||\eta v_{n}- \eta v_{m}|||_{B}^{2}=
 \int_{\Omega}\sum_{i,j=1}^{N}\frac{\partial [\eta(v_{n}-v_{m})]}{\partial x_{i}}\frac{\partial [\eta(v_{n}-v_{m})]}{\partial x_{j}}b^{ij}dx \]
\[\le 2\{\int_{\Omega}[(v_{n}-v_{m})^{2}\frac{\partial  \eta }{\partial x_{i}}\frac{\partial  \eta }{\partial x_{j}}b^{ij} dx+ \int_{\Omega}\eta^{2}\frac{\partial (v_{n}-v_{m}) }{\partial x_{i}} \frac{\partial (v_{n}-v_{m}) }{\partial x_{j}}b^{ij} dx\}
\]   
\[\leb{\eqref{c1}}2M^{2}\int_{\Omega}[(v_{n}-v_{m})^{2}\overline{b} dx+ M^{2}\int_{\Omega}\frac{\partial (v_{n}-v_{m}) }{\partial x_{i}} \frac{\partial (v_{n}-v_{m}) }{\partial x_{j}}b^{ij} dx
\]
\[\leb{\eqref{b0}}M^{2}\{C(r,\Omega,A)^{2} ||\overline{b}||_{L^{\frac{r}{r-2}}(\Omega)}  +1\}|||v_{n}-v_{m}|||_{B} \hh\forall~m,n\in\NN.\]\hk
 Thus $\{\eta v_{n}\}$ is a Cauchy sequence in $W_{B,a}(\Omega)$. Then it converges to $w$ in $W_{B,A}(\Omega)$. By \eqref{i1}, $\{\eta v_{n}\}$ converges to $w$ in $W_{b,A}(\Omega)$. We shall prove $w=\eta v$. We have
   \[||\nabla(\eta v_{n} - \eta v)||_{L^{2}_{b }(\Omega)} \le ||(v_{n}-v)\nabla \eta-  \eta(\nabla v_{n}-\nabla v)||_{L^{2}_{b }(\Omega)} 
 \]  
 \[ \leb{Minkowski} ||(v-v_{n_{k}})\nabla \eta||_{L^{2}_{b }(\Omega)} + ||\eta(\nabla v -  \nabla v_{n})||_{L^{2}_{b }(\Omega)} 
 \]  
\[\le M(||v-v_{n}||_{L^{2}_{b }(\Omega)} + ||\nabla v -  \nabla v_{n}||_{L^{2}_{b }(\Omega)})
\] 
\[\leb{H\ddot{o}lder} M(||b||_{L^{\frac{r}{r-2}}(\Omega)}^{\frac{1}{2}}||v-v_{n}||_{L^{r}(\Omega)} + ||\nabla v -  \nabla v_{n}||_{L^{2}_{b }(\Omega)})
\] 
\[\leb{\eqref{cb}} M\{||b||_{L^{\frac{r}{r-2}}(\Omega)}^{\frac{1}{2}}C(r,\Omega,A)+ 1\}||| v -   v_{n}|||_{b} \mathop  \to \limits_{n \to \infty } 0.
\]\hk
 Thus   $\{\eta v_{n}\}$ converges to $\eta v$ in $W_{b,A}(\Omega)$ and $w=\eta v$. Therefore we get $(iii)$.   
 \end{proof}
 \section{Auxiliary functions}\label{aux}\hk
In \cite{SE} Serrin obtained the local boundedness of solutions of elliptic equations. In this section,  we shall use Serrin's method  to get the global boundedness of solutions of a class elliptic equations, but these technique are only applicable to local regularity. We need to modify the auxiliary functions in \cite{SE} as follows. 
\begin{definition} Let  $s \in (1,\infty)$ and $l \in [3,\infty)$. Put    $\eta_{s,l}= (1-s^{2})l^{s}$,   $a_{s,l}=\frac{1}{2}s(s+1)l^{s-1}$ and    $b_{s,l}= \frac{1}{2} s(s-1)l^{s+1}$  and 
\begin{equation}F_{s,l}(t) = \left\{ {\begin{array}{*{20}{l}}
|t|^{s}&{~if~ |t| ~\le l,}\vspace{.1in}\\
\eta_{s,l}+a_{s,l}|t| +b_{s,l}|t|^{-1}&{~if~ l <|t|}
\end{array}} \right. .
\label{f1b}
\end{equation}\hk
\begin{equation}
G_{s,l}=F_{s,l}F'_{s,l}.
\label{f1b2}
\end{equation}
\label{def61}
\end{definition}\hk
 We have the following results.
\begin{lemma}~Let  $s \in (1,\infty)$ and $l \in [3,\infty)$. Then\\\hk
$(i)$~ $F_{s,l}$ and $G_{s,l}$ are in $C^{1}(\RR)$ ,\\\hk
$(ii)$~~$F'_{s,l}$ and $G'_{s,l}$ are  in $L^{\infty}(\RR)$ and uniformly continuous on $\RR$ and 
\begin{equation}|tF'_{s,l}(t)|\le 4sF_{s,l}(t)\hh \forall~t\in\RR,
\label{ss1}
\end{equation} 
\begin{equation}|F'_{s,l}|^{2}\le s^{2}G'_{s,l}(t)\hh \forall~t\in \RR,
\label{ss2}
\end{equation}
\begin{equation}|G_{s,l}(t)|= sF_{s,l}
^{2-\frac{1}{s}}(t)
\hh ~if~|t|\le 1,
\label{ss3}
\end{equation}
\begin{equation}F_{s,l}(t)\le F_{s,k}(t)
\hh \forall~s\in(1,\infty), t\in\RR,~ l<k.
\label{ss4}
\end{equation}
\label{lem62}
\end{lemma}
\begin{proof} 
$(i)$~~Note that
 \begin{equation}F'_{s,l}(t) = \left\{ {\begin{array}{*{20}{l}}
s.sign t|t|^{s-1} &\hspace{-.5in}{~if~t \in (-l,l],}\vspace{.1in}\\
sign t[ a_{s,l} -b_{s,l}|t|^{-2} ]&{~if~ l<|t|< \infty,}
\end{array}} \right.
\label{f12b}
\end{equation}
\begin{equation}F''_{s,l}(t) = \left\{ {\begin{array}{*{20}{l}}
s(s-1)|t|^{s-2} &{~if~|t| \in [0,l],}\vspace{.1in}\\
2b_{s,l}|t|^{-3}&{~if~ l<|t|<\infty,}
\end{array}} \right.
\label{f12c}
\end{equation}
\begin{equation}
G'_{s,l}(t) = (F'_{s,l})^{2}(t)+F_{s,l}F''_{s,l}(t)
\label{f12f}
\end{equation}\hk
\begin{equation}
G_{s,l}(t) = s(sign t)|t|^{2s-1}\hh \forall~t\in (-l,l),
\label{f12d}
\end{equation}
\begin{equation}
G'_{s,l}(t) = s(2s-1)|t|^{2(s-1)}\hh \forall~t\in (-l,l).
\label{f12e}
\end{equation}\hk
 Since   $s>1$, by \eqref{f1b}, \eqref{f12b}, \eqref{f12d} and \eqref{f12e}, $F_{s,l}|_{(\hspace{-.02in}-2,2)}$ and $G_{s,l}|_{(\hspace{-.02in}-2,2)}$ are of class $C^{1}((\hspace{-.02in}-2,2))$. By definitions of $a_{s,l}$, $b_{s,l}$ and $\eta_{s,l}$, we have
  \begin{equation}
 \lim_{t\to l^{-}}F''(t)=s(s-1)l^{s-2}=2b_{s,l}|l
 |^{-3}=\lim_{t\to l^{+}}F''(t),
  \label{f8}
\end{equation}
 \begin{equation}
\lim_{t\to l^{-}}F'(t)= sl^{s-1}= \frac{1}{2}s(s+1)l^{s-1}-   \frac{1}{2}s(s-1)l^{s-1}= a_{s,l}-b_{s,l}l^{-2}=\lim_{t\to l^{+}}F'(t),
  \label{f9}
\end{equation}
 \begin{equation}
\lim_{t\to l^{-}}F(t)= l^{s}=(1-s^{2})l^{s}+\frac{1}{2}s(s+1)l^{s} + \frac{1}{2}s(s-1)l^{s}
  \label{f10}
\end{equation}
\[=\eta_{s,l}+a_{s,l}l +b_{s,l}l^{-1}=\lim_{t\to l^{+}}F(t). 
\]\hk
By \eqref{f12b}-\eqref{f12d} and \eqref{f8}-\eqref{f10}, $F'_{s,l}|_{\RR\setminus(-1,1)}$ and $G_{s,l}|_{\RR\setminus (-1,1)}$ are of class $C^{1}(\RR\setminus(-1,1))$. Thus $F_{s,l}$ and $G_{s,l}$ are in $C^{1}(\RR)$.\\\hk
$(ii)$~~ By \eqref{f12b} and \eqref{f12c}, $F'_{s,l}$ is  uniformly continuous and bounded on $\RR$. Thus, by  \eqref{f12f}, $G'_{s,l}$ is uniformly continuous and bounded on $\RR$.  \\\hk
  Let $t\in \RR\setminus [-l,l]$. Put $x=l^{-1}|t|$. By defintions, we have
   \begin{equation}
 a_{s,l}-b_{s,l}|t|^{-2}=\frac{1}{2}s(s+1)l^{s-1}-\frac{1}{2} s(s-1)l^{s+1}|t|^{-2}\hspace{-0.16in}\geb{l.|t|^{-1} \le 1
 }\hspace{-0.04in}\frac{1}{2}sl^{s-1}[(s+1)-(s-1)] ,
  \label{f11}
  \end{equation}
  \[4sF_{s,l}(t)-|tF'_{s,l}(t)| \eqb{\eqref{f11}} 4s\eta_{s,l} +4sa_{s,l}|t|+4sb_{s,l}|t|^{-1}-a_{s,l}|t|+b_{s,l}|t|^{-1}   \]
\[=4s(1-s^{2})l^{s}+2s^{2}(s+1)l^{s-1}|t|+ 2s^{2}(s-1)l^{s+1}|t|^{-1}-\frac{1}{2}s(s+1)l^{s-1}|t|\]
\[+\frac{1}{2} s(s-1)l^{s+1}|t|^{-1}
\]  
\[=4s(1-s^{2})l^{s}+(2s^{2}-\frac{1}{2}s)(s+1)l^{s-1}|t|+(2s^{2}+\frac{1}{2}s)(s-1)l^{s+1}|t|^{-1}\]
\[=l^{s+1}|t|^{-1}\{4s(1-s^{2})l^{-1}|t| +(2s^{2}-\frac{1}{2}s)(s+1)l^{-2}t^{2} +(2s^{2}+\frac{1}{2}s)(s-1)\}
\]
\[=l^{s+1}|t|^{-1}f(x),\]
where
\[f(x)=(2s^{2}-\frac{1}{2}s)(s+1)x^{2}+4s(1-s^{2})x+(2s^{2}+\frac{1}{2}s)(s-1)\hk\forall~x\in \RR.\]\hk
Let $x$ in $[1,\infty)$. By computations, we get
\[f'(x)=2(2s^{2}-\frac{1}{2}s)(s+1)x+4s(1-s^{2})\]
\[\geb{x\ge 1}\hspace{-.1in} 2(2s^{2}-\frac{1}{2}s)(s+1)+4s(1-s^{2})= 3s^{2}+3s\ggeb{s>1}0\hk\forall~x\in [1,\infty),\]
\[f(1)=(2s^{2}-\frac{1}{2}s)(s+1)+4s(1-s^{2})+(2s^{2}+\frac{1}{2}s)(s-1)=3s>0.\]\hk
It follows that $f(x)>0$ for every $x$ in $[1,\infty)$ and
\begin{equation}
|tF'_{s,l}(t)|\le 4sF_{s,l}(t)\hh\forall~s\in(1,\infty), t\in \RR\setminus [-l,l].
  \label{ff11}
\end{equation}\hk
On the other hand
\begin{equation}
|tF'_{s,l}(t)|= s|t|^{s} =sF_{s,l}(t)\hh\forall~ t\in [-l,l].
 \label{f11b}
  \end{equation}\hk
 Combining \eqref{ff11} and  \eqref{f11b}, we get \eqref{ss1}.\\\hk
      By \eqref{ss1} and \eqref{f12c}, $F_{s,l}$ and $F''_{s,l}$ are non-negative on $\RR$. It implies
   \begin{equation}
|F'_{s,l}(t)|^{2}\leb{\eqref{f12f}} G'_{s,l}(t)\hh\forall~ t\in \RR
 \label{f11bb}
  \end{equation}
 and we get  \eqref{ss2}.  \\\hk
  We have
 \[|G(t)|=s|t|^{2s -1}=s|t|^{s(2 -\frac{1}{s})}=sF(t)^{2 -\frac{1}{s}}\hk if~|t|\le 1, \]
 which implies \eqref{ss3}.  \\\hk
  Fix $s\in(1,\infty)$, $l\in (3,\infty)$ and $k \in(l,\infty)$.  Put $x=l^{-1}t\ge 1$ for every $t\in [l,k]$,
\begin{equation}g(t)= t^{s}- (1-s^{2})l^{s}-\frac{1}{2}s(s+1)l^{s-1}t - \frac{1}{2} s(s-1)l^{s+1}|t|^{-1}\hh\forall~t\in [l,k],
\label{lk0}
 \end{equation}
\begin{equation}h(x)=sx^{s-1}-\frac{1}{2}s(s+1)+\frac{1}{2}s(s-1)x^{-2}\hh\forall~x\in[1,\infty).
\label{lk0b}
 \end{equation}\hk
 We have
 \begin{equation}h(1)=0,
 \label{lk1}
 \end{equation}
 \begin{equation}h'(x)=s(s-1)x^{s-2} -s(s-1)x^{-3}\ge 0\hh\forall~x\in[1,\infty), \label{lk2}
 \end{equation}
 \begin{equation}
 h(x)\geb{\eqref{lk1},\eqref{lk2}} 0 \hh\forall~x\in[1,\infty),
  \label{lk3}
 \end{equation}
 \begin{equation}g(l)=0,
\label{lk4}
 \end{equation}
 \begin{equation}g'(t)= st^{s-1}-\frac{1}{2}s(s+1)l^{s-1}+\frac{1}{2} s(s-1)l^{s+1}|t|^{-2}
\label{lk5}
 \end{equation}
\[\eqb{\eqref{lk0b}}l^{s-1}h(x) \geb{\eqref{lk3}} 0\hh\forall~t\in [l,k].
\]\hk
Combining \eqref{lk4} and \eqref{lk5}, by  mean value theorem, we  have
\[t^{s}\ge (1-s^{2})l^{s}+\frac{1}{2}s(s+1)l^{s-1}t + \frac{1}{2} s(s-1)l^{s+1}|t|^{-1}\hh\forall~t\in [l,k]
\]
or 
\begin{equation}F_{s,k}(t)\ge F_{s,l}(t)\hh\forall~t\in [l,k].
\label{sla}
\end{equation}\hk
Fix $t$ in $[k,\infty)$. Put
\[\overline{h}(\xi)= F_{s,\xi}(t)=(1-s^{2})\xi^{s}+ \frac{1}{2}s(s+1)\xi^{s-1}t+\frac{1}{2}s(s-1)\xi^{s+1}t^{-1}\hk\forall~\xi\in[l,k].\]
\hk
Since $\xi^{-1}t++\xi t^{-1} \ge 2$, we have
\[\overline{h}'(\xi)= s(1-s^{2})\xi^{s-1}+ \frac{1}{2}s(s^{2}-1)\xi^{s-2}t++\frac{1}{2}s(s^{2}-1)\xi^{s}t^{-1}\]
\[= s(s^{2}-1)\xi^{s-1}[-1+ \frac{1}{2}(\xi^{-1}t++\xi t^{-1})]\ge 0 \hk\forall~\xi\in[l,k].\]\hk
It implies
\begin{equation}F_{s,k}(t)=\overline{h}(k)\ge  \overline{h}(l) =F_{s,l}(t)\hh\forall~t\in [k,\infty).
\label{slb}
\end{equation}\hk
Combining \eqref{sla} and \eqref{slb}, we get \eqref{ss4}. 
\end{proof}\hk

For the case  $s\in (\frac{1}{2},1]$, we have following auxiliary functions.
   \begin{definition} Let  $s \in (\frac{1}{2},1]$. Put     
\begin{equation}\theta(t) = \left\{ {\begin{array}{*{20}{l}}
\frac{3}{8} |t|^{\frac{1}{2}}(|t|^{2}-\frac{10}{3}|t| +5)&{~if~ |t| ~\le 1,}\vspace{.1in}\\
1 &{~if~ 1 <|t|,}
\end{array}} \right. 
\label{fs1}
\end{equation}
\begin{equation}\overline{\theta}_{s}(t) = |t|^{s}\hh\forall~t\in\RR ,
\label{fs2}
\end{equation}
\begin{equation}F_{s} = \theta\overline{\theta}_{s},
\label{fs3}
\end{equation}
\begin{equation}\overline{F}(t) = \left\{ {\begin{array}{*{20}{l}}
0&{~if~ |t| ~\le 1,}\vspace{.1in}\\
|t| &{~if~ 1 <|t|,}
\end{array}} \right. 
\label{fs1z}
\end{equation}
\begin{equation}
G_{s}=F_{s}F'_{s}.
\label{fs4}
\end{equation}
\label{defs1}
\end{definition}\hk
We shall use the commands D[f,t], Expand[expr], Collect[expr,t] of the software Mathematica in computations. In order to verify a equation, we can use the commands  Expand and Collect to the most complicated side to get the remain side of the equation. Of course we can do these jobs by hand.  We have the following results.
\begin{lemma} Let $s\in (\frac{1}{2},1]$. We have\\\hk
$(i)$~ $F_{s}$  is in $C^{1}(\RR)$ and $F'_{s}$ is in $L^{\infty}(\RR)$ and uniformly continuous on $\RR$,\\\hk
$(ii)$~  $G_{s}$ is in $C^{1}(\RR)$ and $G'_{s}$ is in $L^{\infty}(\RR)$ and uniformly continuous on $\RR$.
\label{lemfs1}
\end{lemma}
 \begin{proof}
 $(i)$~~ We have
 \begin{equation}\theta'(t) \eqb{D[.],Expand[.]} 
\frac{3}{8}signt. |t|^{-\frac{1}{2}}(\frac{5}{2}t^{2}-5|t|+\frac{5}{2})\hh~if~|t|\le 1,
\label{fs500}
\end{equation}
 \begin{equation}\theta''(t) \eqb{D[.]}  
\frac{3}{8} |t|^{-\frac{3}{2}}(\frac{15}{4}t^{2}-\frac{5}{2}t-\frac{5}{4} )\hh~if~|t|\le 1,
\label{fs50}
\end{equation}
\begin{equation} \theta(\RR\setminus(-1,1)) =\{1\},
\label{fs5}
\end{equation}
\begin{equation}  \theta'(\RR\setminus(-1,1)) =\theta_{s}''(\RR\setminus(-1,1)) =\{0\},
\label{fs6}
\end{equation}
\begin{equation}\overline{\theta}_{s}'(t)=s.sign t.|t|^{s-1}\hh\forall~t\in \RR\setminus\{0\},
\label{fs6b}
\end{equation}
  \begin{equation}\overline{\theta}_{s}''(t)=s(s-1)|t|^{s-2}\hh\forall~t\in \RR\setminus\{0\},
  \label{fs6c}
\end{equation}
\begin{equation}F_{s}(t)=\theta\overline{\theta}_{s}(t) \hh\forall~t\in\RR,
\label{fs6z}
\end{equation}
\begin{equation}F_{s}(t)= 
\frac{3}{8} |t|^{s+\frac{1}{2}}(|t|^{2}-\frac{10}{3}|t| +5)\hh{~if~ ~|t| ~\le 1,}
\label{fs6zb}
\end{equation}
\begin{equation}F_{s}'(t)=\theta'\overline{\theta}_{s}(t)+\theta\overline{\theta}_{s}'(t)\hh\forall~t\in \RR,
\label{fs7zb}
\end{equation}
\begin{equation}F_{s}'(t)=[\frac{3}{8} |t|^{s+\frac{1}{2}}(|t|^{2}-\frac{10}{3}|t| +5)]'
\label{fs7zbb}
\end{equation}
\[ \eqb{D[.],Expand[.],Collect[.,t]} 
\frac{3}{8}sign t.|t|^{s-\frac{1}{2}}[ (\frac{5}{2} + s) t^2 - (5 + \frac{10s}{3})|t| + \frac{5}{2} + 5s ]\]
\[\hh{~if~ ~|t| ~\le 1,}\]
\begin{equation}F_{s}''(t)=\theta''\overline{\theta}_{s}(t)+2\theta'\overline{\theta}_{s}'(t)+\theta\overline{\theta}_{s}''(t)\hh\forall~t\in \RR,
\label{fs8z}
\end{equation}
\begin{equation}F_{s}''(t)=[\frac{3}{8}sign t.|t|^{s-\frac{1}{2}}[ (\frac{5}{2} + 5s) t^2 - (5 + \frac{10s}{3})|t| + \frac{5}{2} + 5s ]'
\label{fs8zb}
\end{equation}
\[ \eqb{Expand, Collect[.,t]} \frac{3}{8}|t|^{s-\frac{3}{2}}[(\frac{15}{4}+ 4 s + s^2) t^2 - (\frac{5}{2} + \frac{20 s}{3} + \frac{10 s^2}{3}) t - \frac{5}{4} + 5 s^2]\]
\[\hh {if~ |t| ~\le 1.} \]\hk
 We have
 \begin{equation}\lim_{t\to 1^{-}}\theta(t)\eqb{\eqref{fs1}}1=\lim_{t\to 1^{+}}\theta(t),
 \label{fs9z}
\end{equation}
\begin{equation}\lim_{t\to 1^{-}}\theta'(t)\eqb{\eqref{fs500}}0=\lim_{t\to 1^{+}}\theta'(t),
\label{fs10z}
\end{equation} 
 \begin{equation}\lim_{t\to 1^{-}}\theta''(t)\eqb{\eqref{fs50}}0=\lim_{t\to 1^{+}}\theta''(t).\label{fs11z}
\end{equation}\hk 
 It implies $\theta\in C^{2}(\RR)$. Thus, by \eqref{fs6z} and \eqref{fs9z}-\eqref{fs11z},  $F$ is in $C^{1}(\RR)$ and $F'_{s}$ is  in $L^{\infty}(\RR)$ and uniformly continuous on $\RR$ if $s\in (\frac{1}{2},1]$. \\\hk
 $(ii)$~~ Let $s\in (\frac{1}{2},1]$. We have 
  \[G_{s}(t)= (F_{s}F_{s}')(t)=sign t |t|^{2s}[s|t|^{-1}\theta^{2}(t) +sign t .\theta\theta'(t)]\hh\forall~t\in\RR,\] 
 \[ = \left\{ {\begin{array}{*{20}{l}}
\frac{9}{64}sign t |t|^{2s}(\hspace{-.01in}|t|^{2}\hspace{-.01in}-\hspace{-.01in}\frac{10}{3}|t|\hspace{-.01in} +5\hspace{-.01in})[ (\frac{5}{2} + s) t^2 - (5 + \frac{10s}{3})|t| + \frac{5}{2} + 5s ]\hspace{-1in}  &\vspace{.1in}\\
&{~if~ |t| \le\hspace{-.02in} 1,}\vspace{.1in}\\
s.sign t |t|^{2s-1} &{~if~ 1 <|t|,}
\end{array}} \right. \]
  \begin{equation}G_{s}'(t)= (F_{s}')^{2}(t)+F_{s}F_{s}''(t)\hh\forall~t\in\RR,
\label{fs12z}
\end{equation}
\[= |t|^{2s-2}[s^{2}+s(s-1)]=s(2s-1)|t|^{2s-2} \hh~if~ 1 <|t|,\]
 \begin{equation}
G_{s}'(t) =
\{\frac{9}{64}sign t |t|^{2s}(\hspace{-.01in}|t|^{2}\hspace{-.01in}-\hspace{-.01in}\frac{10}{3}|t|\hspace{-.01in} +5\hspace{-.01in})[ (\frac{5}{2} + s) t^2 - (5 + \frac{10s}{3})|t| + \frac{5}{2} + 5s ]\}'
\label{fs12zb}
\end{equation}
\[ \eqb{D[],Expand, Collect}\frac{9}{64}t^{2 s -1}[(10 + 9 s + 2 s^2) t^4 - (40 + \frac{140 s}{3} + \frac{40 s^2}{3})t^3 \]
\[+ (\frac{190}{3} + \frac{950 s}{9}+ \frac{380 s^2}{9})t^2 - (\frac{100}{3} +   100 s + \frac{200 s^2}{3})t + 25 s + 50 s^2]\hk if~|t|\le 1.\]
\hk
Since $0<2s-1\le 1$, by \eqref{fs9z}-\eqref{fs11z}, $G_{s}$ is in $C^{1}(\RR)$ and $G'_{s}$ is  in $L^{\infty}(\RR)$ and uniformly continuous on $\RR$. \end{proof}
\begin{lemma}  There are  positive real numbers  $\delta\in (0,\frac{1}{3})$ and $c_{0}$ such that
 \begin{equation} 
\hh\hh|F'_{s}|^{2}\le c_{0} G'_{s}(t)\hh \forall~s\in(\frac{2}{3}-\delta,1], t\in \RR,
\label{sss2}
\end{equation}
\label{lemfs1b}
\end{lemma}
 \begin{proof} ~~ Let $(\alpha,s,t)\in (0,1]\times (\frac{1}{2},1]\times [-1,1]$.  Put
 \begin{equation}h(\alpha,s,t)= [\frac{9}{64}t^{2s -1}]^{-1}[\alpha (F_{s}')^{2}(t)+F_{s}F_{s}''(t)].
\label{hh}
\end{equation}\hk
We prove this lemma by following steps.\\\hk
{\bf Step 1.}~~We have
\begin{equation}2t^4 - \frac{40}{3}t^3 + \frac{380}{9}t^2 - \frac{200}{3}t + 50
\label{0z}
\end{equation}
\[\eqb{Expand} 2(t^2 - \frac{10}{3}t+1)^2 + 16(t- \frac{80}{3.16})^{2} -16(\frac{5}{3})^{2} + 48>3\hh\forall~t\in[0,1].\]\hk
Fix $t$ in $[0,1)$. Put $k(s)= h(1,s,t)$ for every $s$ in $(\frac{1}{2},1]$. Then
\begin{equation}k(s)\eqb{\eqref{hh}}[\frac{9}{64}t^{2s -1}]^{-1}G_{s}'(t)\eqb{\eqref{fs12zb}}(10 + 9 s + 2 s^2) t^4 - (40 + \frac{140 s}{3} + \frac{40 s^2}{3})t^3 
\label{0zaz}
\end{equation}
\[ + (\frac{190}{3} + \frac{950 s}{9}+ \frac{380 s^2}{9})t^2- (\frac{100}{3} +   100 s + \frac{200 s^2}{3})t + 25 s + 50 s^2
\]
\[\eqb{Expand[.],Collect[.,s]}(50 - \frac{200 t}{3} + \frac{380 t^2}{9}- \frac{40 t^3}{3}+ 2 t^4)s^2 \] 
\[+ (25 - 100 t + \frac{950 t^2}{9}- \frac{140 t^3}{3}+9 t^4)s-\frac{100 t}{3} + \frac{190 t^2}{3} - 40 t^3 + 10 t^4,\]
\begin{equation}k'(s)\hspace{-.06in}\eqb{D[.]}\hspace{-.06in}2(2t^4 - \frac{40}{3}t^3 + \frac{380}{9}t^2 - \frac{200}{3}t + 50)s + (9 t^4 - \frac{140}{3}t^3 + \frac{950}{9}t^2 - 100t+25)
\label{0zazb}
\end{equation}
\[\geb{\eqref{0z}, s>\frac{1}{2}} (2t^4 - \frac{40}{3}t^3 + \frac{380}{9}t^2 - \frac{200}{3}t + 50) + (9 t^4 - \frac{140}{3}t^3 + \frac{950}{9}t^2 - 100t+25)\]
\[\eqb{Expand[.],Collect[.,t]}11 t^{4}- 60 t^3 + \frac{1330}{9}t^2- \frac{500}{3}t  + 75.\]\hk
 On the other hand
\[0<11(t^{2}- \frac{30}{11}t + 2)^2 +\frac{2174 }{99} (t - \frac{99.70}{3.2174})^2 + \frac{6747}{1087}
\]
\[\eqb{Expand, Collect}11 t^{4}- 60 t^3 + \frac{1330}{9}t^2- \frac{500}{3}t  + 75 .\] 
\hk
It implies
\[0< k'(s)\le 10^{3} \hh\forall~s\in (\frac{1}{2},1].
\]
 Thus, by the mean value theorem, $k$ is increasing on $(\frac{1}{2},1]$ and
\begin{equation} 0 < k(s)-k(s')\le 10^{3}(s-s')\hh if~\frac{1}{2}<s'<s\le 1.
\label{s}
\end{equation}\hk
{\bf Step 2.}~~Let $t\in [-1,1]$, by computations we have 
\[
h(1,\frac{2}{3},t)= k(\frac{2}{3})\hspace{-.05in} \eqb{\eqref{0zaz},Expand} \hspace{-.05in} \frac{152}{9}t^4 - \frac{2080}{27}t^3  +  \frac{12350}{81}t^2- \frac{3500}{27}t + \frac{350}{9}.\]\hk
 On other hand
\[10^{-2}<\frac{152}{9}(t^2 - \frac{1040}{456}t +\frac{5}{4})^2 + \frac{3830}{171}(t - \frac{285}{383})^2  - \frac{3830}{171}(\frac{285}{383})^2 + \frac{225}{18}\]
\[\eqb{Expand}\frac{152}{9}t^4 - \frac{2080}{27}t^3  +  \frac{12350}{81}t^2- \frac{3500}{27}t + \frac{350}{9}.\]
\hk
 Thus 
\begin{equation}
h(1,\frac{2}{3},t) > 10^{-2}\hh\forall~t\in [-1,1].
\label{wwzb}
\end{equation}\hk 
{\bf Step 3.}~~Put $\delta= 10^{-4}$. Using \eqref{s} and \eqref{wwzb}, we get 
 \[ h(1,s,t)> 10^{-3}\hh\forall~(s,t)\in (\frac{2}{3}-\delta,\frac{2}{3}]\times[0,1].\] \hk
By \eqref{0zazb},  $h(1,.,t)$ is increasing on $(\frac{1}{2},1]$. Thus
 \begin{equation}
 h(1,s,t)> 10^{-3}\hh\forall~(s,t)\in (\frac{2}{3}-\delta,1]\times[0,1].
 \label{h1}
 \end{equation}\hk
{\bf Step 4.}~~ Note that for any $t\in [-1,1]$ and $\alpha \in [0,1]$
\[\frac{\partial h(\alpha,s,t)}{\partial \alpha}\hspace{-.07in}\eqb{\eqref{hh}} [\frac{9}{64}t^{2s -1}]^{-1} (F_{s}')^{2}(t)\hspace{-.07in}\eqb{\eqref{fs7zbb}} [ (\frac{5}{2} + s) t^2 - (5 + \frac{10s}{3})|t| + \frac{5}{2} + 5s ]^{2}\le 100. 
\]\hk
Choosing $\alpha_{0} =1-10^{-6}$, by \eqref{h1}  and the mean value theorem, we get
\begin{equation}h(\alpha_{0},s,t)\ge h(1,s,t) - 10^{-4}\ggeb{\eqref{h1}}10^{-4}\hk\forall~(s,t)\in (\frac{2}{3}-\delta, 1]\times [0,1]. 
\label{alpha}
\end{equation}\hk
{\bf Step 5.}~~ Put $\beta=\frac{1}{1-\alpha_{0}}$. By computations, we have
 \[\beta G_{s}'(t)-(F_{s}'(t))^{2}=((\beta-1)F'^{2}+ \beta FF'')(t)=\beta(\alpha_{0} F'^{2}+ FF'')(t)
 \]
\[\eqb{\eqref{hh}}\frac{9\beta}{64}t^{2s-1}h(\alpha_{0},s,t)\ggeb{\eqref{alpha}}\frac{9 t^{2s-1}}{10^{4}64(1-\alpha_{0})}\ge 0\]
or
\[(F_{s}'(t))^{2}\le \beta G_{s}'(t)\hh\forall~(s,t)\in (\frac{2}{3}-\delta,1]\times [-1,1].\]\hk
 On the other hand
 \[2G'(t)- (F'(t))^{2}=(F'(t))^{2}+2(FF'')(t)=s^{2}t^{2s-2}+ s(s-1)t^{2s-2} \]
\[= st^{2s-2}(2s-1)\geb{\frac{1}{2}<s} 0 \hh if~~ 1\le  |t|, s> \frac{1}{2}.\]\hk
Putting  $c_{0}= \max\{\frac{1}{1-\alpha_{0}},2\}$, we get the lemma.\end{proof}
\begin{lemma} Let  $\delta$ be in Lemma \ref{lemfs1b}.  There are a positive real numbers  $k_{s}$  and $k_{0}$ such that
\begin{equation}
|tF'_{s}(t)|\le 5F_{s}(t)\hh\forall~s\in(\frac{2}{3}-\delta,1],t\in\RR,
\label{sss1}
\end{equation}
\begin{equation}
 |G(t)|\le k_{s}F^{2-\frac{1}{s}}(t)\hh if~|t|\le 1,
 \label{sss3}
\end{equation}
\begin{equation} 
 \overline{F}(t)\le F_{s}^{\frac{1}{s}}(t)\le \overline{F}(t)+k_{0}\hh \forall~t\in \RR.
 \label{sss4}
\end{equation}
\label{lemfs1c}
\end{lemma}
 \begin{proof}~~  By computations, we obtain
 \[5F_{s}(t)-|tF_{s}'(t)|=\]
 \[=\frac{3}{8} |t|^{s+\frac{1}{2}}[5(|t|^{2}-\frac{10}{3}|t| +5)-((\frac{5}{2} + s) t^2 - (5 + \frac{10s}{3})|t| + \frac{5}{2} + 5s )]
 \]
\[ \eqb{Expand,Collect[.,t]}\frac{3}{8} |t|^{s+\frac{1}{2}}[\frac{5-2s}{2} t^{2}-\frac{35-10s}{3}|t|+\frac{45-10s}{2}]\]
 \[\geb{0<s\le 1, |t|\le 1}\frac{3}{8} |t|^{s+\frac{1}{2}}[-\frac{35}{3}+\frac{35}{2}]\ge 0\hh~\if~|t|\le 1,\]
 \[5F_{s}(t)-|tF_{s}'(t)|=(5-s)|t|^{s}\geb{0<s\le 1} 0\hh if~1\le |t|.\]\hk
  Thus we get \eqref{sss1}.\\\hk     
  Since  $\theta$, $\theta'$, $\overline{\theta}_{s}$ and $\overline{\theta}_{s}'$ are continuous on $[-1,1]$ and $\frac{1}{s}-1\ge 0$, there is a positive real number $k_{s}$ such that
  \[|G_{s}(t)|=|\theta\overline{\theta}_{s}(t)[\overline{\theta}_{s}'\theta+\overline{\theta}_{s}\theta')(t)]|= [\theta\overline{\theta}_{s}]^{2-\frac{1}{s}}(t)|\overline{\theta}_{s}'\theta^{\frac{1}{s}}\overline{\theta}_{s}^{\frac{1}{s}-1}+\overline{\theta}_{s}^{\frac{1}{s}}\theta'\theta^{\frac{1}{s}-1}|(t)
\]  
\[\eqb{s\le 1} [F_{s}]^{2-\frac{1}{s}}(t)|\overline{\theta}_{s}'\theta^{\frac{1}{s}}\overline{\theta}_{s}^{\frac{1}{s}-1}+\overline{\theta}_{s}^{\frac{1}{s}}\theta'\theta^{\frac{1}{s}-1}|(t)\le k_{s}[F_{s}]^{2-\frac{1}{s}}(t)\hk if~|t|\le 1.
\]  \hk
Thus we get \eqref{sss3}.\\\hk 
Note that
\[F(t)= F_{s}^{\frac{1}{s}}(t)\hh if~ |t|\ge 1,\]
\[F(t)=0\le F_{s}^{\frac{1}{s}}(t)\hh if~ |t|\le 1,\]
\[F_{s}^{\frac{1}{s}} (t)=  |t|^{1+\frac{1}{2s}}[\frac{3}{8}(|t|^{2}-\frac{10}{3}|t| +5)]^{\frac{1}{s}}\le [\frac{3}{8}(1+\frac{10}{3} +5)]^{\frac{1}{s}}\le{k_{0}}~,~ if~|t|\le 1, s>\frac{1}{2},\]
where $k_{0}= [\frac{3}{8}(1+\frac{10}{3} +5)]^{2}$. \\\hk
It implies \eqref{sss4}.
\end{proof}

\section{Boundedness for  solutions of elliptic equations}\label{bo}\hk
 We have following global boundedness of solutions  as follows. 
\begin{proposition}[Global boundedness] Let  $\delta$ be as in Lemma \ref{lemfs1b},  $\gamma \in [1,\infty)$,   $A$ be  admissible with respect to $\Omega$, $a$ be a positive  measurable functions on $\Omega$, $a_{0}$, $a_{1}$, $a_{2}$ be non-negative  measurable functions on $\Omega$ and $u\in W_{B,A}(\Omega)\cap L^{\gamma}(\Omega)$. Assume \eqref{c1}, \eqref{c3} and \eqref{c6} hold, $a\in L^{\frac{\overline{r}}{\overline{r}-2}}(\Omega)$,  
\begin{equation}
b^{-1} a_{0}^{2}+a_{1}+ a_{2}\le a
,
\label{cc8}
 \end{equation}
 \begin{equation}
\frac{\gamma}{\overline{r}}> \frac{2}{3}-\delta
~,
\label{cc9}
 \end{equation}
  \begin{equation}\int_{\Omega}  \frac{\partial u}{\partial x_{i}}\frac{\partial \varphi}{\partial x_{j}}b^{ij}dx \le \int_{\Omega} [a_{0}|\nabla u|+a_{1}|u|  +a_{2}]|\varphi| dx 
 \label{le460}
 \end{equation}
for every $\varphi\in W_{B,A}(\Omega)$.\\\hk
 Then there is a real number $c(r,\gamma,\Omega,||u||_{L^{\gamma}(\Omega)})$ such that 
 \begin{equation}||u||_{L^{\infty}(\Omega)}\le c(r,\gamma,\Omega,||u||_{L^{\gamma}(\Omega)}). \label{le460abc}
 \end{equation}
 \label{pro46}\end{proposition}
  \begin{proof} We prove this proposition by two steps.\\\hk
  {\bf Step 1}. Assume $\gamma > \overline{r}$.   
    Put $\kappa = \frac{r}{\overline{r}}$, $s_{m}= \kappa^{m}\frac{\gamma}{\overline{r}}\ge \frac{\gamma}{\overline{r}}>1 $ and $\overline{s}_{m}= \kappa^{m}\gamma$ for every non-negative  integer $m$. \\\hk
     Fix a non-negative  integer $m$.  Let $l\in [3,\infty)$, $F_{l}\equiv F_{s_{m},l}$ and $G_{l}\equiv G_{s_{m},l}$ be defined as in Definition \ref{def61}. Since $F_{s_{m},l}(0)=G_{s_{m},l}(0)=0$,  Lemma \ref{2w1ba} and Lemma \ref{w1bb}, $w\equiv u^{+}$,  $v_{s_{m},l}\equiv F_{l}\circ u^{+}$ and $\varphi\equiv G_{l}\circ u^{+}$ are in $W_{B,A}(\Omega)$.   We have 
 \begin{equation}\frac{\partial w}{\partial x_{j}}(x) = \left\{ {\begin{array}{*{20}{l}}
\displaystyle\frac{\partial u}{\partial x_{j}}(x) \hh&{\forall~x \in \Omega^{+}=\{z\in\Omega:u(z)>0\},}\vspace{.1in}\\
0&{\forall~x \in \Omega\setminus\Omega^{+}}
\end{array}} \right.
\label{le462aazz20}
\end{equation}
 \begin{equation}\frac{\partial v_{s_{m},l}}{\partial x_{j}}(x) = \left\{ {\begin{array}{*{20}{l}}
\displaystyle F_{l}'(w(x))\frac{\partial u}{\partial x_{j}}(x) \hh&{\forall~x \in \Omega^{+},}\vspace{.1in}\\
0&{\forall~x \in \Omega\setminus\Omega^{+}}
\end{array}} \right.
\label{le462aazz2}
\end{equation}
\[\eqb{F_{l}'(0)=0}F_{l}'(w(x))\frac{\partial w}{\partial x_{j}}(x)\hh\forall~x \in \Omega,\]
 \begin{equation}\frac{\partial \varphi}{\partial x_{j}}(x) = \left\{ {\begin{array}{*{20}{l}}
\displaystyle G_{l}'(w(x))\frac{\partial u}{\partial x_{j}}(x)  \hh&{\forall~x \in \Omega^{+},}\vspace{.1in}\\
0&{\forall~x \in \Omega\setminus\Omega^{+},}
\end{array}} \right.
\label{le462aazz3}
\end{equation}
\[\eqb{G_{l}'(0)=0}G_{l}'(w(x))\frac{\partial w}{\partial x_{j}}(x)\hh\forall~x \in \Omega,\]\hk
Since
  \[\frac{\frac{1}{s_{m}}}{\overline{r}}+\frac{\overline{r}-2}{\overline{r}}+\frac{2-\frac{1}{s_{m}}}{\overline{r}}=1,\]
by H\"{o}lder's inequality, we have
\begin{equation}
  \int_{w\le 1}F_{l}^{2-\frac{1}{s_{m}}}a dx\le |\Omega|^{\frac{1}{\overline{r}s_{m}}}||a||_{L^{\frac{\overline{r}}{\overline{r}-2}}(\Omega)}\{\int_{\Omega}F_{l}^{\overline{r}} dx\}^{\frac{2-\frac{1}{s_{m}}}{\overline{r}}}
 \label{c}
 \end{equation} 
 Using $\varphi$ as a test function  in \eqref{le460},   we have
   \begin{equation}
\int_{\Omega}|\nabla v_{s_{m},l}|^{2} bdx \leb{\eqref{c1}} \int_{\Omega}\frac{\partial v_{s_{m},l}}{\partial x_{i}} \frac{\partial v_{s_{m},l}}{\partial x_{j}}b^{ij}dx
\leb{\eqref{le462aazz2}} \int_{\Omega^{+}}\frac{\partial u}{\partial x_{i}} \frac{\partial u}{\partial x_{j}}|F_{l}'(w)|^{2}b^{ij}dx 
\label{le462aazz2b}\end{equation}
\[\leb{F'(0)=0} \int_{\Omega}\frac{\partial w}{\partial x_{i}} \frac{\partial w}{\partial x_{j}}|F_{l}'(w)|^{2}b^{ij}dx \hspace{-.1in}\leb{\eqref{ss2}}\hspace{-.073in} s_{m}^{2}\hspace{-.03in}
  \int_{\Omega}\frac{\partial w}{\partial x_{i}} \frac{\partial w}{\partial x_{j}}G_{l}'(w) b^{ij}dx\hspace{-.15in} \eqb{\eqref{le462aazz3}} \hspace{-.04in}s_{m}^{2} \hspace{-.06in}\int_{\Omega}\frac{\partial u}{\partial x_{i}}\frac{\partial \varphi}{\partial x_{j}} b^{ij}dx \] 
    \[\leb{\eqref{le460}} s_{m}^{2}\int_{\Omega}|\nabla w||\varphi| a_{0} dx+s_{m}^{2}\int_{\Omega}(|u|+1) |\varphi| a dx \]  
     \[\eqb{\eqref{cc8}} s_{m}^{2}\hspace{-.051in}\int_{\Omega}\hspace{-.051in}|\nabla w| |G_{l}(w)|b^{\frac{1}{2}}a^{\frac{1}{2}}dx+s_{m}^{2}\hspace{-.051in}\int_{\Omega}\hspace{-.051in}|w||G_{l}(w)| a dx+s_{m}^{2}\int_{\Omega}|G_{l}(w)| a dx\]
 \[ \leb{\eqref{f1b2},\eqref{ss3}} \hspace{-.11in}s_{m}^{2}\hspace{-.071in}\int_{\Omega}F_{l}(w)|\nabla w||F_{l}'(w)|b^{\frac{1}{2}}a^{\frac{1}{2}}dx+s_{m}^{2}\hspace{-.07in}\int_{\Omega}|w|F_{l}(w)|F_{l}'(w)| a dx
   \]
  \[+s_{m}^{2}\int_{w> 1}\hspace{-.171in}|w|F_{l}(w)|F_{l}'(w)| a dx+s_{m}^{3}\int_{w\le 1}\hspace{-.17in}F_{l}^{2-\frac{1}{s_{m}}}a dx\]
  \[ \leb{\eqref{ss1}} \hspace{-.1in} s_{m}^{2}\hspace{-.051in}\int_{\Omega}\hspace{-.051in}F_{l}(w)|\nabla w||F_{l}'(w)|b^{\frac{1}{2}}a^{\frac{1}{2}} dx+8s_{m}^{3}\hspace{-.051in} \int_{\Omega}\hspace{-.051in}F_{l}(w)^{2} a dx+ s_{m}^{3}\int_{w\le 1}\hspace{-.17in}F_{l}^{2-\frac{1}{s_{m}}}a dx\]
       \[ \leb{Young,\eqref{cc8},H\ddot{o}lder,\eqref{c}}\hspace{-.15in} 2^{-1} \int_{\Omega}|\nabla w|^{2}F_{l}'(w)^{2}bdx+\frac{1}{2}s_{m}^{4} \int_{\Omega}F_{l}(w)^{2}adx\]
    \[+ 8s_{m}^{3}||a||_{L^{\frac{\overline{r}}{\overline{r}-2}}(\Omega)}\{\int_{\Omega}F_{l}(w)^{\overline{r}} dx\}^{\frac{2}{\overline{r}}}+s_{m}^{3}|\Omega|^{\frac{1}{\overline{r}s_{m}}}||a||_{L^{\frac{\overline{r}}{\overline{r}-2}}(\Omega)}\{\int_{\Omega}F_{l}^{\overline{r}} dx\}^{\frac{2-\frac{1}{s_{m}}}{\overline{r}}}\]
  \[ \leb{\eqref{cc8},\eqref{le462aazz2}} \frac{1}{2} \int_{\Omega}|\nabla v_{s_{m},l}|^{2}bdx+9s_{m}^{4}||a||_{L^{\frac{\overline{r}}{\overline{r}-2}}(\Omega)}\{\int_{\Omega}F_{l}(w)^{\overline{r}} dx\}^{\frac{2}{\overline{r}}}\]
  \[+ s_{m}^{3}|\Omega|^{\frac{1}{\overline{r}s_{m}}}||a||_{L^{\frac{\overline{r}}{\overline{r}-2}}(\Omega)}\{\int_{\Omega}F_{l}^{\overline{r}} dx\}^{\frac{2-\frac{1}{s_{m}}}{\overline{r}}}\] 
\[\leb{2-\frac{1}{s_{m}}\le 2}\frac{1}{2} \int_{\Omega}\hspace{-.04in}|\nabla v_{s_{m},l}|^{2}bdx + 9s_{m}^{4}[(1+|\Omega|^{\frac{1}{\overline{r}s_{m}}})||a||_{L^{\frac{\overline{r}}{\overline{r}-2}}(\Omega)}]\max\{1,\hspace{-.02in}\{\int_{\Omega}\hspace{-.05in}v_{s_{m},l}^{\overline{r}}dx\}^{\frac{2}{\overline{r}}}\}
\]  
  or
  \[
\int_{\Omega}|\nabla v_{s_{m},l}|^{2} bdx \le  18s_{m}^{4}[(|\Omega|+2)||a||_{L^{\frac{\overline{r}}{\overline{r}-2}}(\Omega)}]\max\{1,\hspace{-.02in}\{\int_{\Omega}\hspace{-.05in}v_{s_{m},l}^{\overline{r}}dx\}^{\frac{2}{\overline{r}}}\}
  \]\hk
     Thus
 \[   \{\int_{\Omega}v_{s_{m},l}^{r}dx\}^{\frac{2}{r}}\hspace{-.1in}\leb{\eqref{cb}}\hspace{-.1in} C_{1}^{2}s_{m}^{4}\max\{1,\{\int_{\Omega}\hspace{-.05in}v_{s_{m},l}^{\overline{r}}dx\}^{\frac{2}{\overline{r}}}]
 \] 
  or 
 \begin{equation}\{\int_{\Omega}v_{s_{m},l}^{r} dx\}^{\frac{1}{r}}\le 
  C_{1}s_{m}^{4}\max\{1,\{\int_{\Omega}\hspace{-.05in}v_{s_{m},l}^{\overline{r}}dx\}^{\frac{1}{\overline{r}}}\}]\hh\forall~m=0,1,2,\cdots,
\label{le465}
\end{equation}
where
\[C_{1} =\{18[C(r,\Omega,A)^{2}[(|\Omega|+2)||a||_{L^{\frac{\overline{r}}{\overline{r}-2}}(\Omega)}]+1\}^{\frac{1}{2}}\ge 1. \] \hk
 By \eqref{f1b} and \eqref{ss4}, the sequence $\{v_{s_{m},3n}(x)\}$ is  increasing  and converges  to $w^{s_{m}}(x)$ a.e. on $\Omega$. Thus by  Lebesgue's Monotone Convergence Theorem and \eqref{le465}, we have 
 \[\{\int_{\Omega} w^{s_{m}r}  dx\}^{\frac{1}{r}}\le  
 C_{1}s_{m}^{4}\max\{1,(\int_{\Omega}w^{s_{m}\overline{r}}  dx)^{\frac{1}{\overline{r}}}\}  \hk\forall~m\ge 0
 \]
 or
  \[\{\int_{\Omega} w^{s_{m}r}  dx\}^{\frac{1}{s_{m}r}}\le  
 [C_{1}]^{\frac{1}{s_{m}}}s_{m}^{\frac{4}{s_{m}}}\max\{1,(\int_{\Omega}w^{s_{m}\overline{r}}  dx)^{\frac{1}{s_{m}\overline{r}}}\}  \hk\forall~m\ge 0
 \]\hk
 Note that $\kappa=\frac{r}{\overline{r}}$, $s_{m}= \kappa^{m}\frac{\gamma}{\overline{r}}$, $s_{m}r= \kappa^{m+1}\gamma$, $s_{m}\overline{r}= \kappa^{m}\gamma$, $\frac{4}{s_{m}}=\frac{\overline{r}}{\gamma}\frac{4}{\kappa^{m}} $, $\frac{m}{s_{m}}=\frac{\kappa}{\gamma}\frac{m}{\kappa^{m}} $, 
$s_{m}^{\frac{4}{s_{m}}}= [ \kappa^{m}\frac{\gamma}{\overline{r}}]^{\frac{\overline{r}}{\gamma}\frac{4}{\kappa^{m}}}=[(\frac{\gamma}{\overline{r}})^{\frac{4\overline{r}}{\gamma}}]^{\frac{1}{\kappa^{m}}}[ (\frac{r}{\overline{r}})^{\frac{4\overline{r}}{\gamma}}]^{\frac{m}{\kappa^{m}}}$ and $[C_{1}]^{\frac{1}{s_{m}}}=[C_{1}^{\frac{\overline{r}}{\gamma}}]^{\frac{1}{\kappa_{m}}}
$. Thus, we get
 \begin{equation}\hspace{-.05in}\{\hspace{-.03in}\int_{\Omega} \hspace{-.07in}w^{\kappa^{m+1}\gamma} dx\}^{\frac{1}{\kappa^{m+1}\gamma} } \hspace{-.02in} \le \hspace{-.02in} 
  c_{1}
  ^{\frac{1}{\kappa^{m}}}c_{2}^{\frac{m}{\kappa^{m}}}\max\{1,(\hspace{-.02in}\int_{\Omega}\hspace{-.07in}w^{\kappa^{m}\gamma}  dx)^{\frac{1}{\kappa^{m}\gamma}}\}\hk\forall~m\ge 0,
 \label{le465b}
\end{equation}
where $c_{1}= (\frac{\gamma}{\overline{r}})^{\frac{4\overline{r}}{\gamma}}C_{1}^{\frac{\overline{r}}{\gamma}}$ and $c_{2}=(\frac{r}{\overline{r}})^{\frac{4\overline{r}}{\gamma}}$.\\\hk
 By mathematical induction, we get
\[ \{\int_{\Omega}w^{\kappa^{k}\gamma}dx\}^{\frac{1}{\kappa^{k}\gamma}}
 \le c_{1}^{\sum_{j=1}^{k}\frac{1}{\kappa^{j}}}c_{2}^{\sum_{j=1}^{k}\frac{j}{\kappa^{j}}}\max\{1,(\int_{\Omega}w^{\gamma}  dx)^{\frac{1}{\gamma}}\}\hk\forall~k\in\NN.\]\hk
    Therefore,  by Problem 5 in \cite[p.71]{RU}, it implies 
 \[||u^{+}||_{L^{\infty}(\Omega)}\leq c_{1}^{\sum_{j=1}^{\infty}\frac{1}{\kappa^{j}}}c_{2}^{\sum_{j=1}^{\infty}\frac{j}{\kappa^{j}}}\max\{1,(\int_{\Omega}|u|^{\gamma}  dx)^{\frac{1}{\gamma}}\}. \]\hk
 Replacing $w=u^{+}$ by $w=u^{-}$.  We have 
 \[\frac{\partial w}{\partial x_{j}}(x) = \left\{ {\begin{array}{*{20}{l}}
\displaystyle -\frac{\partial u}{\partial x_{j}}(x) \hh&{\forall~x \in \Omega^{-},}\vspace{.1in}\\
0&{\forall~x \in \Omega\setminus\Omega^{-}}
\end{array}} \right.
\]
 \[\frac{\partial v_{s_{m},l}}{\partial x_{j}}(x) = \left\{ {\begin{array}{*{20}{l}}
\displaystyle - F_{l}'(w(x))\frac{\partial u}{\partial x_{j}}(x) \hh&{\forall~x \in \Omega^{-},}\vspace{.1in}\\
0&{\forall~x \in \Omega\setminus\Omega^{-}}
\end{array}} \right.
\]
\[\eqb{F_{l}'(0)=0}F_{l}'(w(x))\frac{\partial w}{\partial x_{j}}(x)\hh\forall~x \in \Omega,\]
 \[\frac{\partial \varphi}{\partial x_{j}}(x) = \left\{ {\begin{array}{*{20}{l}}
\displaystyle - G_{l}'(w(x))\frac{\partial u}{\partial x_{j}}(x)  \hh&{\forall~x \in \Omega^{-},}\vspace{.1in}\\
0&{\forall~x \in \Omega\setminus\Omega^{-}.}
\end{array}} \right.
\]\hk
Using $\varphi$ as a test function  in \eqref{le460},   we have
   \[
\int_{\Omega}|\nabla v_{s_{m},l}|^{2} bdx \leb{\eqref{c1}} \int_{\Omega}\frac{\partial v_{s_{m},l}}{\partial x_{i}} \frac{\partial v_{s_{m},l}}{\partial x_{j}}b^{ij}dx
\] 
\[= \int_{\Omega^{-}}\frac{\partial u}{\partial x_{i}} \frac{\partial u}{\partial x_{j}}|F_{l}'(w)|^{q}b^{ij}dx \le \int_{\Omega}\frac{\partial w}{\partial x_{i}} \frac{\partial w}{\partial x_{j}}|F_{l}'(w)|^{2}b^{ij}dx \]
  \[\leb{\eqref{ss2}} \hspace{-.12in}s_{m}^{2}\hspace{-.06in}
  \int_{\Omega}\frac{\partial w}{\partial x_{i}} \frac{\partial w}{\partial x_{j}}G_{l}'(w) b^{ij}dx\hspace{-.15in} \eqb{\eqref{le462aazz3}} \hspace{-.04in}s_{m}^{2} \hspace{-.06in}\int_{\Omega}\frac{\partial u}{\partial x_{i}}\frac{\partial \varphi}{\partial x_{j}} b^{ij}dx \] 
   \[\leb{\eqref{le460}} s_{m}^{2}\int_{\Omega}|\nabla u| |\varphi| a_{0} dx+s_{m}^{2}\int_{\Omega}(|u|+1)
    |\varphi|a dx \]   
     \[\eqb{\varphi|_{\RR\setminus\Omega^{-}}=0} s_{m}^{2}\int_{\Omega}|\nabla w| |\varphi| a_{0} dx+s_{m}^{2}\int_{\Omega}(|w|+1) |\varphi| a dx. \] \hk
   Now
    arguing as above, we get 
    \[||u^{-}||_{L^{\infty}(\Omega)}\leq c_{1}^{\sum_{j=1}^{\infty}\frac{1}{\kappa^{j}}}c_{2}^{\sum_{j=1}^{\infty}\frac{j}{\kappa^{j}}}\max\{1,(\int_{\Omega}|u|^{\gamma}  dx)^{\frac{1}{\gamma}}\}. \]\hk
     Thus we obtain the proposition for $\gamma > \overline{r}$.\\\hk
     {\bf Step 2} .  Let  $\gamma \in [1,\overline{r}]$ with condition \eqref{cc9}.  Put $\kappa = \frac{r}{\overline{r}}$, $s_{m}= \kappa^{m}\frac{\gamma}{\overline{r}}>\frac{2}{3}-\delta$ and $\overline{s}_{m}= \kappa^{m}\gamma$ for every non-negative  integer $m$. Since $\gamma \le \overline{r}$,  $s_{0} \le 1$. Put $m_{\gamma} = \max\{m : s_{m}\le 1\}$. Let $m\le m_{\gamma}$, $F_{s_{m}}$, $G_{s_{m}}$ and $c_{0}$ be defined as in  Definition \ref{defs1} and Lemma \ref{lemfs1} for every $m\le m_{\gamma}$. Since $F_{s_{m}}(0)=G_{s_{m}}(0)=0$, by Lemma \ref{2w1ba} and Lemma \ref{w1bb}, $w\equiv u^{+}$,  $v_{m}\equiv F_{m}\circ w$ and $\varphi_{m}\equiv G_{m}\circ w$ are in $W_{B,A}(\Omega)$.  We have 
 \begin{equation}
 \frac{\partial v_{m}}{\partial x_{j}}(x)=  F_{m}'(w)\frac{\partial u}{\partial x_{j}},
\label{2le462aazz2}\end{equation}
\begin{equation} \frac{\partial \varphi_{m}}{\partial x_{j}}(x)= G_{m}'(w)\frac{\partial u}{\partial x_{j}}. 
\label{2le462aazz3}\end{equation}\hk
 Since $s_{m}>\frac{2}{3}-\delta$, we can using Lemmas \ref{lemfs1}-\ref{lemfs1c}. Using $\varphi_{m}$ as a test function  in \eqref{le460},   we have
   \begin{equation}
\int_{\Omega}|\nabla v_{m}|^{2} bdx \leb{\eqref{c1}} \int_{\Omega}\frac{\partial v_{m}}{\partial x_{i}} \frac{\partial v_{m}}{\partial x_{j}}b^{ij}dx\eqb{\eqref{2le462aazz2}} \int_{\Omega}\frac{\partial u}{\partial x_{i}} \frac{\partial u}{\partial x_{j}}|F_{m}'(w)|^{2}b^{ij}dx 
\label{2le462aazz2b}\end{equation} 
  \[\leb{\eqref{sss2}} \hspace{-.12in}c_{0}\hspace{-.06in}
  \int_{\Omega}\hspace{-.02in}\frac{\partial u}{\partial x_{i}} \frac{\partial u}{\partial x_{j}}G_{m}'(w) b^{ij}dx\hspace{-.03in} \eqb{\eqref{2le462aazz3}} \hspace{-.01in}c_{0} \hspace{-.02in}\int_{\Omega}\hspace{-.01in}\frac{\partial u}{\partial x_{i}}\frac{\partial \varphi_{m}}{\partial x_{j}} b^{ij}dx \] 
    \[\leb{\eqref{le460}} c_{0}\int_{\Omega}|\nabla u||\varphi_{m}| a_{0} dx+c_{0}\int_{\Omega}(|u|+1) |\varphi_{m}| a dx \]  
  \[=  c_{0}\hspace{-.051in}\int_{\Omega} \hspace{-.051in} |\nabla u||G_{m}(u)|a_{0} dx+ c_{0}\hspace{-.051in}\int_{\Omega}|u||G_{m}(u)| a dx +  c_{0}\int_{\Omega}|G_{m}(u)| a dx\]
 \[ \leb{\eqref{fs4}, \eqref{sss3}} \hspace{-.1in}c_{0}\hspace{-.051in}\int_{\Omega}\hspace{-.05in}F_{m}(u)|\nabla u||F_{m}'(u)| a_{0} dx+c_{0}\int_{\Omega}|u|F_{m}(u)|F_{m}'(u)| a dx~
   \]
  \[+c_{0}k_{s_{m}}\int_{|u|\le 1}F_{m}^{2-
  \frac{1}{s_{m}}} a dx+c_{0}\int_{|u|> 1}|u|F_{m}(u)|F_{m}'(u)| a dx\]
  \[ \leb{\eqref{cc8},\eqref{c},\eqref{sss1}, H\ddot{o}lder} \hspace{-.02in} c_{0}\hspace{-.01in}\int_{\Omega}\hspace{-.051in}F_{m}(u)|\nabla u||F_{m}'(u)|  b^{\frac{1}{2}}a^{\frac{1}{2}} dx\]
  \[+ c_{0}k_{s_{m}}|\Omega|^{\frac{1}{\overline{r}s_{m}}}||a||_{L^{\frac{\overline{r}}{\overline{r}-2}}(\Omega)}\{\int_{|u|\le 1}\hspace{-.1in}F_{m}(u)^{\overline{r}}  dx\}^{\frac{2-\frac{1}{s_{m}}}{\overline{r}}}+5c_{0}\hspace{-.051in} \int_{\Omega}\hspace{-.051in}F_{m}(u)^{2} a dx\]
       \[ \leb{Young} \frac{1}{2} \int_{\Omega}|\nabla u|^{2}F_{m}'(u)^{2}bdx+\frac{1}{2}c_{0}^{2} \int_{\Omega}F_{m}(u)^{2}adx\]
    \[+c_{0}k_{s_{m}}|\Omega|^{\frac{1}{\overline{r}s_{m}}}||a||_{L^{\frac{\overline{r}}{\overline{r}-2}}(\Omega)}\{\int_{|u|\le 1}\hspace{-.1in}F_{m}(u)^{\overline{r}}  dx\}^{\frac{2-\frac{1}{s_{m}}}{\overline{r}}} + 5c_{0}||a||_{L^{\frac{\overline{r}}{\overline{r}-2}}(\Omega)}\{\int_{\Omega}F_{m}(w)^{\overline{r}} dx\}^{\frac{2}{\overline{r}}}\]
  \[ \leb{\eqref{2le462aazz2}} \frac{1}{2} \int_{\Omega}|\nabla v_{s_{m}}|^{2}bdx+c_{0}(c_{0}+5)||a||_{L^{\frac{\overline{r}}{\overline{r}-2}}(\Omega)}\{\int_{\Omega}v_{s_{m},l}^{\overline{r}}dx\}^{\frac{2}{\overline{r}}}\]
  \[+ c_{0}k_{s_{m}}|\Omega|^{\frac{1}{\overline{r}s_{m}}}||a||_{L^{\frac{\overline{r}}{\overline{r}-2}}(\Omega)}\{\int_{|u|\le 1}\hspace{-.1in}F_{m}(u)^{\overline{r}}  dx\}^{\frac{2-\frac{1}{s_{m}}}{\overline{r}}},\] 
\[\le\hspace{-.03in}\frac{1}{2}\hspace{-.03in} \int_{\Omega}\hspace{-.07in}|\nabla v_{s_{m}}|^{2}bdx + [c_{0}(c_{0}+5)+c_{0}k_{s_{m}}|\Omega|^{\frac{1}{\overline{r}s_{m}}}||a||_{L^{\frac{\overline{r}}{\overline{r}-2}}(\Omega)}]\max\{1,\hspace{-.01in}(\int_{\Omega}\hspace{-.04in}v_{s_{m},l}^{\overline{r}}dx)^{\frac{2}{\overline{r}}}\}
\]  
  or
  \[
\int_{\Omega}|\nabla v_{s_{m}}|^{2} bdx \le  2[c_{0}(c_{0}+5)+c_{0}k_{s_{m}}|\Omega|^{\frac{1}{\overline{r}s_{m}}}||a||_{L^{\frac{\overline{r}}{\overline{r}-2}}(\Omega)}]\max\{1,\hspace{-.01in}(\int_{\Omega}\hspace{-.04in}v_{s_{m},l}^{\overline{r}}dx)^{\frac{2}{\overline{r}}}\}.
  \]\hk
 By \eqref{cb}, we have
\[   \{\hspace{-.04in}\int_{\Omega}\hspace{-.05in}v_{s_{m}}^{r} dx\}^{\frac{2}{r}}\le 2C(r,\Omega,A)^{2}c_{0}[(c_{0}+5)+k_{s_{m}}|\Omega|^{\frac{1}{\overline{r}s_{m}}}||a||_{L^{\frac{\overline{r}}{\overline{r}-2}}(\Omega)}]\max\{1,(\int_{\Omega}v_{s_{m},l}^{\overline{r}}dx)^{\frac{2}{\overline{r}}}\}. \] \hk 
Put $M =\max\{k_{s_{0}},\cdots,k_{s_{m_{\gamma}}}\}$ and 
\[c_{\gamma}= 2C(r,\Omega,A)^{2}c_{0}[(c_{0}+5)+M(1+|\Omega|)||a||_{L^{\frac{\overline{r}}{\overline{r}-2}}(\Omega)}].\]
\hk  Since $s_{m}r= (\frac{r}{2})^{m}\frac{\gamma}{2}r=\kappa^{m+1}\gamma$ and $2s_{m}= 2(\frac{r}{2})^{m}\frac{\gamma}{2}=\kappa^{m}\gamma$, we get  
\begin{equation}  \{\int_{\Omega}(v_{s_{m}})^{\frac{1}{s_{m}}})^{\kappa^{m+1}\gamma} dx\}^{\frac{1}{\kappa^{m+1}\gamma}}\le c_{\gamma}^{\frac{1}{\kappa^{m}\gamma}}\max\{1,\int_{\Omega}((v_{s_{m}})^{\frac{1}{s_{m}}})^{\kappa^{m}\gamma}dx\}^{\frac{1}{\kappa^{m}\gamma}} 
\label{ff}
\end{equation} 
\[\le c_{\gamma}^{\frac{1}{\kappa^{m}\gamma}}\{\int_{\Omega}((v_{s_{m}})^{\frac{1}{s_{m}}})^{\kappa^{m}\gamma}dx\}^{\frac{1}{\kappa^{m}\gamma}}+c_{\gamma}^{\frac{1}{\kappa^{m}\gamma}}. \]\hk 
Since $v_{s_{m}}^{\frac{1}{s_{m}}}= F_{s_{m}}(w)^{\frac{1}{s_{m}}}$, we get
\[   \{\int_{\Omega}w
^{\kappa^{m+1}\gamma} dx\}^{\frac{1}{\kappa^{m+1}\gamma}}\leb{\eqref{fs1z}}\{\int_{\Omega}[\overline{F}(w)+1]
^{\kappa^{m+1}\gamma} dx\}^{\frac{1}{\kappa^{m+1}\gamma}}
\]
\[ \leb{Minkowski} \{\int_{\Omega}\overline{F}(w)
^{\kappa^{m+1}\gamma} dx\}^{\frac{1}{\kappa^{m+1}\gamma}}+ |\Omega|^{\frac{1}{\kappa^{m+1}\gamma}}\]
\[\leb{\eqref{sss4}} \{\int_{\Omega}(v_{s_{m}})^{\frac{1}{s_{m}}})^{\kappa^{m+1}\gamma} dx\}^{\frac{1}{\kappa^{m+1}\gamma}}  + |\Omega|^{\frac{1}{\kappa^{m+1}\gamma}}\]
\[\leb{\eqref{ff}} c_{\gamma}^{\frac{1}{\kappa^{m}\gamma}}\{\int_{\Omega}((v_{s_{m}})^{\frac{1}{s_{m}}})^{\kappa^{m}\gamma}dx\}^{\frac{1}{\kappa^{m}\gamma}}+c_{\gamma}^{\frac{1}{\kappa^{m}\gamma}}+ |\Omega|^{\frac{1}{\kappa^{m+1}\gamma}}\]
\[\leb{\eqref{sss4}} \hspace{-.05in} c_{\gamma}^{\frac{1}{\kappa^{m}\gamma}}\{\int_{\Omega}(w+k_{0})^{\kappa^{m}\gamma}dx\}^{\frac{1}{\kappa^{m}\gamma}}+ c_{\gamma}^{\frac{1}{\kappa^{m}\gamma}}+ |\Omega|^{\frac{1}{\kappa^{m+1}\gamma}} \]
\[ \leb{Minkowski} c_{\gamma}^{\frac{1}{\kappa^{m}\gamma}}\{\int_{\Omega}w^{\kappa^{m}\gamma}dx\}^{\frac{1}{\kappa^{m}\gamma}}+ c_{\gamma}^{\frac{1}{\kappa^{m}\gamma}}(k_{0}|\Omega|^{\frac{1}{\kappa^{m}\gamma}} +1)+ |\Omega|^{\frac{1}{\kappa^{m+1}\gamma}}, 
\]
\[\leb{\kappa\ge 1} k_{1}\{\int_{\Omega}w^{\kappa^{m}\gamma}dx\}^{\frac{1}{\kappa^{m}\gamma}}+k_{1}\hk\forall~m\le m_{\gamma}, \] 
where $k_{1}=(c_{\gamma}+ (k_{0}+1)(1+|\Omega|)+2)^{\frac{2}{\gamma}}$.\\\hk
 By the mathematical induction, we obtain
 \[   \{\int_{\Omega}w
^{\kappa^{m_{\gamma}+1}\gamma} dx\}^{\frac{1}{\kappa^{m+2}\gamma}}\le k_{1}^{m_{\gamma}+1}\{\int_{\Omega}w^{\gamma}dx\}^{\frac{1}{\gamma}}+ \sum_{i=1}^{m_{\gamma}+1}k_{1}^{i}\]\hk
 Replacing $\gamma$ by $\kappa^{m_{\gamma}+1}\gamma$ and arguing as in the proof of step 1, we have
\[|w||_{L^{\infty}(\Omega)}\le c(r,q,\gamma,\Omega,||w||_{L^{\kappa^{m_{\gamma}+1}\gamma}(\Omega)})\]
\[= c(r,q,\gamma,\Omega, k_{1}^{m_{\gamma}+1}\{\int_{\Omega}w^{\gamma}dx\}^{\frac{1}{\gamma}}+ \sum_{i=1}^{m_{\gamma}+1}k_{1}^{i})\]
Arguing as in the proof of step 1, we get the similar result for $w=u^{-}$. Therefore we get the proposition.
 \end{proof}
 \begin{remark}
 Let $N\ge 3$ and $\epsilon\in (0,1)$. Assume $(a_{0}^{2}+a_{1}+ a_{2})\in L^{\frac{N}{2-2\epsilon}}(\Omega)$ and $b\in L^{\frac{N}{\epsilon}}(\Omega)$. We have
 \[ \frac{\overline{r}}{\overline{r}-2}= 1+\frac{2}{\overline{r}-2}\mathop  \downarrow \limits_{\overline{r}\to\frac{2N}{N-2}}=\frac{N}{2},  \]
 \[\frac{2-2\epsilon}{N}+\frac{\epsilon}{N}=\frac{2-\epsilon}{N}. \]\hk
 Thus, by H\"{o}lder's inequality, $(b a_{0}^{2}+a_{1}+ a_{2})\in L^{\frac{N}{2-\epsilon}}(\Omega)\subset L^{\frac{\overline{r}}{\overline{r}-2}}(\Omega)$ if  $\overline{r}$ is sufficiently close to $\frac{2N}{N-2}$. Therefore we can apply Proposition \ref{pro46} to this case. 
 \end{remark} 
\begin{remark}
 If $b$ and $\overline{b}$ are positive constants,   the global boundedness of $u$ have been proved in \cite{GT,LU} for  $A=\partial\Omega$. The solution $u$ in our result may vanish only on a part $A$ of $\partial\Omega$. \\\hk
 If $b$ is a positive constant and $\overline{b}$ is in Morrey space $L^{p,\lambda}(\Omega)$ with $p > 2$, $\lambda \in (0, N)$ and $p + \lambda > N$,  Byun,  Palagachev and Shin obtained the global boundednes  of solutions for a class of elliptic equations  in \cite{BPS}. \\\hk
  If $b$ is constant and $\overline{b}\in L^{p}(\Omega)$ with $p>N$, Duc and Eells obtained the local boundedness for solutions of a class of elliptic equations in \cite{DE}.
 If $b^{-1}\overline{b}\in L^{\infty}(\Omega)$, Trudinger proved the global boundedness  in \cite{TR}.  In our results, the equation may be not strictly nor uniformly elliptic.
 \label{53zz}
\end{remark}\hk
  If the supports of $a_{0}$, $a_{1}$ and $a_{2}$ are small, we do not need  any condition on them in order obtain the boundedness of solutions to elliptic equations outside their supports. 
  \begin{proposition}  Let  $\delta$ be as in Lemma \ref{lemfs1b},  $\gamma \in (\overline{r}(\frac{2}{3}-2\delta),\infty)$,   $A$ be  admissible with respect to $\Omega$,  $a_{0}$, $a_{1}$ and $a_{2}$ be non-negative  measurable functions on $\Omega$ and $u\in W_{B,A}(\Omega)$ and $R\in (0,\infty)$. Assume \eqref{c1}, \eqref{c3} and \eqref{c6} hold.  Let  $s\in (0,\infty)$ and $\{z_{0}, \cdots,z_{k}\}$ are  in $\Omega$. Put  $B_{s}= \Omega\cap(\cup_{i=1}^{k} B'(z_{i}, s))$ and $\Omega_{s}= \Omega\setminus B_{s}$. Assume  $u\in L^{\gamma}(B_{R}\setminus B_{\frac{1}{2}R})$, $a_{0}(x)=a_{1}(x)=a_{2}(x)=0$ for every $x$ in $\Omega_{\frac{1}{2} R}$    and
 \begin{equation}\int_{\Omega} \frac{\partial u}{\partial x_{j}}\frac{\partial \varphi}{\partial x_{j}}b^{ij}dx\le\hspace{-.05in} \int_{\Omega}\hspace{-.05in} [a_{0}|\nabla u|^{p_{0}} +a_{1}|u|^{p_{1}-1}u  +a_{2}]|\varphi| dx, 
 \label{le4460b}
 \end{equation}
for every $\varphi\in W_{B,A}(\Omega)$, where the integral in right hand side may be $\infty$. Then there exists  $C\equiv C(r,\overline{r},q,b,\overline{b},\Omega)$   such that  
 \begin{equation}||u||_{L^{\infty}(\Omega_{R})}\le C||u||_{L^{\gamma}(B_{R}\setminus B_{\frac{1}{2}R})}R^{-\frac{r\overline{r}}{\gamma(r-\overline{r})}}\hk\forall~\gamma \in (\overline{r},\infty),\label{le4460bz}
 \end{equation}
 \begin{equation}||u||_{L^{\infty}(\Omega_{R})}\le C[||u||_{L^{\gamma}(B_{R}\setminus B_{\frac{R}{2}})}+1]R^{-\frac{r\overline{r}}{\gamma(r-\overline{r})}}\hk\forall~\gamma \in (\overline{r}(\frac{2}{3}-2\delta),\overline{r}],\label{le4460bzb}
 \end{equation}
 \begin{equation}||\nabla u||_{L^{2}_{b}(\Omega_{R})}\le C||u||_{L^{\gamma}(B_{R}\setminus B_{\frac{R}{2}})}R^{-\frac{r\overline{r}}{\gamma(r-\overline{r})}-1}\hk\forall~\gamma \in (\overline{r}(\frac{2}{3}-2\delta),\overline{r}],\label{le4460bzzz}
 \end{equation}
    \begin{equation}||\nabla u||_{L^{2}_{b}(\Omega_{R})}\le C[||u||_{L^{\gamma}(B_{R}\setminus B_{\frac{R}{2}})}+1]R^{-\frac{r\overline{r}}{\gamma(r-\overline{r})}-1}\hk\forall~\gamma \in (\overline{r}(\frac{2}{3}-2\delta),\overline{r}].\label{le4460bzzzb}
 \end{equation}
    \label{pro46b}\end{proposition}
   \begin{proof}  Note that
 \begin{equation}
 |\xi_{i}\zeta_{j}b^{ij}| \hspace{-.1in} \leb{Cauchy-Schwarz} \hspace{-.1in}(\xi_{i}\xi_{j}b^{ij})^{\frac{1}{2}} (\zeta_{i}\zeta_{j}b^{ij})^{\frac{1}{2}}\leb{\eqref{c1}}\overline{b}|\xi||\zeta|
 \label{cs}
 \end{equation}
\[ \hspace*{2in} \forall~\xi=(\xi_{1},\cdots,\xi_{N}), \zeta=(\zeta_{1},\cdots,\zeta_{N})\in \RR^{N}.\]\hk
   Put $R_{m}=2^{-1}R\sum_{j=0}^{m}2^{-j}$  for every non-negative integer $m$. Then $\frac{1}{2}R \le R_{m} \uparrow R$ and $\Omega_{R}\subset \Omega_{R_{m}}$ for every $m$ in $\NN$. Let $i \in \{1,\cdots,k\}$,  $\eta_{m,i}$ be a function in $C^{1}(\RR^{N})$ such that $\eta_{m,i}(\RR^{N})\subset [0,1]$, $\eta_{m,i}(B(z_{i},R_{m})=\{0\}$, $\eta_{m,i}(\RR^{N}\setminus B(z_{i},R_{m+1})=\{1\}$ and $|\nabla \eta_{m,i}|\le 4.2^{m+1}R^{-1}$. Put
\begin{equation}\eta_{m}(x)=\Pi_{i=1}^{k}\eta_{m,i}(x)\hh\forall~x\in \RR^{N},m\in\NN. \label{et0}
\end{equation}\hk
Since $\Omega\setminus (\Omega_{R_{m+1}}\cup B_{R_{m}})= B_{R_{m+1}}\setminus B_{R_{m}}$ and $\nabla\eta_{m}(\Omega_{R_{m+1}}\cup B_{R_{m}})=\{0\}$,    we have
 \begin{equation}\eta_{m}(\Omega)\subset [0,1],~ \eta_{m}(\Omega_{R_{m+1}}) = \{1\},~\eta_{m}( B_{R_{m}})=\{0\},
 \label{et2}
\end{equation}
 \begin{equation}||\nabla \eta_{m}||_{L^{\infty}(\Omega)}\le  2^{m+3}N.R^{-1},~\nabla \eta_{m}(\Omega\setminus (B_{R_{m+1}}\setminus B_{R_{m}}))\subset \{0\},
 \label{et2b}
\end{equation}
\begin{equation}\eta_{m}(x).a_{k}(x)=0\hh\forall~x\in\Omega, k=0,1,2.
 \label{le463000b}
\end{equation}
 \begin{equation}|\nabla (\eta_{m}^{2}w)|^{2}b\le |2\eta_{m} w\nabla\eta_{m}+\eta_{m}^{2}\nabla w|^{2}b\le 2|2\eta_{m} w\nabla \eta_{m}|^{2}b+2|\eta_{m}^{4}|\nabla w|^{2}b
 \label{le463000}
\end{equation}
\[ = 8|\eta_{m}|^{2}|w|^{2q}|\nabla\eta_{m}|^{2}b+2|\eta_{m}|^{4}|\nabla w|^{2}b\hh\forall~ w\in W_{B,A}(\Omega).\]\hk
By Lemma \ref{w1bb},  $\eta_{m}^{2}u$ is in $W_{B,A}(\Omega)$. 
 Put 
 \begin{equation}
 \kappa=\frac{r}{\overline{r}}~~~ and ~~s_{m}=\kappa^{m}\frac{\gamma}{\overline{r}}\hh \forall~ m \in \{0,1,2,3,\cdots\}.
 \label{sm}
 \end{equation}\hk
  We prove this proposition by two steps.\\\hk
 {\bf Step 1} . Assume $\gamma > \overline{r
 }$.  
 Since $s_{m}>1$, we can use Lemma \ref{lem62}. Let $l\in (3,\infty)$,  $F_{l}=F_{s_{m},l}$ and $G_{l}=G_{s_{m},l}$ be defined as in Definition  \ref{def61}. Put
  \begin{equation}v_{s_{m},l}= F_{l}\circ u,
 \label{le463000c}
\end{equation}
  \begin{equation}\varphi= \eta_{m}^{4}G_{l}\circ u,
 \label{le463000d}
\end{equation}\hk
 We have
  \begin{equation}
 \frac{\partial v_{m}}{\partial x_{j}}(x)=  F_{l}'(w)\frac{\partial u}{\partial x_{j}},
\label{4zz2}\end{equation}
\begin{equation} \frac{\partial \varphi_{m}}{\partial x_{j}}(x)= \eta_{m}^{4}G_{l}'(w)\frac{\partial u}{\partial x_{j}}+4\eta_{m}^{3}G_{l}(w)\frac{\partial \eta}{\partial x_{j}}. 
\label{4zz3}\end{equation}\hk
  By Lemmas \ref{2w1ba} and  \ref{w1bb},  $v_{s_{m},l}$ and $\varphi$ are in $W_{B,A}(\Omega)$.  Using $\varphi$ as a test function  in \eqref{le4460b},   we have
    \begin{equation}0
    \eqb{\eqref{le463000b}} \int_{\Omega}\frac{\partial u}{\partial x_{i}} \frac{\partial \varphi}{\partial x_{j}}b^{ij}dx\eqb{\eqref{le463000d}}  \int_{\Omega}\eta_{m}^{4} \frac{\partial u}{\partial x_{i}} \frac{\partial u}{\partial x_{j}}G_{l}'(u)b^{ij}dx
  \label{z001}
\end{equation}
\[+4\int_{\Omega}\eta_{m}^{3} \frac{\partial u}{\partial x_{i}} G_{l}(u)\frac{\partial \eta_{m}}{\partial x_{j}} b^{ij}dx.
  \]\hk
  It implies  
 \begin{equation}
  \hspace*{-1.35in}\int_{\Omega}\eta_{m}^{4} \frac{\partial u}{\partial x_{i}} \frac{\partial u}{\partial x_{j}}G_{l}'(u)b^{ij}dx \le 4~ |\int_{\Omega} \eta_{m}^{3} \frac{\partial u}{\partial x_{i}} G_{l}(u)\frac{\partial \eta_{m}}{\partial x_{j}} b^{ij}dx|. 
  \label{z01}
\end{equation}\hk
By computations, we get
 \begin{equation} 8\int_{\Omega}\eta_{m}^{2}|v_{s_{m},l}|^{2}|\nabla\eta_{m}|^{2} bdx \leb{\eqref{et2b}}8(N.2^{m+3}R^{-1})^{2}\int_{B_{R_{m+1}}\setminus B_{R_{m}}} \hspace{-.5in} \eta_{m}^{2}|v_{s_{m},l}|^{2} bdx
 \label{z03c}
 \end{equation}
\[\leb{H\ddot{o}lder}\hspace{-.02in}(N2^{m+5}R^{-1})^{2} ||b||_{L^{\frac{\overline{r}}{\overline{r}-2}}(\Omega)} \{\int_{B_{R_{m+1}}\setminus B_{R_{m}}}\hspace{-.55in}\eta_{m}^{\overline{r}}|v_{s_{m},l}|^{\overline{r}} dx\}^{\frac{2}{\overline{r}}},
\]
  \begin{equation}\int_{\Omega}\eta_{m}^{4}|\nabla v_{s_{m},l}|^{4}bdx
 \leb{\eqref{c1}} \int_{\Omega}\eta_{m}^{4}\frac{\partial v_{s_{m},l}}{\partial x_{i}}\frac{\partial v_{s_{m},l}}{\partial x_{j}}b^{ij}dx
 \label{z02}
\end{equation}
 \[ \eqb{\eqref{le463000c}} \int_{\Omega}\eta_{m}^{4}\frac{\partial u}{\partial x_{i}}\frac{\partial u}{\partial x_{j}}|F_{l}'(u)|^{2}b^{ij}dx\leb{\eqref{ss2}}
 s_{m}^{2}\int_{\Omega}\eta_{m}^{4}\frac{\partial u}{\partial x_{i}}\frac{\partial u}{\partial x_{j}}G_{l}'(u)b^{ij}dx
 \]
   \[ \leb{\eqref{z01}}\hspace{-.09in} 4s_{m}^{2}|\int_{\Omega}\eta_{m}^{3} \frac{\partial u}{\partial x_{i}} G_{l}(u)\frac{\partial \eta_{m}}{\partial x_{j}} b^{ij}dx|\hspace{-.19in} \leb{\eqref{cs},\eqref{et2b}}\hspace{-.19in} 4s_{m}^{2}\int_{B_{R_{m+1}}\setminus B_{R_{m}}}\hspace{-.19in}\eta_{m}^{3}|\nabla u|  G_{l}(u)|\nabla \eta_{m}| \overline{b}dx
\]
\[ \leb{\eqref{f1b2}} 4s_{m}^{2}\int_{B_{R_{m+1}}\setminus B_{R_{m}}}\eta_{m}^{3}|\nabla u|  |\nabla \eta_{m}||F_{l}'(u)||F_{l}(u)|  \overline{b}dx\]
\[ \eqb{\eqref{le463000c}} 4s_{m}^{2}\int_{B_{R_{m+1}}\setminus B_{R_{m}}}\eta_{m}^{3}|\nabla\eta_{m}||\nabla v_{s_{m},l}|v_{s_{m},l}  \overline{b}dx \]
\[ \leb{\eqref{et2b}}N2^{m+5}R^{-1}s_{m}^{2} \int_{B_{R_{m+1}}\setminus B_{R_{m}}}\eta_{m}^{2}|\nabla v_{s_{m},l}|b^{\frac{1}{2}}\eta_{m}v_{s_{m},l}b^{-\frac{1}{2}}\overline{b}dx
\]
\[
\leb{Young}\hspace{-.03in}\frac{1}{2}
\hspace{-.03in} \int_{\Omega}\hspace{-.03in} \eta_{m}^{4}|\nabla v_{s_{m},l}|^{2}bdx + \frac{(N2^{m+5}R^{-1}s_{m}^{2} )^{2}}{2} \hspace{-.03in} \int_{B_{R_{m+1}}\setminus B_{R_{m}}} \hspace{-.4in}\eta_{m}^{2} v_{s_{m},l}^{2}(b^{-\frac{1}{2}}\overline{b})^{2}dx \] 
\[\leb{H\ddot{o}lder} \frac{1}{2} 
\int_{\Omega}\eta_{m}^{4}|\nabla v_{s_{m},l}|^{2}bdx+\frac{(N2^{m+5}R^{-1}s_{m}^{2} )^{2}}{2} ||b^{-\frac{1}{2}}\overline{b}||_{L^{\frac{2\overline{r}}{\overline{r}-2}}(\Omega)} \{\int_{B_{R_{m+1}}\setminus B_{R_{m}}}\hspace{-.45in}\eta_{m}^{\overline{r}} v_{s_{m},l}^{\overline{r}}dx\}^{\frac{2}{\overline{r}}}\]
or
\begin{equation} \int_{\Omega}\eta_{m}^{4}|\nabla v_{s_{m},l}|^{2}bdx \le (N2^{m+5}R^{-1}s_{m}^{2} )^{2}||b^{-\frac{1}{2}}\overline{b}||_{L^{\frac{2\overline{r}}{\overline{r}-2}}(\Omega)} \{\int_{B_{R_{m+1}}\setminus B_{R_{m}}}\hspace{-.35in} \eta_{m}^{\overline{r}} v_{s_{m},l}^{\overline{r}}dx\}^{\frac{2}{\overline{r}}}.
 \label{z03}
\end{equation}\hk
 Thus
 \[ 
\{\int_{\Omega}\hspace{-.02in}\eta_{m}^{2r}|v_{s_{m},l}|^{r} dx\}^{\frac{1}{r}}
 \leb{\eqref{cb}}C(r,\Omega,A) \{\int_{\Omega}|\nabla (\eta_{m}^{2}v_{s_{m},l})|^{2} bdx  \}^{\frac{1}{2}}\]
\[ \leb{\eqref{le463000},\eqref{z03c},\eqref{z03}} \hspace{-.4in}d2^{m}R^{-1}s_{m}^{2}\{\int_{B_{R_{m+1}}\setminus B_{R_{m}}}\hspace{-.6in}\eta_{m}^{\overline{r}}|v_{s_{m},l}|^{\overline{r}} dx\}^{\frac{1}{\overline{r}}},\]
where 
\[d =  C(r,\Omega,A)[N2^{5}||b||_{L^{\frac{\overline{r}}{\overline{r}-2}}(\Omega)}^{\frac{1}{2}}+ 2N2^{5}||b^{-\frac{1}{2}}\overline{b}||_{L^{\frac{2\overline{r}}{\overline{r}-2}}(\Omega)}^{\frac{1}{2}}  ].\]\hk
 It implies
\begin{equation}
\{\int_{\Omega_{R_{m+1}}}\hspace{-.32in}|v_{s_{m},l}|^{r} dx\}^{\frac{1}{s_{m}r}} \le d^{\frac{1}{s_{m}}}2^{\frac{m}{s_{m}}}R^{-\frac{1}{s_{m}}}s_{m}^{\frac{2}{s_{m}}}\{\hspace{-.04in}\int_{B_{R_{m+1}}\setminus B_{R_{m}}}\hspace{-.52in}|v_{s_{m},l}|^{\overline{r}} dx\}^{\frac{1}{s_{m}\overline{r}}}
\label{xav}
\end{equation}
\[\le d^{\frac{1}{s_{m}}}2^{\frac{m}{s_{m}}}R^{-\frac{1}{s_{m}}}s_{m}^{\frac{2}{s_{m}}}\{\hspace{-.04in}\int_{\Omega_{R_{m}}}\hspace{-.22in}|v_{s_{m},l}|^{\overline{r}} dx\}^{\frac{1}{s_{m}\overline{r}}}
\]\hk
By \eqref{f1b} and \eqref{ss4}, the sequence $\{v_{s_{m},3n}(x)\}$ is  increasing  and converges  to $w^{s_{m}}(x)$ a.e. on $\Omega$. Thus by  Lebesgue's Monotone Convergence Theorem and \eqref{xav}, we have 
\begin{equation}
\{\int_{\Omega_{R_{m+1}}}\hspace{-.32in}w^{s_{m}r} dx\}^{\frac{1}{s_{m}r}} \le d^{\frac{1}{s_{m}}}2^{\frac{m}{s_{m}}}R^{-\frac{1}{s_{m}}}s_{m}^{\frac{2}{s_{m}}}\{\hspace{-.04in}\int_{B_{R_{m+1}}\setminus B_{R_{m}}}\hspace{-.52in}w^{s_{m}\overline{r}} dx\}^{\frac{1}{s_{m}\overline{r}}}
\label{xawb}
\end{equation}
\[\le d^{\frac{1}{s_{m}}}2^{\frac{m}{s_{m}}}R^{-\frac{1}{s_{m}}}s_{m}^{\frac{2}{s_{m}}}\{\hspace{-.04in}\int_{\Omega_{R_{m}}}\hspace{-.22in}w^{s_{m}\overline{r}} dx\}^{\frac{1}{s_{m}\overline{r}}}
\]\hk
Note that $\kappa=\frac{r}{\overline{r}}$, $s_{m}= \kappa^{m}\frac{\gamma}{\overline{r}}$, $s_{m}r= \kappa^{m+1}\gamma$, $s_{m}\overline{r}= \kappa^{m}\gamma$, $\frac{1}{s_{m}}=\frac{\overline{r}}{\gamma}\frac{1}{\kappa^{m}} $, $\frac{m}{s_{m}}=\frac{\overline{r}}{\gamma}\frac{m}{\kappa^{m}} $ and $s_{m}^{\frac{2}{s_{m}}}= [ \kappa^{m}\frac{\gamma}{\overline{r}}]^{\frac{2\gamma}{\overline{r}}\frac{1}{\kappa^{m}}}=[(\frac{\gamma}{\overline{r}})^{\frac{2\gamma}{\overline{r}}}]^{\frac{1}{\kappa^{m}}}[ \kappa^{\frac{2\gamma}{\overline{r}}}]^{\frac{m}{\kappa^{m}}}$. It implies
\begin{equation}
\{\int_{\Omega_{R_{m+1}}}\hspace{-.26in}w^{\kappa^{m+1}\gamma} dx\}^{\frac{1}{\kappa^{m+1}\gamma}}\hspace{-.03in}\le d^{\frac{1}{s_{m}}}2^{\frac{m}{s_{m}}}R^{-\frac{1}{s_{m}}}s_{m}^{\frac{2}{s_{m}}}\{\int_{B_{R_{m+1}}\setminus B_{R_{m}}}\hspace{-.62in}w^{\kappa^{m}\gamma} dx\}^{\frac{1}{\kappa_{m}\gamma}}
\label{za2}
\end{equation}
\[\le\hspace{-.03in} d^{\frac{1}{s_{m}}}2^{\frac{m}{s_{m}}}R^{-\frac{1}{s_{m}}}s_{m}^{\frac{2}{s_{m}}}\{\hspace{-.03in}\int_{\Omega_{R_{m}}}\hspace{-.06in} w^{\kappa^{m}\gamma} dx\}^{\frac{1}{\kappa_{m}\gamma}}\hh\forall~m\ge 0.  
\]\hk
Put $\overline{d}_{1}= d^{\frac{\overline{r}}{\gamma}}(\frac{\gamma}{\overline{r}})^{\frac{2\gamma}{\overline{r}}}+1$ and $\overline{d}_{2}=2^{\frac{\overline{r}}{\gamma} } \kappa^{\frac{2\gamma}{\overline{r}}}+1$. By computations, we have 
\[ d^{\frac{1}{s_{m}}}2^{\frac{m}{s_{m}}}R^{-\frac{1}{s_{m}}}s_{m}^{\frac{2}{s_{m}}}
=d^{\frac{\overline{r}}{\gamma}\frac{1}{\kappa^{m}}}2^{\frac{\overline{r}}{\gamma}\frac{m}{\kappa^{m}} }R^{-\frac{\overline{r}}{\gamma}\frac{1}{\kappa^{m}}}[(\frac{\gamma}{\overline{r}})^{\frac{2\gamma}{\overline{r}}}]^{\frac{1}{\kappa^{m}}}[ \kappa^{\frac{2\gamma}{\overline{r}}}]^{\frac{m}{\kappa^{m}}} 
\]
\[= [d^{\frac{\overline{r}}{\gamma}}(\frac{\gamma}{\overline{r}})^{\frac{2\gamma}{\overline{r}}}]^{\frac{1}{\kappa^{m}}}[2^{\frac{\overline{r}}{\gamma} } \kappa^{\frac{2\gamma}{\overline{r}}}]^{\frac{m}{\kappa^{m}}}[R^{-\frac{\overline{r}}{\gamma}}]^{\frac{1}{\kappa^{m}}}\le \overline{d}_{1}^{\frac{1}{\kappa^{m}}}\overline{d}_{2}^{\frac{m}{\kappa^{m}}}[R^{-\frac{\overline{r}}{\gamma}}]^{\frac{1}{\kappa^{m}}}.
\]\hk
 By \eqref{za2}, we have
 \[\{\int_{\Omega_{R_{m+1}}}\hspace{-.3in}w^{\kappa^{m+1}\hspace{-.03in}\gamma}\hspace{-.04in} dx\}^{\frac{1}{\kappa^{m+1}\gamma}}\hspace{-.03in}\le\hspace{-.03in} \overline{d}_{1}^{\frac{1}{\kappa^{m}}}\overline{d}_{2}^{\frac{m}{\kappa^{m}}}[R^{-\frac{\overline{r}}{\gamma}}]^{\frac{1}{\kappa^{m}}}\{\hspace{-.05in} \int_{B_{R_{m+1}}\setminus B_{R_{m}}}\hspace{-.08in}w^{\frac{1}{\kappa^{m}\gamma}} dx\}^{\frac{1}{\kappa^{m}\gamma}}\]
\begin{equation}
\le\hspace{-.03in} \overline{d}_{1}^{\frac{1}{\kappa^{m}}}\overline{d}_{2}^{\frac{m}{\kappa^{m}}}[R^{-\frac{\overline{r}}{\gamma}}]^{\frac{1}{\kappa^{m}}}\{\hspace{-.05in} \int_{\Omega_{R_{m}}}\hspace{-.08in}w^{\frac{1}{\kappa^{m}\gamma}} dx\}^{\frac{1}{\kappa^{m}\gamma}}\hk\forall~m\ge 0
\label{za2b}
\end{equation}
\begin{equation}
\{\int_{\Omega_{R_{1}}}\hspace{-.15in}w^{\kappa\gamma} dx\}^{\frac{1}{\kappa\gamma}}\le\hspace{-.03in} \overline{d}_{1}\overline{d}_{2}[R^{-\frac{\overline{r}}{\gamma}}] \int_{B_{R_{m}}}\hspace{-.08in}w^{\frac{1}{\gamma}} dx\}^{\frac{1}{\gamma}}
\label{za2c}
\end{equation}\hk
By mathematical induction, it implies
\[
\{\int_{\Omega_{R}}|w|^{\kappa^{m}\gamma} dx\}^{\frac{1}{\kappa^{m}\gamma}}\le\int_{\Omega_{R_{m}}}|w|^{\kappa^{m}\gamma}dx\}^{\frac{1}{\kappa^{m}\gamma}}\le
\]
\[\le\overline{d}_{1}^{\sum_{j=0}^{m-1}\frac{1}{\kappa^{j}}}\overline{d}_{2}^{\sum_{j=1}^{m-1}\frac{j}{\kappa^{j}}}[R^{-\frac{\overline{r}}{\gamma}}]^{\sum_{j=0}^{m-1}\frac{1}{\kappa^{j}}}\{\int_{\Omega_{R_{1}}}\hspace{-.08in}|w|^{\kappa\gamma} dx\}^{\frac{1}{\kappa
\gamma}} \]
\[ \leb{\eqref{za2},\kappa=\frac{r}{\overline{r}},R_{0}=\frac{1}{2}R} \overline{d}_{3}\{ \int_{B_{R}\setminus B_{\frac{1}{2}R}}|w|^{\gamma} dx\}^{\frac{1}{\gamma}}
R^{-\frac{r\overline{r}}{\gamma(r-\overline{r})}}\hh\forall~m \ge 2,
\]
where $\overline{d}_{3}=\overline{d}_{1}^{\sum_{j=0}^{\infty}\frac{1}{\kappa^{j}}}\overline{d}_{2}^{ \sum_{j=0}^{\infty}\frac{j}{\kappa^{j}}}$.\\\hk
 Therefore,  by Problem 5 in \cite[p.71]{RU}, it implies 
 \begin{equation}
||w||_{L^{\infty}(\Omega_{R})}\le  \overline{d}_{3}\{ \int_{B_{R}\setminus B_{\frac{1}{2}R}}|u|^{\gamma} dx\}^{\frac{1}{\gamma}}
R^{-\frac{r\overline{r}}{\gamma(r-\overline{r})}},
\label{zaa1}
\end{equation}
which implies  \eqref{le4460bz} for $u^{+}$.\\\hk
 {\bf Step 2} . Assume $\frac{\gamma}{\overline{r}}\in (\frac{2}{3} -\delta,1]$. Put $\kappa$ and $s_{m}$ as in \eqref{sm}. Since $\gamma \le \overline{r}$,  $s_{0} \le 1$. Put $m_{\gamma} = \max\{m : s_{m}\le 1\}$. Let $F_{s_{m}}$, $G_{s_{m}}$ and $c_{0}$ be defined as in  Definition \ref{defs1} and Lemma \ref{lemfs1} for every $m\le m_{\gamma}$. Since $F_{s_{m}}(0)=G_{s_{m}}(0)=0$, by Lemma \ref{2w1ba} and Lemma \ref{w1bb}, $w\equiv u^{+}$,  $v_{m}\equiv F_{m}\circ w$ and $\varphi_{m}\equiv G_{m}\circ w$ are in $W_{B,A}(\Omega)$.  We have 
 \begin{equation}
 \frac{\partial v_{m}}{\partial x_{j}}(x)=  F_{m}'(w)\frac{\partial u}{\partial x_{j}},
\label{2le462aazz2z}\end{equation}
\begin{equation} \frac{\partial \varphi_{m}}{\partial x_{j}}(x)= G_{m}'(w)\frac{\partial u}{\partial x_{j}}. 
\label{2le462aazz3b}\end{equation}\hk
 Since $s_{m}>\frac{2}{3}-\delta$, we can using Lemmas \ref{lemfs1}-\ref{lemfs1c} for $m\le m_{\gamma}$. Let $m\le m_{\gamma}$.  Using $\varphi_{m}$ as a test function  in \eqref{le4460b}, by computations, we get
 \begin{equation} 8\int_{\Omega}\eta_{m}^{2}|v_{s_{m}}|^{2}|\nabla\eta_{m}|^{2} bdx \leb{\eqref{et2b}}8(N.2^{m+3}R^{-1})^{2}\int_{B_{R_{m+1}}\setminus B_{R_{m}}} \hspace{-.5in} \eta_{m}^{2}|v_{s_{m}}|^{2} bdx
 \label{z03cd}
 \end{equation}
\[\leb{H\ddot{o}lder}\hspace{-.02in}(N2^{m+5}R^{-1})^{2} ||b||_{L^{\frac{\overline{r}}{\overline{r}-2}}(\Omega)} \{\int_{B_{R_{m+1}}\setminus B_{R_{m}}}\hspace{-.55in}\eta_{m}^{\overline{r}}|v_{s_{m}}|^{\overline{r}} dx\}^{\frac{2}{\overline{r}}},
\]
  \begin{equation}\int_{\Omega}\hspace{-.04in}\eta_{m}^{4}|\nabla v_{s_{m}}|^{4}bdx \hspace{-.1in}
 \leb{\eqref{c1}} \int_{\Omega}\hspace{-.04in}\eta_{m}^{4}\frac{\partial v_{s_{m},l}}{\partial x_{i}}\frac{\partial v_{s_{m}}}{\partial x_{j}}b^{ij}dx
 \hspace{-.15in}\eqb{\eqref{le463000c}} \int_{\Omega}\hspace{-.04in}\eta_{m}^{4}\frac{\partial u}{\partial x_{i}}\frac{\partial u}{\partial x_{j}}|F_{l}'(u)|^{2}b^{ij}dx
 \label{z02ab}
\end{equation}
 \[\leb{\eqref{sss2}}
 c_{0}\int_{\Omega}\eta_{m}^{4}\frac{\partial u}{\partial x_{i}}\frac{\partial u}{\partial x_{j}}G_{l}'(u)b^{ij}dx \leb{\eqref{z01}}4 c_{0}|\int_{\Omega}\eta_{m}^{3} \frac{\partial u}{\partial x_{i}} G_{l}(u)\frac{\partial \eta_{m}}{\partial x_{j}} b^{ij}dx|
 \]
\[\leb{\eqref{cs},\eqref{et2b}} 4c_{0}\int_{B_{R_{m+1}}\setminus B_{R_{m}}}\eta_{m}^{3}|\nabla u|  G_{l}(u)|\nabla \eta_{m}| \overline{b}dx\]
\[ \leb{\eqref{fs4}}\hspace{-.08in}4 c_{0}\int_{B_{R_{m+1}}\setminus B_{R_{m}}}\hspace{-.6in}\eta_{m}^{3}|\nabla u|  |\nabla \eta_{m}||F_{l}'(u)||F_{l}(u)|  \overline{b}dx\eqb{\eqref{le463000c}}4 c_{0}\int_{B_{R_{m+1}}\setminus B_{R_{m}}}\hspace{-.4in}\eta_{m}^{3}|\nabla\eta_{m}||\nabla v_{s_{m},l}|v_{s_{m}}  \overline{b}dx \]
\[ \leb{\eqref{et2b}}4N2^{m+3}R^{-1}c_{0} \int_{B_{R_{m+1}}\setminus B_{R_{m}}}\eta_{m}^{2}|\nabla v_{s_{m}}|b^{\frac{1}{2}}\eta_{m}v_{s_{m}}b^{-\frac{1}{2}}\overline{b}dx
\]
\[
\leb{Young}\hspace{-.03in}\frac{1}{2}
\hspace{-.03in} \int_{\Omega}\hspace{-.03in} \eta_{m}^{4}|\nabla v_{s_{m}}|^{2}bdx + \frac{4(N2^{m+3}R^{-1}c_{0}  )^{2}}{2} \hspace{-.03in} \int_{B_{R_{m+1}}\setminus B_{R_{m}}} \hspace{-.4in}\eta_{m}^{2} v_{s_{m}}^{2}(b^{-\frac{1}{2}}\overline{b})^{2}dx \] 
\[\leb{H\ddot{o}lder} \frac{1}{2} 
\int_{\Omega}\eta_{m}^{4}|\nabla v_{s_{m}}|^{2}bdx+2(N2^{m+3}R^{-1}c_{0})^{2}||b^{-\frac{1}{2}}\overline{b}||_{L^{\frac{2\overline{r}}{\overline{r}-2}}(\Omega)} \{\int_{B_{R_{m+1}}\setminus B_{R_{m}}}\hspace{-.5in}\eta_{m}^{\overline{r}} v_{s_{m}}^{\overline{r}}dx\}^{\frac{2}{\overline{r}}}\]
or
\begin{equation} \int_{\Omega}\eta_{m}^{4}|\nabla v_{s_{m}}|^{2}bdx \le 4(N2^{m+3}R^{-1}c_{0})^{2}||b^{-\frac{1}{2}}\overline{b}||_{L^{\frac{2\overline{r}}{\overline{r}-2}}(\Omega)} \{\int_{B_{R_{m+1}}\setminus B_{R_{m}}}\hspace{-.45in} \eta_{m}^{\overline{r}} v_{s_{m}}^{\overline{r}}dx\}^{\frac{2}{\overline{r}}}.
 \label{z03d}
\end{equation}
\hk
  Thus
 \[ 
\{\int_{\Omega}\hspace{-.02in}\eta_{m}^{2r}|v_{s_{m}}|^{r} dx\}^{\frac{1}{r}}\hspace{-.13in}
 \leb{\eqref{cb}}\hspace{-.13in}C(r,\Omega,A) \{\int_{\Omega}|\nabla (\eta_{m}^{2}v_{s_{m}})|^{2} bdx  \}^{\frac{1}{2}}\hspace{-.4in}\leb{\eqref{le463000},\eqref{z03cd},\eqref{z03d}} \hspace{-.4in}\overline{d}2^{m}R^{-1}\{\int_{B_{R_{m+1}}\setminus B_{R_{m}}}\hspace{-.6in}\eta_{m}^{\overline{r}}|v_{s_{m}}|^{\overline{r}} dx\}^{\frac{1}{\overline{r}}},\]
where 
\[\overline{d} =  4C(r,\Omega,A)[(N2^{5}R^{-1})^{2} ||b||_{L^{\frac{\overline{r}}{\overline{r}-2}}(\Omega)} + 2(N2^{3}R^{-1}c_{0})^{2}||b^{-\frac{1}{2}}\overline{b}||_{L^{\frac{2\overline{r}}{\overline{r}-2}}(\Omega)} ].\]\hk
 It implies
 \begin{equation}
\{\int_{\Omega_{R_{m+1}}}\hspace{-.32in}|v_{s_{m}}|^{r} dx\}^{\frac{1}{s_{m}r}} \le \overline{d}^{\frac{1}{s_{m}}}2^{\frac{m}{s_{m}}}R^{-\frac{1}{s_{m}}}\{\hspace{-.04in}\int_{B_{R_{m+1}}\setminus B_{R_{m}}}\hspace{-.52in}|v_{s_{m}}|^{\overline{r}} dx\}^{\frac{1}{s_{m}\overline{r}}}
\label{xavz1}
\end{equation}
\[\leb{B_{R_{m+1}}\subset \Omega} \overline{d}^{\frac{1}{s_{m}}}2^{\frac{m}{s_{m}}}R^{-\frac{1}{s_{m}}}\{\hspace{-.04in}\int_{\Omega_{R_{m}}}\hspace{-.22in}|v_{s_{m}}|^{\overline{r}} dx\}^{\frac{1}{s_{m}\overline{r}}}\le \overline{d}^{\frac{1}{s_{m}}}2^{\frac{m}{s_{m}}}R^{-\frac{1}{s_{m}}}\{\int_{\Omega_{R_{m}}}|v_{s_{m}}|^{\overline{r}} dx\}^{\frac{1}{s_{m}\overline{r}}}
\]\hk
 Thus
  \begin{equation}
\{\int_{\Omega_{R_{m+1}}}\hspace{-.32in}|w|^{s_{m}r} dx\}^{\frac{1}{s_{m}r}}\hspace{-.12in} \leb{\eqref{fs1z}} \int_{\Omega_{R_{m+1}}}\hspace{-.32in}[\overline{F}(w)+1]^{s_{m}r} dx\}^{\frac{1}{s_{m}r}}
\label{xavz2}
\end{equation}
\[\leb{Minkowski}\int_{\Omega_{R_{m+1}}}\hspace{-.32in}\overline{F}(w)^{s_{m}r} dx\}^{\frac{1}{s_{m}r}} +|\Omega|^{\frac{1}{s_{m}r}} 
\]
\[\leb{\eqref{sss4}} \hspace{-.07in}      \{\int_{\Omega_{R_{m+1}}}\hspace{-.32in}|v_{s_{m}}|^{r} dx\}^{\frac{1}{s_{m}r}}+|\Omega|^{\frac{1}{s_{m}r}}\hspace{-.12in} \leb{\eqref{xavz1}} \hspace{-.12in}\overline{d}^{\frac{1}{s_{m}}}2^{\frac{m}{s_{m}}}R^{-\frac{1}{s_{m}}}\{\hspace{-.04in}\int_{B_{R_{m+1}}\setminus B_{R_{m}}}\hspace{-.52in}|v_{s_{m}}|^{\overline{r}} dx\}^{\frac{1}{s_{m}\overline{r}}}+|\Omega|^{\frac{1}{s_{m}r}}\]
\[\leb{\eqref{sss4}} \overline{d}^{\frac{1}{s_{m}}}2^{\frac{m}{s_{m}}}R^{-\frac{1}{s_{m}}}\{\hspace{-.04in}\int_{B_{R_{m+1}}\setminus B_{R_{m}}}\hspace{-.52in}|w+k_{0}|^{s_{m}\overline{r}} dx\}^{\frac{1}{s_{m}\overline{r}}}+|\Omega|^{\frac{1}{s_{m}r}}\]
\[\leb{Minkowski} \overline{d}^{\frac{1}{s_{m}}}2^{\frac{m}{s_{m}}}R^{-\frac{1}{s_{m}}}\{\hspace{-.04in}\int_{B_{R_{m+1}}\setminus B_{R_{m}}}\hspace{-.52in}|w|^{s_{m}\overline{r}} dx\}^{\frac{1}{s_{m}\overline{r}}}+ \overline{d}^{\frac{1}{s_{m}}}2^{\frac{m}{s_{m}}}R^{-\frac{1}{s_{m}}}k_{0}(1+|\Omega|)+|\Omega|^{\frac{1}{s_{m}r}}\]
\[\leb{B_{R_{m+1}}\subset \Omega} \overline{d}^{\frac{1}{s_{m}}}2^{\frac{m}{s_{m}}}R^{-\frac{1}{s_{m}}}\{\hspace{-.04in}\int_{B_{R_{m+1}}\setminus B_{R_{m}}}\hspace{-.22in}|w|^{s_{m}\overline{r}} dx\}^{\frac{1}{s_{m}\overline{r}}}+ \overline{d}^{\frac{1}{s_{m}}}2^{\frac{m}{s_{m}}}R^{-\frac{1}{s_{m}}}k_{0}(1+|\Omega|)+|\Omega|^{\frac{1}{s_{m}r}}.
\]\hk
 Put $k_{2}= max\{\overline{d}^{\frac{1}{s_{m}}}2^{\frac{m}{s_{m}}}R^{-\frac{1}{s_{m}}}k_{0}(1+|\Omega|)+|\Omega|^{\frac{1}{s_{m}r}}: m\le m_{\gamma}\}$. By mathematical induction, we get
\[\{\int_{\Omega_{R_{m_{\gamma}+1}}}\hspace{-.32in}|w|^{\kappa^{(m_{\gamma}+1)\gamma}} dx\}^{\frac{1}{(m_{\gamma}+1)\gamma}}=\{\int_{\Omega_{R_{m_{\gamma}+1}}}\hspace{-.32in}|w|^{s_{m_{\gamma}+1}r} dx\}^{\frac{1}{s_{m}r}}\]
\[ \leb{\eqref{xavz2}} k_{2}^{m_{\gamma}}\{\int_{\Omega_{R_{\kappa\gamma}}}\hspace{-.12in}|w|^{\kappa\gamma} dx\}^{\frac{1}{\kappa\gamma}}+\sum_{l=1}^{2m_{\gamma}-1}k_{2}^{l}\le k_{2}^{m_{\gamma}+1}\{\int_{B_{R}\setminus B_{\frac{1}{2}R}}\hspace{-.12in}|w|^{\gamma} dx\}^{\frac{1}{\gamma}}+\sum_{l=1}^{2m_{\gamma}+1}k_{2}^{l}
\}\]\hk
Replacing $\gamma$ by $\kappa^{m_{\gamma}+1}\gamma$ and arguing as in Step 1, we have
\begin{equation}||w||_{L^{\infty}(\Omega_{R})}\le  \overline{d}_{3}[k_{2}^{m_{\gamma}+1}\{\int_{B_{R}\setminus B_{\frac{1}{2}R}}\hspace{-.12in}|w|^{\gamma} dx\}^{\frac{1}{\gamma}}+\sum_{l=1}^{2m_{\gamma}+1}k_{2}^{l}]
R^{-\frac{r\overline{r}}{\gamma(r-\overline{r})}},
\label{zaa2}
\end{equation}
which implies  \eqref{le4460bzb} for $u^{+}$.\\\hk
 {\bf Step 3} . Replacing $w=u^{+}$ by $w=u^{-}$.  We have 
 \[\frac{\partial w}{\partial x_{j}}(x) = \left\{ {\begin{array}{*{20}{l}}
\displaystyle -\frac{\partial u}{\partial x_{j}}(x) \hh&{\forall~x \in \Omega^{-},}\vspace{.1in}\\
0&{\forall~x \in \Omega\setminus\Omega^{-}}
\end{array}} \right.
\]
 \[\frac{\partial v_{s_{m},l}}{\partial x_{j}}(x) = \left\{ {\begin{array}{*{20}{l}}
\displaystyle - F_{l}'(w(x))\frac{\partial u}{\partial x_{j}}(x) \hh&{\forall~x \in \Omega^{-},}\vspace{.1in}\\
0&{\forall~x \in \Omega\setminus\Omega^{-}}
\end{array}} \right.
\]
\[\eqb{F_{l}'(0)=0}F_{l}'(w(x))\frac{\partial w}{\partial x_{j}}(x)\hh\forall~x \in \Omega,\]
 \[\frac{\partial \varphi}{\partial x_{j}}(x) = \left\{ {\begin{array}{*{20}{l}}
\displaystyle - G_{l}'(w(x))\frac{\partial u}{\partial x_{j}}(x)  \hh&{\forall~x \in \Omega^{-},}\vspace{.1in}\\
0&{\forall~x \in \Omega\setminus\Omega^{-}.}
\end{array}} \right.
\]\hk
We have
\[\int_{\Omega}\eta_{m}^{4}\frac{\partial v_{s_{m},l}}{\partial x_{i}}\frac{\partial v_{s_{m}}}{\partial x_{j}}b^{ij}dx
 = \int_{\Omega}\eta_{m}^{4}\frac{\partial u}{\partial x_{i}}\frac{\partial u}{\partial x_{j}}|F_{l}'(u)|^{2}b^{ij}dx
 \]
 as in \eqref{z02ab}. Therefore we get the results in Step 1 and Step 2 for $w$, then for $u$. \\\hk
  {\bf Step 4} .  As above, we can find a function $\overline{\eta}$ in $C^{1}(\Omega)$ such that 
 \begin{equation}\overline{\eta}(\Omega)\subset [0,1], \overline{\eta}(\Omega_{R}) = \{1\},~\overline{\eta}(B_{\frac{1}{2}R})=\{0\}
 \label{et26}
\end{equation}
 \begin{equation} \nabla \overline{\eta}(\Omega\setminus(B_{R}\setminus B_{\frac{R}{2}}))=\{0\},~
 ||\nabla \overline{\eta}||_{L^{\infty}(\Omega)}\le  8N.R^{-1},
 \label{et2b6}
\end{equation}
\begin{equation}\overline{\eta}(x).a_{k}(x)=0\hh\forall~x\in\Omega, k=0,1,2,
 \label{et2b6b}
\end{equation}\hk
 Put $\phi=\overline{\eta}^{2} u$. By \eqref{cc9} and Lemma \ref{w1bb},   $\phi$ is in $W_{B,A}(\Omega)$.  Using $\phi$ as a test function  in \eqref{le4460b},   we have 
\begin{equation}0\eqb{\eqref{le4460b}}\int_{\Omega} \frac{\partial u}{\partial x_{i}} \frac{\partial \phi}{\partial x_{j}}b^{ij}dx= \int_{\Omega}\overline{\eta}^{2} \frac{\partial u}{\partial x_{i}}\frac{\partial u}{\partial x_{j}}b^{ij}dx +2\int_{\Omega}\overline{\eta}\frac{\partial u}{\partial x_{i}}u \frac{\partial \overline{\eta}}{\partial x_{j}} b^{ij}dx,
\label{zazb}
\end{equation}
 \[\int_{\Omega_{R}}\hspace{-.1in}|\nabla u|^{2} bdx
\le\hspace{-.05in} \int_{\Omega}\overline{\eta}^{2}|\nabla u|^{2} bdx\leb{\eqref{c1}}\int_{\Omega}\overline{\eta}^{2}\frac{\partial u}{\partial x_{i}}\frac{\partial u}{\partial x_{j}} b^{ij}dx\hspace{-.297in}\eqb{\eqref{zazb},~\eqref{et2b6}}\hspace{-.17in}-2\int_{B_{R}\setminus B_{\frac{R}{2}}}\hspace{-.13in}\overline{\eta}\frac{\partial u}{\partial x_{i}}u \frac{\partial \overline{\eta}}{\partial x_{j}} b^{ij}dx\]
\[\leb{Cauchy-Schwarz}\hspace{-.1in}2\int_{B_{R}\setminus B_{\frac{R}{2}}} \hspace{-.04in}\overline{\eta}|u| [\frac{\partial u}{\partial x_{i}} \frac{\partial u}{\partial x_{j}} b^{ij}]^{\frac{1}{2}}[\frac{\partial \overline{\eta}}{\partial x_{i}} \frac{\partial \overline{\eta}}{\partial x_{j}} b^{ij}]^{\frac{1}{2}}dx\]
\[\le 2||\nabla \overline{\eta}||_{L^{\infty}(\Omega)}||u||_{L^{\infty}(\Omega_{R})}\int_{B_{R}}\overline{\eta}|\nabla u| \overline{b}dx\leb{\eqref{le4460bz},\eqref{le4460bzb}} 16NcR^{-\frac{r\overline{r}}{\gamma(r-\overline{r})}-1}\int_{B_{R}}\overline{\eta}|\nabla u| \overline{b}b^{-\frac{1}{2}}b^{\frac{1}{2}}dx\]
\[\leb{H\ddot{o}lder}16N c||u||_{L^{\infty}(\Omega_{R})}R^{-\frac{r\overline{r}}{\gamma(r-\overline{r})}-1}||\overline{b}b^{-\frac{1}{2}}||_{L^{2}(\Omega)}\{\int_{\Omega_{R}}|\nabla u|^{2} bdx\}^{\frac{1}{2}}
\]
or
\[\{\int_{\Omega_{R}}|\nabla u|^{2} bdx\}^{\frac{1}{2}}
\le 16Nc||u||_{L^{\infty}(\Omega_{R})}\overline{b}b^{-\frac{1}{2}}||_{L^{2}(\Omega)}R^{-\frac{r\overline{r}}{\gamma(r-\overline{r})}-1}\]
where $c=C[||u||_{L^{\gamma}(B_{R}\setminus B_{\frac{R}{2}})}+1]$.\\\hk
Thus, by \eqref{le4460bz} and \eqref{le4460bzb}, 
 we get \eqref{le4460bzzz} and \eqref{le4460bzzzb}. 
\end{proof}
\begin{remark} If $u$ is in $L^{\gamma}(\Omega)$, we can replace $\displaystyle \{\int_{B_{R}}\hspace{-.12in}|u|^{\gamma} dx\}^{\frac{1}{\gamma}}$ by $||u||_{L^{\gamma}(\Omega)}$ in \eqref{zaa2} and replace $C(r,\overline{r},q,\gamma,||u||_{L^{\gamma}(B_{R})}\Omega)$ by a constant $C(r,\overline{r},q,\gamma,||u||_{L^{\gamma}(\Omega)},\Omega)$, which is independent from $R$.
\label{rr}
\end{remark}\hk
 We consider  following condition: there is $\nu\in (0,\infty)$ such that
 \begin{equation}
|||w|||_{B}\le\nu \{\int_{\Omega}\sum_{i,j=1}^{N}\frac{\partial w}{\partial x_{i}}\frac{\partial w}{\partial x_{j}}b^{ij}  dx\}^{\frac{1}{2}}\le \nu |||w|||_{B}\hh\forall~w\in W_{B,A}(\Omega).
 \label{c9}
 \end{equation}\hk 
 \begin{remark}~~Assume \eqref{c1} holds. If  $\overline{b}=\lambda b$ with some $\lambda\in (0,\infty)$, then $|||.|||_{b}$ and $|||.|||_{B}$ are equivalent and we obtain \eqref{c9}.\\\hk
Assume  $b^{ij} =b_{i}\delta^{i}_{j}$ for every $i,j \in \{1,\cdots,N\}$ with positive measurable functions $b_{i}$. Arguing as in the proof of the above  lemma, we see 
\[|||u|||_{B}=  \{\int_{\Omega}\sum_{i}^{N}|\frac{\partial u}{\partial x_{i}}|^{2}b_{i}  dx\}^{\frac{1}{2}}\hh\forall~u\in W_{B,A}(\Omega). \]\hk
Thus we get \eqref{c9}.
  \label{equiv}
 \end{remark} \hk
  Hereafter we always assume  \eqref{c9} holds.
\begin{definition} We put \\\hk
  $\displaystyle||v||_{B} = \{\int_{\Omega}\sum_{i,j=1}^{N}\frac{\partial v}{\partial x_{i}}\frac{\partial v}{\partial x_{j}}b^{ij}dz\}^{\frac{1}{2}} \hh\forall~v \in W_{B,A}(\Omega),$\\\hk
    $\displaystyle||v||_{b} = \{\int_{\Omega}|\nabla  v|^{2}bdx\}^{\frac{1}{2}} \hh\forall~v \in W_{b,A}(\Omega).$
    \label{nor}
\end{definition}\hk
Hence these new norms are equivalent to the norms defined in Definition \ref{d44} for  $W_{b,A}(\Omega)$ and   $W_{B,A}(\Omega)$ respectively. Hereafter we will use the norms  for these spaces instead of old norms.
Now we consider a concrete case.
 \begin{proposition} Let  $\alpha\in [1,\frac{r}{2})$, $\zeta \in (0,1)$, $z\in \Omega$, $a\in L^{\infty}(\Omega)$ . Assume    \eqref{c1}-\eqref{c4} and \eqref{c9} hold and  $||a||_{L^{1}(\Omega)} \le M $.   Then\\\hk 
 $(i)$~~ There is a unique solution $v$ in $W_{B,A}(\Omega)$
to the following equation 
 \begin{equation}<v,\varphi>\equiv\int_{\Omega} \frac{\partial v}{\partial x_{i}}\frac{\partial \varphi}{\partial x_{j}} b^{ij}dx= \int_{\Omega}a\varphi dx\hk\forall~\varphi\in W_{B,A}(\Omega).
 \label{281z}
\end{equation}\hk
$(ii)$~~There is a positive real number $k_{\Omega,M,r,B,\nu,\alpha}$  independent of $R$ such that
 \begin{equation}||v||_{L^{\alpha}(\Omega)}\le k_{\Omega,M,r,B,\alpha}\hk\forall~~ \alpha \in [1,\frac{r}{2}). \label{281az}
\end{equation}\hk
$(iii)$~~If $a(x)=0$ for every $x$ in $\Omega\setminus B(z,\frac{1}{2}R)$ with some $R\in(0,\infty)$, then there are $t(\zeta)\in (\frac{2N^{2}+2N-2}{N^{2}+2N-1},2)$, $\overline{r}(\zeta)\in (2,\frac{t(\zeta)(N+1)-2}{N-t(\zeta)})$ and a real number $C(B,\overline{r},M,b,\zeta)$  independent of $R$  such that for every  $t\in(t(\zeta),2), \overline{r}\in (2,\overline{r}(\zeta))$ we have
 \begin{equation} 
 |v(x)| \le  C(B,r,\overline{r},M,b,z,\zeta)|x-z|^{2-N-\zeta} \hk\forall~x\in \Omega\setminus B(z,R), 
 \label{281bzd}
 \end{equation}
  \begin{equation}||\nabla v||_{L^{2}_{b}(\Omega\setminus B(z,R))}\le C(B,r,\overline{r},M,b,\zeta)R^{1-N-\zeta-\nu}.\label{le2460bzzza}
 \end{equation}
 \label{lem62xz}
 \end{proposition}
  \begin{proof} ~$(i)$~~First we assume $a\ge 0$.  Put
\[<u,w>= \int_{\Omega}\frac{\partial u}{\partial x_{i}} \frac{\partial w}{\partial x_{j}}b^{ij}dx\hh\forall~u, w\in W_{B,A}(\Omega),\]
\[T(\phi)=\int_{\Omega}a\varphi dx\hh\forall~\varphi\in W_{B,A}(\Omega).\]\hk
 Thus $<.,.>$ is the inner product in $W_{B,A}(\Omega)$. We have
\[|T(\varphi)|\le ||a||_{L^{\infty}(\Omega)}\int_{\Omega}|\varphi| dx\leb{H\ddot{o}lder} ||a||_{L^{\infty}(\Omega)}|\Omega|^{\frac{r-1}{r}}\{\int_{\Omega}|\varphi|^{r}dx\}^{\frac{1}{r}} 
\]
\[\leb{\eqref{cb}} C(r,\Omega,A)||a||_{L^{\infty}(\Omega)}|\Omega|^{\frac{r-1}{r}}||\varphi||_{B}\hk\forall~\varphi\in W_{B,A}(\Omega),
\]\hk
 Thus  $T$ is a continuous linear mapping on $W_{B,A}(\Omega)$. By Riesz's theorem, there is an unique solution $v$ in $W_{B,A}(\Omega)$ of the following equation
\begin{equation}\int_{\Omega} \frac{\partial v}{\partial x_{i}}\frac{\partial \varphi}{\partial x_{j}} b^{ij}dx= \int_{\Omega}a\varphi dx\hh\forall~\varphi\in W_{B,A}(\Omega).
 \label{z74} 
  \end{equation}\hk
 $(ii)$~~By Lemma \ref{w1bb},  $|v|$ is in $W_{B,A}(\Omega)$ and 
 \begin{equation}\displaystyle \frac{\partial |v|}{\partial x_{i}}= sign(v)\frac{\partial v}{\partial x_{i}}\hh\forall~i\in \{1,\cdots, N\}.
 \label{vv}
 \end{equation}\hk
  Inserting $v$ and $|v|$ as  test functions in \eqref{z74}, we get
\[  < |v|, |v|>\hspace{-.1in}\leb{\eqref{vv}}\hspace{-.1in} < v, v> =\hspace{-.08in} \int_{\Omega}av dx\le \hspace{-.08in}\int_{\Omega}a|v| dx \hspace{-.1in}\eqb{\eqref{z74}}\hspace{-.1in}< v, |v|>=< |v|, v>. \]\hk
Thus 
\[  < |v|, |v|>\le < v, |v|>~ and~ {}^{\backprime\backprime}< v, v>\le < |v|, v>~or ~< |v|,- v> \le< v,-v>{}^{\prime\prime}. \]\hk
 Therefore
\[  < |v|-v, |v|>~\le 0~~ and~~<|v|-v, -v>~\le 0.\]\hk
It implies $ || |v|-v||_{B}\le 0$ or $v=|v|\ge 0$. \\\hk
   Since $v\ge 0$, by Lemma \ref{w1bb}, there is  a sequence of non-negative functions $\{v_{n}\}$  in $C^{1}(\Omega,A)$ such that $\{v_{n}\}$  converges to $v$ in $W_{B,A}(\Omega)$. Let $y$ be in $(0,\infty)$ and $\epsilon=\frac{1}{2}y$. Put
  \[\phi(z)= \frac{2}{y}-\frac{1}{z+\epsilon}\hh\forall~z\in (-\frac{1}{2}\epsilon,\infty).\]
  \hk 
  We have
  \[\phi'(z)=\frac{1}{(z+\epsilon)^{2}}\le \frac{4}{\epsilon^{2}}\hh\forall~z\in (-\frac{1}{2}\epsilon,\infty),\]
   \[|\phi''(z)|=|-\frac{2}{(z+\epsilon)^{3}}|\le \frac{16}{\epsilon^{3}}\hh\forall~z\in (-\frac{1}{2}\epsilon,\infty).\]\hk
    Thus $\phi\in C^{1}(-\frac{1}{2}\epsilon,\infty)$, $\phi'\in L^{\infty}(-\frac{1}{2}\epsilon,\infty)$, $\phi'$ is uniformly continuous on $(-\frac{1}{2}\epsilon,\infty)$ and $\phi(0)=0$.     
     By Lemma \ref{2w1ba}, $\varphi\equiv\phi\circ v$ is in $W_{B,A}(\Omega)$. Similarly  $w\equiv \psi\circ v$ is in $W_{B,A}(\Omega)$, where
\begin{equation}\psi(z)= \log(z+\epsilon)-\log(\frac{y}{2} )\hh\forall~z\in (-\frac{1}{2}\epsilon,\infty).
\label{zg0} 
\end{equation}\hk
      Put  $\Omega^{\xi}=\{x\in \Omega: v(x)>\xi\}$ for every $\xi\in(0,\infty)$.  We have 
 \begin{equation}      
    \varphi(x)\le \frac{2}{y}\hh\forall~x\in \Omega,    
 \label{zg3} 
\end{equation}      
\begin{equation}\frac{\partial \varphi }{\partial x_{j}} =\frac{\partial v }{\partial x_{j}}\frac{1}{ (v+\epsilon)^{2}},
\label{zg1} 
\end{equation}
\begin{equation}\frac{\partial w}{\partial x_{j}} = \frac{\partial v }{\partial x_{j}}\frac{1}{ v+\epsilon}.
\label{zg2} 
\end{equation}\hk
 Using $\varphi$ as a test function  in \eqref{281z},  we get
\[C(r,\Omega,A)^{-2}\{\int_{\Omega^{y}}|w|^{r}dx\}^{\frac{2}{r}} \leb{\eqref{cb}} \int_{\Omega}|\nabla w|^{2} bdx\eqb{\eqref{zg2}}\int_{\Omega}\frac{|\nabla v|^{2}}{ (v+\epsilon)^{2}} bdx\]
\[\leb{\eqref{c1}}\int_{\Omega}\frac{\partial v}{\partial x_{i}}\frac{\partial v}{\partial x_{j}}\frac{1}{ (v+\epsilon)^{2}} b^{ij}dx\eqb{\eqref{281z}}\int_{\Omega}a\varphi dx\le ||a||_{L^{1}(\Omega)}\frac{2}{y}\le M\frac{2}{y}, \]
or 
\begin{equation} \{\int_{\Omega^{y}}|w|^{r}dx\}^{\frac{2}{r}}\le MC(r,\Omega,A)^{2}\frac{2}{y}.\label{gam}
\end{equation}\hk
 Thus
\begin{equation}\{\int_{\Omega^{y}}\hspace{-.07in}\log^{r}(\frac{v}{y})dx\}^{\frac{2}{r}}\le\{\int_{\Omega^{y}}\hspace{-.07in}\log^{r}(\frac{v}{\frac{y}{2}})dx\}^{\frac{2}{r}}\le\{\int_{\Omega^{y}}\hspace{-.07in}\log^{r}(\frac{v+\epsilon}{\frac{y}{2}})dx\}^{\frac{2}{r}}
 \label{g81cz}
\end{equation}
\[\eqb{\eqref{zg0}} \{\int_{\Omega^{y}}|w|^{r}dx\}^{\frac{2}{r}}\leb{\eqref{gam}} MC(r,\Omega,A)^{2}\frac{2}{y}.\]\hk
Put $s=2y$ and $\lambda_{v}(\xi)=|\Omega^{\xi}|$ for every $\xi\in (0,\infty)$. We have
\[(\log 2)^{2}|\Omega^{s}|^{\frac{2}{r}}\leb{v(x)>s} \{\int_{\Omega^{s}}\log^{r}(2\frac{v}{s})dx\}^{\frac{2}{r}} \]
\[\leb{s>y} \{\int_{\Omega^{y}}\log^{r}(2\frac{v}{s})dx\}^{\frac{2}{r}}=\{\int_{\Omega^{y}}\hspace{-.15in}\log^{r}(\frac{v}{y})dx\}^{\frac{2}{r}}\leb{\eqref{g81cz}} M C(r,\Omega,A)^{2}\frac{4}{s}
\]
or
\begin{equation}s^{\frac{r}{2}}\lambda_{v}(s)\le  [4(\log 2)^{-2}M C(r,\Omega,A)^{2}]^{\frac{r}{2}}\hh\forall~s\in (0,\infty).
\label{zxz} 
\end{equation}\hk
 Let $\alpha\in (1,\frac{r}{2})$. Since $| \Omega|<\infty$, by Proposition 6.24 in \cite[p.198]{FO} and  \eqref{zxz} , we have
\[\int_{\Omega} |v|^{\alpha}dx = \alpha\int_{0}^{\infty} \hspace{-.12in}s^{\alpha-1}\lambda_{v}(s)ds\]
\[\le \alpha\int_{0}^{1}\hspace{-.12in} s^{\alpha-1}|\Omega|ds+\alpha[4(\log 2)^{-2}M C(r,\Omega,A)^{2}]^{\frac{r}{2}}\int_{1}^{\infty}\hspace{-.12in} s^{\alpha-1-\frac{r}{2}} ds\]
\[=|\Omega|+\frac{2}{r-2\alpha} [4(\log 2)^{-2}M C(r,\Omega,A)^{2}]^{\frac{r}{2}}=k_{\Omega,p,M,r},\]
where $k_{\Omega,p,M,r}=\{|\Omega|+\frac{2}{\gamma-2\alpha} [4(\log 2)^{-2}M C(r,\Omega,A)^{2}]^{\frac{r}{2}}\}^{\frac{1}{\alpha}}$.\\\hk
Thus we get \eqref{281az}. \\\hk
Now we consider the case in which $a$ can change the sign. Put $a_{1} =\max\{a,0\}$ and $a_{2} =\max\{-a,0\}$. Note that  $\max\{||a_{1}||_{L^{1}(\Omega)}, ||a_{2}||_{L^{1}(\Omega)}\}\le ||a||_{L^{1}(\Omega)}$. By the same techniques as above, we can find $v_{1}$ and $v_{2}$ in $W_{B,A}(\Omega)$ such that
 \[\int_{\Omega} \frac{\partial v_{k}}{\partial x_{i}}\frac{\partial \varphi}{\partial x_{j}} b^{ij}dx= \int_{\Omega}a_{k}\varphi dx\hh\forall~\varphi\in W_{B,A}(\Omega), k=1,2,
 \]
\[||v_{k}||_{L^{\alpha}_{b}(B_{R})}\le k_{\Omega,p,M,r}\hh\forall~\alpha\in (\frac{\overline{r}}{2},\frac{\gamma}{2}), k=1,2. \]\hk
 Put $v=v_{1}-v_{2}$. We have 
 \[\int_{\Omega} \frac{\partial v}{\partial x_{i}}\frac{\partial \varphi}{\partial x_{j}} b^{ij}dx= \int_{\Omega}a\varphi dx\hk\forall~\varphi\in W_{B,A}(\Omega)~, \]
\[||v||_{L^{p}_{b}(B_{R})}\le 2k_{\Omega,p,M,r}\hh\forall~p\in (\frac{\overline{r}}{2},\frac{r}{2})~. \]\hk
Thus we get $(ii)$. \\\hk
$(iii)$~~ By simple computations, we  have
\begin{equation}
\lim_{r\to\infty}
\frac{r}{2}=\infty>\frac{4}{3}\hh if ~N=2,
\label{ac0}
\end{equation}
\begin{equation}\lim_{r\to\frac{2N}{N-2}}\frac{r}{2}=\frac{N}{N-2}\ge\frac{4}{3}\hh\forall~N\in\{3,\cdots,8\},
\label{ac0b}
\end{equation}
  \[\lim_{r\to\infty}[\lim_{\gamma\to\frac{r}{2}}\{\lim_{\overline{r}\to 2}\frac{r\overline{r}}{\gamma(r-\overline{r})}\}]
 =\lim_{r\to\infty}[\frac{4)}{r-2}]= 0=N-2\hh if~~N=2,\]
 \[\lim_{r\to\frac{2N}{N-2}}[\lim_{\gamma\to\frac{r}{2}}\{\lim_{\overline{r}\to 2}\frac{r\overline{r}}{\gamma(r-\overline{r})}\}]= \lim_{r\to\frac{2N}{N-2}}[\frac{4}{r-2}]= N-2 \hh\forall~N\in\{3,\cdots,8\}.\]
\hk
 Since $\lim_{t\to 2}\frac{t(N+1)-2}{N-t}= 2^{\ast}$ and $r=\frac{t(N+1)-2}{N-t}$, we get 
  \[\lim_{t\to 2}[\lim_{\gamma\to\frac{\gamma}{2}}\{\lim_{\overline{r}\to 2}\frac{r\overline{r}}{\gamma(r-\overline{r})}\}]= \lim_{r\to\infty}[\lim_{\gamma\to\frac{r}{2}}\{\lim_{\overline{r}\to 2}\frac{r\overline{r}}{\gamma(r-\overline{r})}\}]= N-2\hk \forall~N\in\{2,3,\cdots,8\}.\]
\hk
 Thus we can choose $t(\zeta)\in (\frac{2N^{2}+2N-2}{N^{2}+2N-1},2)$, $\overline{r}(\zeta)\in (2,\frac{t_{\zeta}(N+1)-2}{N-t_{\zeta}})$ and $\gamma$ near $\frac{1}{2}\gamma$ such that 
 \begin{equation}\frac{r\overline{r}}{\gamma(r-\overline{r})}< \nu+N-2+\zeta\hh\forall~ t\in(t_{\zeta},2), \overline{r}\in (2, \overline{r}(\zeta)), N\in\{2,\cdots,8\},
 \label{ac2}
\end{equation}
\hk
By \eqref{ac0}-\eqref{ac2}, Proposition \ref{pro46b} and Remark \ref{rr},  we obtain \eqref{le2460bzzza} and
 \begin{equation} 
 |v(x)| \le  C(B,r,\overline{r},M,b,\zeta)R^{2-N-\zeta} \hk a.e~on~\Omega\setminus B(z,R), 
 \label{281bz}
\end{equation}\hk
  Let $i\in \NN$. Arguing as above with $2^{i}R$ instead of $R$, as in \eqref{281bz}, we get
 \[|v(x)| \le  C(B,r,\overline{r},M,b,\zeta)(2^{i}R)^{2-N-\zeta} \hh a.e~on~\Omega_{i}\equiv [\Omega\cap B(z,2^{i+1}R)]\setminus B(z,2^{i}R).\]\hk
Changing  values of $v$ in a measurable set $A_{i}\subset \Omega_{i}$ with null measure, we can (and shall) suppose
 \[|v(x)| \le  C(B,r,\overline{r},M,b,\zeta)(2^{i}R)^{2-N-\zeta}=C(B,r,\overline{r},M,b,\zeta)(\frac{|x-z|}{2^{i}R})^{N-2+\zeta}|x-z|^{2-N-\zeta}
 \]
 \[\le 2^{N-2+\zeta}C(B,r,\overline{r},M,b,\zeta)|x-z|^{2-N-\zeta} \hh \forall~x\in~\Omega_{i}, i=0,\cdots.\]\hk
 Thus, replacing $C(B,r,\overline{r},M,b,\zeta)$ by $2^{N-2+\zeta}C(B,r,\overline{r},M,b,\zeta)$, we get \eqref{281bzd}.
 Using \eqref{281bz} and arguing as in the proof of \eqref{le4460bzzzb}, we obtain \eqref{le2460bzzza}.
\end{proof}
 \section{Green functions}\label{gr}\hk
 \begin{definition} Let $\eta$ be a non-negative function in $C^{\infty}(\RR)$ with support contained in $[0,\frac{1}{2} ]$  and $\displaystyle \int_{\RR}\eta dx =1$.  Put
 \begin{equation}\nu(s)= 1-\int_{0}^{s}\eta(\xi)d\xi\hh\forall~s\in \RR,
  \label{psi1}
 \end{equation} 
 \begin{equation}\psi(x)=\alpha_{N}\nu(|x|)\hh\forall~x\in \RR^{N},
 \label{psi2}
 \end{equation}
 \begin{equation}\psi_{\rho}(x)=\rho^{-N}\psi(\rho^{-1}x)\hh\forall~\rho\in (0,1), x\in\RR^{N},
  \label{psi3}
 \end{equation}
 where $\displaystyle \alpha_{N}=[\int_{\RR^{N}}\nu(|x|)dx]^{-1}$.
 \label{psi}
 \end{definition}\hk
 Then 
  \begin{equation}\nu'(s)= -\eta(\xi)\le 0 \hh\forall~s\in \RR,
  \label{psi4}
 \end{equation} 
   \begin{equation}\psi_{\rho}\in C^{\infty}(\Omega)~and~ ||\psi_{\rho}||_{L^{1}(\Omega)}=1\hk\forall~\rho\in (0,1],
  \label{psi5}
 \end{equation}
  \begin{equation} supp(\psi_{\rho})\subset B(0,\rho)~and~ \psi_{\rho}(x)=\rho^{-N}\hk\forall~\rho\in (0,1), x\in B(0,\frac{1}{2}\rho),
  \label{psi6}
 \end{equation}
  \begin{equation}  \frac{\psi_{\rho}}{\partial x_{i}}(x)=\rho^{-N-1}\frac{x_{i}}{|x|}\nu'(\rho^{-1}x)\hk\forall~ x\in B(0,\rho), \rho\in (0,1], i\in\{1,\cdots,N\},
  \label{psi7}
 \end{equation}\hk 
   Let $B(z,2\rho)$ be a subset of $\Omega$. Put $a_{z,\rho}(x)= \psi_{\rho}(x-z)$ for  every $x$ in $\Omega$.    Applying \eqref{281bzd} and \eqref{281az}  in Proposition \ref{lem62xz} for $a=a_{z,\rho}$,  we get $v_{z,\rho}$ in $ W_{B,A}(\Omega)$ such that  
 \begin{equation}
 \int_{\Omega}\frac{\partial v_{z,\rho}}{\partial x_{i}}\frac{\partial \varphi}{\partial x_{j}}b^{ij}dx=\int_{\Omega}\psi_{\rho}(x-z)\varphi dx\hh\forall~\varphi\in W_{B,A}(\Omega)
 \label{gre}
 \end{equation} 
 Put $G_{\rho}(x,z)= v_{z,\rho}(x)$ for every $x$ in $\Omega$. We have following results
 \begin{lemma} Let $\delta$ be  as in Lemma \ref{lemfs1b},   $\zeta \in (0,\min\{\delta,\frac{1}{120}\})$,  $A$ be  admissible with respect to $\Omega$. Let $N$ be in $\{2,\cdots,8\}$,  $t(\zeta)$ and $\overline{r}(\zeta)$ be as in Proposition \ref{lem62xz}, $t\in (t(\zeta),2)$, $r=\frac{t(N+1)-2}{N-t}$, $\overline{r}\in (2,\overline{r}(\zeta))$. Let $z\in \Omega$ such that $B(z,4\rho_{z})\subset \Omega$ with $\rho_{z}\in (0,1)$. Assume    \eqref{c1}-\eqref{c4} and \eqref{c9} hold. Then there is a positive real number $C$ independent from $\rho$ and $\alpha >1$ such that
\begin{equation}
 0\le G_{\rho}(x,z) \le  C|x-z|^{-N+2-\zeta} \hk  a.e. ~x\in\Omega\setminus B'(z,2\rho), 
 \label{81bzb}
 \end{equation}
 \begin{equation}
 ||G_{\rho}(.,z)||_{L^{\alpha}(\Omega)} \le  C, 
 \label{81bzbc}
 \end{equation}
  \begin{equation}
\{\int_{\Omega\setminus B(z,s)}|\nabla  G_{\rho}(x,z)|^{2} bdx\}^{\frac{1}{2}}< Cs^{-N+1-\zeta}\hh if~~s> 2\rho
 \label{81c}
\end{equation}
 \label{lem62x}
\end{lemma}\hk
Now we get the Green function as follows.
 \begin{proposition}[Green function]  Let    $\zeta \in (0,1)$,  $A$ be  admissible with respect to $\Omega$. Let $N$ be $in \{2,\cdots,8\}$,  $t(\zeta)$ and $\overline{r}(\zeta)$ be as in Proposition \ref{lem62xz}, $t\in (t(\zeta),2)$, $r=\frac{t(N+1)-2}{N-t}$, $\overline{r}\in (2,\overline{r}(\zeta))$. Let $z\in \Omega$ such that $B(z,4\rho_{z})\subset \Omega$ with $\rho_{z}\in (0,1)$. Assume    \eqref{c1}-\eqref{c4} hold. Then there is a function $G$ on $\Omega\times \Omega$ having following properties.\\\hk
$(i)$~ There is a sequence $\{\rho_{m}\}\subset (0,
\rho_{z})$ converging to $0$ having  following properties: there is a subset $A_{z}$ of $\Omega$ for every $z$ in $\Omega$  such that $|A_{z}|=0$ and  $\{v_{z,\rho_{m}}(x)\}$  converges to $G(x,z)$  for every $x$ in $\Omega\setminus A_{z}$.\\\hk
$(ii)$~ ~~There is $\overline{C}(B,r,\overline{r},\zeta,\Omega)$  such that $\overline{C}(B,r,\overline{r},\zeta,\Omega)$ is  independent of $\rho$ and 
 \begin{equation} 
 0\le G(x,z) \le  \overline{C}(B,r,\overline{r},\zeta,\Omega)|x-z|^{-N+2-\zeta}, 
 \label{81bzz}
\end{equation}
\hh\hh a.e.~$x\in\Omega\setminus\{z\}~,~if~ \overline{r }\in (2,\overline{r}(\zeta))~and~  r \in (r(\zeta),2^{\ast})$.\\\hk
$(iii)$  Let  $u\in W_{B,A}(\Omega)$ and $f\in L^{1}(\Omega)$ such that
\begin{equation}
\int_{\Omega} \frac{\partial u}{\partial x_{i}} \frac{\partial \varphi}{\partial x_{j}} b^{ij}dx=\int_{\Omega}f\varphi dx\hh\forall~\varphi\in W_{B, A}(\Omega).
\label{834}
\end{equation}\hk
Then 
\begin{equation}\int_{\Omega}G(x,y)f(x)dx= u(y)\hh a.e.~y\in \Omega.
\label{834b}
\end{equation}\hk
$(iv)$~~$u\in L^{\frac{sN}{N-s(2-\zeta)}}(\Omega)$ if $f\in L^{s}(\Omega)$ with a $s\in[1,\frac{N}{2-\zeta})$.\\\hk
$(v)$~~$u\in L^{q}(\Omega)$ for  every $q$ in $[1,\infty)$ if $f\in L^{s}(\Omega)$ with some $s\in[\frac{N}{2-\zeta},\infty)$.
\label{pro62}
\end{proposition}
\begin{proof} $(i)$~~Let $\{\rho_{m}\}$ be a strictly decreasing sequence in $(0,\frac{1}{2}\rho_{z})$ which converges to $0$. We have
\[ \int_{\Omega\setminus B(z,\rho_{m})}\hspace{-.5in}|\nabla v_{z,\rho}|dx \le  ||b^{-1}||_{L^{1}(\Omega)}^{\frac{1}{2}} \{\int_{\Omega\setminus B(z,\frac{1}{2}\rho_{z})}\hspace{-.5in}|\nabla v_{z,\rho}|^{2}bdx\}^{\frac{1}{2}}\leb{\eqref{81c}}C(b,r,\overline{r},\Omega)s^{\frac{2-N-\zeta}{2}}\hk\forall~m\in\NN.
\]
By \eqref{81bzb} and \eqref{81bzbc}, it implies  $\{ v_{z,\rho_{m}}\}_{\rho_{m}<\frac{1}{2}\rho_{z}}$ is bounded in $W^{1,1}(\Omega\setminus B(z,\frac{1}{2}\rho_{z}))$. By  Rellich–Kondrachov's Theorem and Theorem 4.9 in \cite[p.285, p.94]{BR}, we can find a subsequence $\{\rho_{m,1}\}$ of $\{\rho_{m}\}$  and $A_{1} \subset \Omega$ with $|A_{1}|=0$  such that $\{v_{z,\rho_{m,1}}|_{\Omega\setminus B(z,\frac{1}{2}\rho_{z})}\}$  converges  pointwise  on $(\Omega\setminus B(z,\frac{1}{2}\rho_{z}))\setminus A_{1}$. By mathematical induction, we can find  sequences $\{\{\rho_{m,k}\}\}_{k}$ and $\{A_{k}\}$ such that $\{\rho_{m,k+1}\}$ is a subsequence of $\{\rho_{m,k}\}$, $A_{k} \subset \Omega$ with $|A_{k}|=0$, and $\{v_{z,\rho_{m,k}}|_{\Omega\setminus B(z,2^{-k}\rho_{z})}\}$  converges  pointwise  on $(\Omega\setminus B(z,2^{-k}\rho_{z}))\setminus A_{k}$ for every $k$ in $\NN$.\\\hk
Now using the diagonal sequence technique, we can find a subsequence $\{\rho_{m_{k}}\}$ of $\{\rho_{m}\}$  and $A_{z} \subset \Omega$ with $|A_{z}|=0$  such that $\{v_{z,\rho_{m_{k}}}|_{\Omega\setminus B(z,2^{-n}\rho_{z})}\}$  converges  pointwise  on $(\Omega\setminus B(z,2^{-n}\rho_{z}))\setminus A_{z}$ for every $n$ in $\NN$. Thus, $\{v_{z,\rho_{m_{k}}}\}$ converges  pointwise  on $\Omega\setminus (A_{z}\cup \{z\})$. Put
\[G(x,z)= \lim_{k\to \infty} v_{z,\rho_{m_{k}}}(x)\hh\forall~x\in \Omega\setminus (A_{z}\cup \{z\}).\]\hk
$(ii)$~~By \eqref{81bzb}, we get \eqref{81bzz}.  \\\hk
$(iii)$~~ We prove $(iii)$ by following steps.\\\hk
{\bf Step 1.} ~Put $v_{1,m_{k}}=v_{z,\rho_{m_{k}}}\chi_{\Omega\setminus B(z,2\rho_{m_{k}})}$ and $v_{2,m_{k}}=v_{z,\rho_{m_{k}}}\chi_{B(z,2\rho_{m_{k}})}$ for every $k\in\NN$.  Then there is a positive real number $C$ such that
\begin{equation}
0\le v_{1,m_{k}}(x)\leb{\eqref{81bzb}} C|x-z|^{-N+2-\zeta}\hk                             ~a.e. ~x\in \Omega\setminus \{z\}, \forall~k\in\NN.
\label{rm1}
\end{equation}\hk
 Let $\alpha$ be as in Lemma \ref{lem62x}, we have
\begin{equation}
||v_{2,m_{k}}||_{L^{\alpha}(\Omega)}\leb{\eqref{81bzbc}} C\hh\forall~k\in\NN,
\label{rm2}
\end{equation}\hk
 Let $f \in L^{\infty}(\Omega)$ and $M \in (0,\infty)$ such that $\Omega \subset B(0,M)$. Put 
\begin{equation} \overline{f}(x) = \left\{ {\begin{array}{*{20}l}
 f(x) & {} & {\forall~ x \in \Omega,}  \vspace{.1in}\\
      0\hh& {} & {\forall~ \xi \in \RR^{N}\setminus\Omega,}  \\
\end{array}} \right.
\label{834bb}
\end{equation}
\[g(x)= \overline{C}(B,r,\overline{r},\zeta,\Omega)\chi_{B(0,2M)\setminus \{0\}}(x)|x|^{-N+2-\zeta}\hh\forall~x\in \RR^{N}.\]\hk
 Using the polar coordinates with the pole at $0$, we have
 \begin{equation}\int_{\RR^{N}}g(x)dx =\int_{B(0,2M)}g(x)dx = \overline{C}(B,r,\overline{r},\zeta,\Omega)|\Sigma_{N-1}| \int_{0}^{2M}t^{1-\zeta}dt
  \label{834c1}
\end{equation}
  \[=\overline{C}(B,r,\overline{r},\zeta,\Omega)\frac{|\Sigma_{N-1}|}{2-\zeta}(2M)^{2-\zeta}<\infty. \]
   \hk Thus, by Young's Theorem in \cite[Theorem 4.15, p. 104]{BR},
\[\int_{\RR^{N}}\int_{\RR^{N}}g(z-y)|\overline{f}|(y)dydz\leb{\eqref{834bb},\eqref{834c1}} ||g||_{L^{1}(\RR^{N})}||\overline{f}||_{L^{1}(\RR^{N})} <\infty.\]\hk
 By Fubini's Theorem \cite[Theorem 4.5, p.91]{BR}, there is  a measurable subset $X$ of $\Omega$ such that $|X|=0$ and 
 \begin{equation}
\int_{\RR^{N}}g(z-y)|\overline{f}|(y)dy < \infty \hh\forall~z\in \Omega\setminus X.
\label{8834b}
\end{equation}
 By \eqref{81bzb}, $|v_{1,m_{k}}(y)| f(y)|\le g(z-y)|f(y)|$  a.e. $y$ on $\Omega\setminus\{z\}$ for every $z\in \Omega$. By Lebesgue's Dominated Convergence theorem and  $(i)$, it implies 
\begin{equation}
\int_{\Omega}G(y,z)f(y)dy =\lim_{k\to\infty}\int_{\Omega}v_{1,m_{k}}(y)f(y)dy  \hh\forall~z\in \Omega\setminus X.
\label{8834c}
\end{equation}\hk
On the other hand, we have
\begin{equation}
\lim_{k\to\infty}\int_{\Omega}\hspace{-.05in}v_{2,m_{k}}fdy \hspace{-.04in} \leb{H\ddot{o}lder} \hspace{-.04in} ||f||_{L^{\infty}(\Omega)}\lim_{k\to\infty}\hspace{-.04in}|B(z,2\rho_{m_{k}})|^{\frac{\alpha-1}{\alpha}}\{\int_{B(z,2\rho_{m_{k}})}\hspace{-.27in}|v_{2,m_{k}}|^{\alpha}dy\}^{\frac{1}{\alpha}}
\label{8834cd}
\end{equation}
\[\leb{\eqref{rm2}} C||f||_{L^{\infty}(\Omega)}\lim_{k\to\infty}|B(z,2\rho_{m_{k}})|^{\frac{\alpha-1}{\alpha}}= 0 .\]\hk
 Combining \eqref{8834c} and \eqref{8834cd},
  we get
\begin{equation}\int_{\Omega}G(y,z)f(y)dy =\lim_{k\to\infty}\int_{\Omega}v_{z,\rho_{m_{k}}}(y)f(y)dy  \hh\forall~z\in \Omega\setminus X.
\label{8834cdf}
\end{equation}  \hk 
Since $u$ is in $W_{B,A}(\Omega)$, we have
 \begin{equation}
 \int_{\Omega}|u|dx\leb{H\ddot{o}lder} |\Omega|^{\frac{r-1}{r}}\{\int_{\Omega}|u|^{r}dx\}^{\frac{1}{r}}\leb{\eqref{cb}} |\Omega|^{\frac{r-1}{r}}C(r,\Omega,A)||u||_{B}<\infty.
 \label{ul1}
 \end{equation}\hk
On the other hand, by using $v_{z,\rho_{m_{k}}}$ and $u$  as test functions for \eqref{834} and \eqref{gre} respectively,  and  by Theorem 7.7 in \cite[p.138]{RU}, there is $Y$ such that $X\subset Y\subset \Omega$, $|Y|=0$ and 
\begin{equation}
\lim_{k\to\infty}\int_{\Omega}v_{z,\rho_{m_{k}}}(x)f(x)dx\eqb{\eqref{834}}\lim_{k\to\infty}\int_{\Omega} \frac{\partial u}{\partial x_{i}}\frac{\partial v_{z,\rho_{m_{k}}}}{\partial x_{ji}}b^{ij}  dx
\label{sept1}
\end{equation}
\[\eqb{b^{ij}=b_{ji}}\lim_{k\to\infty}\int_{\Omega} \frac{\partial v_{z,\rho_{m_{k}}}}{\partial x_{i}} \frac{\partial u}{\partial x_{j}}b^{ij} dx\eqb{\eqref{gre}}  
\lim_{k\to\infty}\int_{\Omega}\psi_{\rho_{m_{k}}}(x-z)u(x) dx\]
\[ \eqb{Theorem~ 7.7,\eqref{ul1}}u(z)\hh\forall~z\in \Omega\setminus Y.\]\hk
Combining  \eqref{8834cdf} and \eqref{sept1}, we get
\begin{equation}\int_{\Omega}G(y,z)f(y)dy =u(z)\hh a.e.~z\in \Omega.
\label{sept1b}
\end{equation}\hk
{\bf Step 2.}~There is a real number $M_{1}$ such that
\begin{equation}
|a_{z,\rho}|\le M_{1}\rho^{-N}\chi_{B(z,\rho)}.
\label{st1}
\end{equation}\hk
Let $y\in B(z,2\rho_{m_{k}})$. Applying \eqref{sept1b} for $f=a_{z,\rho_{m_{k}}}$ and $u=   v_{z,\rho_{m_{k}}}$, we get
\[ v_{z,\rho_{m_{k}}}(y)\hspace*{-.05in}\eqb{\eqref{sept1b}}\hspace*{-.05in} \int_{\Omega}G(x,y)a_{z,\rho_{m_{k}}}(x)dx\hspace*{-.15in}\leb{\eqref{81bzz},\eqref{st1}}\hspace*{-.15in}\overline{C}(B,r,\overline{r},\zeta,\Omega)M_{1}\rho_{m_{k}}^{-N}\hspace*{-.05in}\int_{B(z,\rho_{m_{k}})}\hspace*{-.4in}|x-y|^{-N+2-\zeta}dx
\]
\[\le\overline{C}(B,r,\overline{r},\zeta,\Omega)M_{1}\rho_{m_{k}}^{-N}\hspace*{-.05in}\int_{B(z-y,\rho_{m_{k}})}\hspace*{-.23in}|t|^{-N+2-\zeta}dt
\]
\[\leb{B(z-y,\rho_{m_{k}})\subset B(0,3\rho_{m_{k}})}\hspace{-.23in}CM\rho_{m_{k}}^{-N}\hspace{-.05in}\int_{B(0,3\rho_{m_{k}})}\hspace{-.43in} |t|^{-N+1-\zeta}dt=CM\rho_{m}^{-N}\int_{\Sigma_{N-1}}\int_{0}^{3\rho_{m_{k}}}\hspace{-.15in}\xi^{-N+2-\zeta+N-1}d\xi ds\]
\[=\frac{3^{2-\zeta}|\Sigma_{N-1}|}{2-\zeta}\rho_{m_{k}}^{-N+2-\zeta}\leb{|z-y|<2\rho}\frac{3^{2-\zeta}2^{N-2+\zeta}|\Sigma_{N-1}|}{2-\zeta}|z-y|^{-N+2-\zeta}.\]\hk
Thus, by \eqref{81bzb}, there is a positive real number $C_{1}$ independent from $k$ such that 
\begin{equation}
0\le v_{z,\rho_{m_{k}}}(x)\le C_{1}|x-z|^{-N+2-\zeta} \hh a.e.~x\in \Omega\setminus \{z\}.
\label{st2}
\end{equation}\hk
{\bf Step 3.}~Let $f\in L^{1}(\Omega)$. Put $M$ , $\overline{f}$  as in Step 1 and 
\[g_{1}(x)=C_{1}|x-z|^{-N+2-\zeta} \hh \forall~~x\in \Omega\setminus \{z\}.\]\hk
 Then $\overline{f}$ and $g_{1}$ are integrable on $\Omega$ and 
  \[v_{z,\rho_{m_{k}}}(x)|\overline{f}|(x)\leb{\eqref{st2}} g_{1}(x-z)|\overline{f}|(x)\hh a.e.~x~\in \Omega\setminus\{z\}.\]\hk
  Arguing as in Step 1, we get
  \begin{equation}\int_{\Omega}G(y,z)f(y)dy= \lim_{k\to\infty}\int_{\Omega}v_{z,\rho_{m_{k}}}(y)f(y)dy  \hh\forall~z\in \Omega\setminus X.
  \label{st3}
\end{equation}\hk
 Combining \eqref{sept1} and \eqref{st3}, we obtain \eqref{834b}.\\\hk
 $(iv)$~~Let $s\in [1,\frac{N}{2-\zeta})$. Since $\frac{1}{s}-\frac{2-\zeta}{N}= \frac{N-s(2-\zeta)}{sN}$, by  Lemma \ref{Riez}, \eqref{81bzz} and \eqref{834b}, we get $u\in L^{\frac{sN}{N-s(2-\zeta)}}(\Omega)$.\\\hk 
  $(v)$ Let $s\in [\frac{N}{2-\zeta},\infty)$ and $s'\in [1,\frac{N}{2-\zeta})$. Then $f$ is in $L^{s'}(\Omega)$. By $(iv)$, $u\in L^{\frac{s'N}{N-s'(2-\zeta)}}(\Omega)$. Since $\lim_{s'\to\frac{N}{2-\zeta}} \frac{s'N}{N-s'(2-\zeta)}=\infty$, we get $(v)$.
\end{proof}
\begin{remark} If $b$ and $\overline{b}$ are constant and $A=\partial\Omega$, we get Proposition \ref{pro62} with $\zeta=0$ for every $N$ in $\{3,4,\cdots\}$ by using Theorem $(1.1)$ in \cite{GW} and the techniques in the proof of Proposition \ref{pro62}. 
\label{gw1}
\end{remark}
\begin{remark}
  Let  $\sigma_{N}=|B(0,1)|$, $\beta \ge 2$, $b(x)=1$ for every $x$ in $\RR^{N}$, $\{x_{5}, x_{6},\cdots\}$ be a dense subset of $\RR^{N}$, $r_{k}= 2^{-k^{4k\beta}}$, $v_{k}=|B(x_{k},r_{k})|= 2^{-Nk^{4k\beta}}\sigma_{N}$  for $k\in~\{5, 6,\cdots\}$, $\overline{k}= \frac{k+3}{2}$ for every odd integer $k$,  $\overline{k}= \frac{k+2}{2}$ for every even integer $k$ and 
 \[ \overline b  = \chi_{\RR^{N} \backslash (\bigcup\limits_{k \ge 5}^\infty   B(x_{k},2r_{k})\backslash B(x_{k},r_{k}))}+   
\sum_{k=5}^{\infty}k^{4k}\chi_{B(x_{k},2r_{k})\backslash B(x_{k},r_{k}) }.\]\hk
  We have\\
 $\displaystyle \int_{B(x_{k},2r_{k})}\hspace*{-.35in} k^{4\beta k}dx=  k^{4\beta k}2^{N}v_{k}= 2^{N} \sigma_{N} k^{4\beta k}2^{-k^{4\beta k}}2^{-(N-1)k^{4\beta k}}\hspace*{-.15in} \leb{N\ge 2} \hspace*{-.05in}2^{N}\sigma_{N}2^{-k^{4\beta k}}\hspace*{-.35in}\leb{N\ge 2,\beta\ge2, k\ge 5}\hspace*{-.32in} 2^{N}\sigma_{N}2^{-k}, $\\
  $\displaystyle \int_{B(x_{k},2r_{k})\setminus B(x_{k},r_{k})}\hspace{-.58in}k^{4k}dx\hspace{-.03in}= \hspace{-.02in}k^{4k}\sigma_{N}(2^{N(1-k^{4\beta k})}-2^{-Nk^{4\beta k}})= k^{4k}
\sigma _{N}2^{-Nk^{4\beta k}}(2^{N}-1)\hspace{-.11in} \ggeb{N\ge 2}\hspace{-.1in} 3k^{4k}v_{k}, $\\
$\displaystyle \sum_{j=5}^{\overline{k}}\int_{B(x_{k},2r_{k})}\hspace*{-.25in}j^{4j}dx\eqb{\overline{k}\le k-1} 2^{N}\sum_{j=1}^{k-1}(k-1)^{4(k-1)}v_{k}\le 2^{N}k.k^{-4}k^{4k}v_{k}\le 2^{N}k^{-3}k^{4k}v_{k}.$\\
$\displaystyle \sum_{j=\overline{k}+1}^{k-1}\int_{B(x_{k},2r_{k})}\hspace*{-.45in}j^{4j}dx=2^{N} \sum_{j=\overline{k}+1}^{k-1}j^{4j}v_{k}\le 2^{N} (\sum_{j=\overline{k}+1}^{k-1}j^{4j-4k})k^{4k}v_{k}\le 2^{N} (\sum_{j=\overline{k}+1}^{k-1}j^{-4})k^{4k}v_{k}$\\
$\displaystyle\hk\le 2^{N} (\int_{\overline{k}}^{k}\hspace*{-.05in} t^{-4}dt)k^{4k}v_{k}\le 2^{N} \overline{k}^{-3}k^{4k}v_{k}\hspace*{-.09in}\leb{\frac{k}{2}\le \overline{k}
}\hspace*{-.09in} 2^{N+3}k^{-3}k^{4k}v_{k},$
\\$\displaystyle \sum_{j=k+1}^{\infty}\int_{B(x_{j},2r_{j})}\hspace*{-.4in}j^{4j}dx
= 2^{N}\sum_{j=k+1}^{\infty}j^{4j}2^{-Nj^{4\beta j}}\sigma_{N}= 2^{N}\sum_{j=k+1}^{\infty}j^{4j}2^{-j^{4j}}2^{j^{4j}-Nj^{4\beta j}+Nk^{4\beta k}}v_{k}$\\
$\displaystyle = 2^{N}\sum_{i=1}^{\infty}(k+i)^{4(k+i)}2^{-(k+i)^{4(k+i)}}2^{(k+i)^{4(k+i)}-N(k+i)^{4\beta (k+i)}+Nk^{4\beta k}}v_{k}$\\
$\displaystyle\leb{\beta\ge 2} 2^{N}\sum_{i=1}^{\infty}2^{(k+i)^{4(k+i)}-N(k+i)^{4\beta k}(k+i)^{4 i}(k+i)^{4i}+Nk^{4\beta k}}v_{k}$\\
$\displaystyle\le 2^{N}\sum_{i=1}^{\infty}2^{(k+i)^{4(k+i)}-4N(k+i)^{4\beta k}(k+i)^{4i}+Nk^{4\beta k}}v_{k}\le 2^{N}\sum_{i=1}^{\infty}2^{-N(k+i)^{4\beta k}(k+i)^{4i}}v_{k}$\\
$\displaystyle\le  2^{N}\sum_{i=1}^{\infty}2^{-i}v_{k}= 2^{N}v_{k}
$\\\hk
 Thus $b$, $b^{-1}$ and $\overline{b}^{-1}$ are in $L^{\infty}(\RR^{N})$, $\overline{b}$ is in $L_{loc}^{\beta}(\RR^{N})$, $b(B(2^{-j},2.2^{-j^{\beta}})= 2^{N}b(B(2^{-j},2^{-j^{\beta}})$ for every $k\ge 4$ and \\
 $\hspace*{1in}\displaystyle \overline{b}(B(x_{k},2r_{k}))= \int_{B(x_{k},2r_{k})}\hspace*{-.4
 in}\overline{b}dx \ge \int_{B(x_{k},2r_{k})\setminus B(x_{k},r_{k})}\hspace*{-.4in}k^{4k}dx\ge 3k^{4k}v_{k},$\\
 $\displaystyle\overline{b}(B(x_{k},r_{k}))= \int_{B(x_{k},r_{k})}\hspace*{-.4in}\overline{b}dx 
 \le  \sum_{j=5}^{k-1}\int_{B(x_{k},r_{k})}\hspace*{-.25in}j^{4j}dx+ \int_{B(x_{k},r_{k})}dx+ \sum_{j=1}^{\infty}\int_{B(x_{k+j},2r_{k+j})}\hspace*{-.25in}(k+j)^{4(k+j)}dx $\\
 $\displaystyle\le 2^{N}[k^{-3}+2^{3}k^{-3}+ k^{-4k}+ k^{-4k}]k^{4k}v_{k},$\\
 $\hspace*{1in}\displaystyle\lim_{k\to\infty} \frac{\overline{b}(B(x_{k},2r_{k}))}{\overline{b}(B(x_{k},r_{k}))}\ge \lim_{k\to\infty}\frac{3}{2^{N}(9k^{-3}+  2.k^{-4k})}=\infty.$\\\hk
 Therefore $\overline{b}$ is not of class $A_{2}$,  $b$ and $\overline{b}$   do not satisfy \eqref{cha}, but $b|_{\Omega}$ and $\overline{b}|_{\Omega}$ satisfy conditions of Proposition \ref{pro62} for every bounded open subset $\Omega$ of $\RR^{N}$ and sufficiently large $\beta$.
 \label{r4}
\end{remark}
\begin{remark} Let $\Omega= B(0,1)$, $\overline{b}$ be as in Remark \ref{r4}, $\gamma\in (0,1)$ and $b(x)= dis(x,\partial\Omega)^{\gamma}$. Then $b$ and $\overline{b}$ satisfy conditions of Proposition \ref{pro62} for  sufficiently small $\gamma$  and sufficiently large $\beta$. Let $b^{ij}=b_{i}\delta_{i}^{j}$ for every $i,j$ in $\{1,\cdots,N\}$, $b_{1}=\cdots=b_{N-1}=b$ and $b_{N}=\overline{b}$. Applying our results, we can study   degenerate and non-uniform elliptic equations with $\overline{b}$ is not in the class $A_{2}$.
\label{r4b}
\end{remark}
\section{Global H\"{o}lder continuity of solutions for elliptic equations}\label{ho}\hk
 To study the  H\"{o}lder continuity of solutions for elliptic equations, we have to use the Riesz potential defined as follows.
 \begin{definition} Let $\varrho\in (0,N)$. The Riesz potential of order $\varrho$ on $\Omega$, denoted by $I_{\Omega,\varrho}$, is defined as:
\[ I_{\Omega,\varrho}f(x) = \int_{\Omega}\frac{f(y)}{|x-y|^{N-\varrho}}dy\hh\forall~f\in L^{1}(\Omega).\]
 \label{Riez0}
\end{definition}\hk
We have following properties of $I_{\Omega,\varrho}$.
\begin{lemma}[Hardy–Littlewood–Sobolev]
 $(i)$~~If $1<p<q<\infty$ such that $\frac{1}{q}=\frac{1}{p}-\frac{\varrho}{N}$, then $I_{\Omega,\varrho}f$ is in $L^{q}(\Omega)$ for every $f\in L^{p}(\Omega)$. \\\hk
$(i)$~~If $p\ge \frac{N}{\varrho}$ and $f\in L^{p}(\Omega)$ then $I_{\Omega,\varrho}f$ is in $L^{q}(\Omega)$ for every $q\in (1,\infty)$.
\label{Riez}
\end{lemma}
\begin{proof} By Theorem 2.6 in \cite[p.35]{LP}, we get $(i)$.\\\hk
$(ii)$~~If $f\in L^{p}(\Omega)$ with some $p\ge \frac{N}{\varrho}$, then $f\in L^{\frac{N}{\varrho+k^{-1}}}(\Omega)$ for every $k \in (\frac{1}{N-\varrho},\infty)$. By $(i)$, $I_{\Omega,\varrho}f\in L^{kN}(\Omega)$. Thus we get $(ii)$.
\end{proof}\hk
Let $\zeta\in [0,\frac{1}{4})$ and $n\in\NN$.  We consider the following condition:  there is a positive real number $C_{n,\zeta,b,\overline{b},\Omega}$ such that
\begin{equation}
|\nabla_{x}G(x,z)|\le C_{n,\zeta,b,\overline{b},\Omega}|x-z|^{-N+1+\zeta}\hh\forall~x,z\in (\Omega).
\label{thuyen}
\end{equation}
\begin{remark}~~If $A=\partial\Omega$ and $b^{ij}$ are Dini-continuous, then we get condition \ref{thuyen} with $\zeta=0$ in \cite[p.333]{GW}.
\end{remark}\hk
 Let  $\psi$ and  $\psi_{\rho}$ be defined as in Definition \ref{psi} and $l\in \{1,...,N\}$.  Let $B(z,2\rho)$ be a subset of $\Omega$. Put $\displaystyle w_{1,\rho,l}(x)= \frac{\partial \psi_{\rho}}{\partial x_{l}}(x-z)$ and $\displaystyle w_{1,\rho,l}(x)= \frac{\partial \psi_{\rho}}{\partial x_{l}}(x-z)$ for  every $x$ in $\Omega$. Then $a_{z,\rho,l}\in C^{\infty}_{c}(\Omega)$, $supp(a_{z,\rho,l})\subset B(z,\rho)$ and there is  a positive real number $M$ independent from $(z,\rho,l)$ such that $||a_{z,\rho,l}||_{L^{1}(\Omega)}\le M\rho^{-1}$ and
 \begin{equation}  |a_{z,\rho,l}|\leb{\eqref{psi7}} M\rho^{-N-1}\chi_{B(z,\rho)}.
 \label{zeta}
 \end{equation}\hk
 Applying  Proposition \ref{lem62xz} for   $a=a_{z,\rho,l}$, we get $v$ as in this proposition. Denote $v$ by $v_{z,\rho,l}$ and put
  \[H_{\rho,l}(x,z)= v_{z,\rho,l}(x)\hh\forall~(x,z)\in\Omega\times\Omega.\] \hk
   Then
 \begin{equation}
\int_{\Omega}\frac{\partial v_{z,\rho,l}}{\partial x_{i}}\frac{\partial \varphi}{\partial x_{i}}b^{ij}dx=\int_{\Omega}\frac{\partial \psi_{\rho}}{\partial x_{l}}(x-z)\varphi dx\hh\forall~\varphi \in W_{B,A}(\Omega). 
 \label{7zxb}
 \end{equation}\hk
 We have following properties of $H_{\rho,l}$.
  \begin{lemma} Let $\delta$ be  as in Lemma \ref{lemfs1b},   $\zeta \in (0,\min\{\delta,\frac{1}{120}\})$,  $A$ be  admissible with respect to $\Omega$. Let $N$ be in $\{2,\cdots,8\}$,  $t(\zeta)$ and $\overline{r}(\zeta)$ be as in Proposition \ref{lem62xz}, $t\in (t(\zeta),2)$, $r=\frac{t(N+1)-2}{N-t}$, $\overline{r}\in (2,\overline{r}(\zeta))$. Let $z\in \Omega$ such that $B(z,4\rho_{z})\subset \Omega$ with $\rho_{z}\in (0,1)$. Assume    \eqref{c1}-\eqref{c4}, \eqref{c9} and \eqref{thuyen} hold. Then there is a positive real number $\overline{C}$ independent from $\rho$ and $\alpha_{1}\in (1,\frac{N}{N-1+\zeta})$ such that
\begin{equation}
 | H_{\rho,l}(x,z)| \le  \overline{C}|x-z|^{-N+1-\zeta} \hh   ~x\in\Omega, 
 \label{81bzbz}
 \end{equation}
 \begin{equation}
 | H_{\rho,l}(x,z)| \le  \overline{C}s^{-N+1-\zeta} \hh   ~x\in(\Omega\setminus B(z,s)), s \in (\rho,diam(\Omega)], 
 \label{81bzbz2}
 \end{equation}
 \begin{equation}
 ||H_{\rho,l}(.,z)||_{L^{\alpha_{1}}(\Omega)} \le  \overline{C}, 
 \label{81bzbcz}
 \end{equation} 
  \begin{equation}
\{\int_{\Omega\setminus B(z,s)}|\nabla  H_{\rho,l}|^{2} bdx\}^{\frac{1}{2}}< \overline{C}s^{-N+1-\zeta}\hh if~~s> 2\rho
 \label{81cz}
\end{equation}
 \label{lem62xh}
\end{lemma}\hk
\begin{proof}
Let $C=C_{n,\zeta,b,\overline{b},\Omega}$ and $y\in \Omega$. We have 
\[|v_{z,\rho,l}(y)|\eqb{\eqref{834b},\eqref{7zxb}} |\int_{\Omega}G(x,y)\frac{\partial \psi_{\rho}}{\partial x_{l}}(x-z)dx|=|\int_{\Omega}\frac{\partial G}{\partial x_{l}}(x,y)\psi_{\rho}(x-z)dx|
\]
\[\leb{\eqref{thuyen},\eqref{zeta}}\hspace{-.1in}CM\rho^{-N}\int_{B(z,\rho)}\hspace{-.23in}|x-y|^{-N+1-\zeta}dx.\]\hk
Thus
\begin{equation}|v_{z,\rho,l}(y)|\leb{\eqref{thuyen},\eqref{zeta}}\hspace{-.1in}CM\rho^{-N}\int_{B(z,\rho)}\hspace{-.23in}|x-y|^{-N+1-\zeta}dx\le CM\rho^{-N}\int_{B(z-y,\rho)}\hspace{-.33in}|t|^{-N+1-\zeta}dt
\label{ho1}
\end{equation}\hk
Let $y \in B(z,2\rho)$. By \eqref{ho1} we have
\[|v_{z,\rho,l}(y)|\hspace{-.1in}\leb{B(z-y,\rho)\subset B(0,3\rho)}\hspace{-.23in}CM\rho^{-N}\hspace{-.05in}\int_{B(0,3\rho)}\hspace{-.23in} |t|^{-N+1-\zeta}dt=CM\rho^{-N}\int_{\Sigma_{N-1}}\int_{0}^{3\rho}\hspace{-.15in}\xi^{-N+1-\zeta+N-1}d\xi ds\]
\[=\frac{3^{1-\zeta}|\Sigma_{N-1}|}{1-\zeta}\rho^{-N+1-\zeta}\leb{|z-y|<2\rho}\frac{3^{1-\zeta}2^{N-1+\zeta}|\Sigma_{N-1}|}{1-\zeta}|z-y|^{-N+1-\zeta}.\]\hk
It implies \eqref{81bzbz} for every $y$ in $B(z,2\rho)$. Now let $y\in \Omega\setminus B(z,2\rho)$.  We have
\begin{equation}
|x-y|\ge |y-z| - |z-x|\ge \frac{1}{2}|y-z|\hh\forall~x\in B(z,\rho).
\label{psi11}
\end{equation}\hk
Thus
\[|v_{z,\rho,l}(y)|\hspace{-.2in}\leb{\eqref{ho1},\eqref{psi11}}\hspace{-.2in}2^{N-1+\zeta}CM\rho^{-N}\int_{B(z,\rho)}\hspace{-.23in}|y-z|^{-N+1-\zeta}dx=2^{N-1+\zeta}CM|y-z|^{-N+1-\zeta}.
\]\hk
It implies \eqref{81bzbz} for every $y$ in $\Omega\setminus B(z,2\rho)$.  Let $s\in (2\rho,diam(\Omega)]$ and $y\in (\Omega\setminus B(z,s))$. By above results, we get
\[|v_{z,\rho,l}(y)|\le 2^{N-1+\zeta}CM(\frac{|y-z|}{s})^{-N+1-\zeta}s^{-N+1-\zeta}\leb{-N+1-\zeta<0} 2^{N-1+\zeta}CMs^{-N+1-\zeta},\]
which implies \eqref{81bzbz2}.\\\hk
Using \eqref{81bzbz} and  polar coordinates in integration, we get \eqref{81bzbcz}. 
 Using \eqref{81bzbz2} and arguing as in the proofs of \eqref{le4460bzzz} and \eqref{le4460bzzzb}, we get \eqref{81cz}.
\end{proof}\hk
Now we obtain a function $H_{l}$ as follows.
  \begin{proposition}~~Let $N\in \{2,\cdots,8\}$,  notations and conditions be as in Lemma \ref{lem62xh}. Then there is a function $H_{l}$ on $\Omega\times \Omega$ having following properties.\\\hk 
 $(i)$~ There are $C_{1}>0$ and a sequence $\{\rho_{m}\}\subset (0,\rho_{z})$ converging to $0$ such that
 \begin{equation} \lim_{k\to\infty}H_{\rho,l}(x,z) = H_{l}(x,z)\hh a.e. ~x\in \Omega,
 \label{weak1z1}
 \end{equation}
\begin{equation} 
 |H_{l}(x,z)| \le  C_{1}|x-z|^{-N+1-\zeta}\hh\forall~x\in\Omega\setminus\{z\}. 
 \label{81bzz1}
\end{equation}\hk
$(ii)$~~  Let  $u\in W_{B,A}(\Omega)$ and $f\in L^{1}(\Omega)$ such that
\begin{equation}
\int_{\Omega} \frac{\partial u}{\partial x_{i}} \frac{\partial \varphi}{\partial x_{j}} b^{ij}dx=\int_{\Omega}f\varphi dx\hh\forall~\varphi\in W_{B,A}(\Omega).
\label{2834}
\end{equation}\\\hk
Then 
\begin{equation}
\int_{\Omega}H_{l}(x,z)f(x)dx= -\frac{\
\partial u}{\partial x_{l}}(z)\hh a.e~z\in \Omega,\label{2834b}
\end{equation}
\begin{equation}|\frac{\partial u}{\partial x_{l}}|\le C_{1}I_{\Omega,1-\zeta}(|f|)\hh\forall~l\in \{1,2,\cdots,N\} .
\label{2834bc}
\end{equation}\\\hk
$(iii)$~~$\nabla u\in L^{\frac{sN}{N-s(1-\zeta)}}(\Omega)$ if $f\in L^{s}(\Omega)$ and $s\in[1,\frac{N}{1-\zeta})$.\\\hk
$(iv)$~~$\nabla u\in L^{q}(\Omega)$ for  every $q$ in $[1,\infty)$ if $f\in L^{s}(\Omega)$ and $s\in[\frac{N}{1-\zeta},\infty)$.
\label{nab}
\end{proposition}
\begin{proof}
Arguing as in the proof of Proposition \ref{pro62} and using Lemma \ref{lem62xh}, we obtain $(i)$.\\\hk
$(ii)$~~Put
\[g_{1}(x)= \overline{C}\chi_{B(0,2diam(\Omega))\setminus \{0\}}(x)|x|^{-N+1-\zeta}\hh\forall~x\in \RR^{N}.\]\hk
 Using the polar coordinates with the pole at $0$, we have
 \begin{equation}\int_{\RR^{N}}g_{1}(x)dx =\int_{B(0,2diam(\Omega))}g_{1}(x)dx = \overline{C}|\Sigma_{N-1}| \int_{0}^{2diam(\Omega)}t^{-\zeta}dt
  \label{834c}
\end{equation}
  \[=\overline{C}\frac{|\Sigma_{N-1}|}{1-\zeta}(2diam(\Omega))^{1-\zeta}<\infty. \]
   \hk Thus, by Young's Theorem in \cite[Theorem 4.15, p. 104]{BR},
\[\int_{\RR^{N}}\int_{\RR^{N}}|g_{1}(z-y)f(y)|dydz\leb{\eqref{834bb},\eqref{834c}} ||g_{1}||_{L^{1}(\RR^{N})}||f||_{L^{1}(\RR^{N})} <\infty\hk\forall~f\in L^{1}(\RR^{N}).\]\hk
 Let $f\in L^{1}(\Omega)$. By Fubini's Theorem \cite[Theorem 4.5, p.91]{BR}, there is  a measurable subset $X_{f}$ of $\Omega$ such that $|X_{f}|=0$ and 
 \begin{equation}
\int_{\Omega}|g_{1}(z-y)f(y)|dy < \infty \hh\forall~z\in \Omega\setminus X_{f}.
\label{8834bbz}
\end{equation}\hk
By \eqref{81bzbz}, \eqref{weak1z1}, \eqref{8834bbz} and Lebesgue's dominated convergence theorem, we obtain 
\begin{equation}
\lim_{n\to\infty}\int_{\Omega}H_{\rho_{n},l}(y,z)f(y)dy = \int_{\Omega}H_{l}(y,z)f(y)dy  \hh\forall~z\in \Omega\setminus X_{f}.
\label{8834bz}
\end{equation}\hk
On other hand, by \eqref{rt}, $1<t$. It implies $1<\frac{t}{2-t}$. Since $u\in W_{B,A}(\Omega)$, we get
\begin{equation}
 \int_{\Omega}|\nabla u|dx \leb{Cauchy-Schwarz} \{\int_{\Omega}b^{-1}dx\}^{\frac{1}{2}}\{\int_{\Omega}|\nabla u|^{2}bdx\}^{\frac{1}{2}}\lleb{\eqref{c1},\eqref{c3}}\infty,
 \label{rtz}
 \end{equation}

By using $v_{z,\rho_{m},l}$ and $u$ as test functions for \eqref{2834} and \eqref{7zxb} respectively, and  by Theorem 7.7 in \cite[p.138]{RU}, there is $Y$ such that $X\subset Y\subset \Omega$, $|Y|=0$ and 
\[
\int_{\Omega}H_{l}(x,z)f(x)dx\eqb{\eqref{8834bz}}\lim_{m\to\infty}\int_{\Omega}v_{z,\rho_{m},l}(x)f(x)dx \eqb{\eqref{2834}}\lim_{m\to\infty}\int_{\Omega} \frac{\partial u}{\partial x_{i}}\frac{\partial v_{z,\rho_{m},l}}{\partial x_{j}}b^{ij}  dx
\]
\[\eqb{b^{ij}=b^{ji}}\lim_{m\to\infty}\int_{\Omega} \frac{\partial v_{z,\rho_{m},l}}{\partial x_{i}} \frac{\partial u}{\partial x_{j}}b^{ij} dx \eqb{\eqref{7zxb}}\lim_{m\to\infty}\hspace{-.05in}\int_{\Omega}\hspace{-.07in}\frac{\partial \psi_{\rho_{m}} }{\partial x_{l}}(x-z)u(x) dx\]
\[ \eqb{Lemma~\ref{5z}}\hspace{-.1in}-\lim_{m\to\infty}\hspace{-.05in}\int_{\Omega}\hspace{-.05in}\psi_{\rho_{m}}(x-z)\frac{\partial u }{\partial x_{l}}(x) dx\hspace{-.2in}\eqb{Theorem~ 7.7,\eqref{rtz}}\hspace{-.05in}-\frac{\partial u }{\partial x_{l}}(z)\hk a.e~z~in ~\Omega\setminus Y,\]\hk
 It implies \eqref{2834b}. By \eqref{81bzz1} and \eqref{2834b}, we get \eqref{2834bc}.\\\hk
 $(iii)$~~Let $s\in [1,\frac{N}{1-\zeta})$. Since $\frac{1}{s}-\frac{1-\zeta}{N}= \frac{N-s(1-\zeta)}{sN}$, by  Lemma \ref{Riez} and \eqref{2834bc}, we get $\nabla u\in L^{\frac{sN}{N-s(1-\zeta)}}(\Omega)$.\\\hk 
  $(iv)$ Let $s\in [\frac{N}{1-\zeta},\infty)$. By  Lemma \ref{Riez} and \eqref{2834bc}, we get $(iv)$.
\end{proof}\hk
\begin{remark} If $b$ and $\overline{b}$ are constant, $b^{ij}$ is Dini-continuous for every $i,j$ in $\{1,\cdots,N\}$ and $A=\partial\Omega$, we get Proposition \ref{nab} with $\zeta=0$ for every $N$ in $\{3,4,\cdots\}$ by using Theorem $(3.3)$ in \cite{GW} and the techniques in the proof of Proposition \ref{nab} 
\label{gw2}
\end{remark}
 To study the H\"{o}lder continuity of solutions for elliptic equations, we assume  
 that $\partial\Omega$ is  minimally smooth defined in \cite[p.189]{ST}. This condition is satisfied if  $\Omega$ is convex or $\partial\Omega$ is $C^{1}$ embedded in $\RR^{N}$(see \cite[p.189]{ST}).\\\hk
 We have the H\"{o}lder continuity of solution to elliptic equations as follows. 
 \begin{proposition} Let $N\in \{2,\cdots,8\}$, Let notations and conditions be as in Lemma \ref{lem62xh}. Let $\theta \in (\zeta,1)$, $f\in L^{\frac{N}{2-\theta}}(\Omega)$ and $u\in W_{B,A}(\Omega)\cap L^{\frac{N}{1-\theta+\zeta}}(\Omega)$. 
 Assume that $\partial\Omega$ is minimally smooth and
  \begin{equation}\int_{\Omega}\sum_{i,j=1}^{N}b^{ij} \frac{\partial u}{\partial x_{i}}\frac{\partial \varphi}{\partial x_{j}}dx = \hspace{-.05in} \int_{\Omega}\hspace{-.05in} f\varphi dx \hh\forall~\varphi\in W_{B,A}(\Omega), 
 \label{pr10z}
 \end{equation}\hk
 Then there  are  $C_{2}\in (0,\infty)$ and $\tau\in (0,1)$   such that \\\hh
$|u(x)-u(y)|\le C_{3}|x-y|^{\tau}\hh\forall~x,y\in\Omega.$
\label{prozz}
\end{proposition}
\begin{proof}
  By $(iv)$ of Proposition \ref{nab} (for $s=\theta$),  $\nabla u\in L^{\frac{N}{1-\theta+\zeta}}(\Omega)$. Because $u$ is in $L^{\frac{N}{1-\theta+\zeta}}(\Omega)$, $u$ is in the Sobolev space $W^{1,\frac{N}{1-\theta+\zeta}}(\Omega)$. Since $\partial\Omega$ is minimally smooth, by Whitney's Extension  theorem \cite[Theorem 5, p.181]{ST}, there is $v\in W^{1,\frac{N}{1-\theta+\zeta}}(\RR^{N})$ such that $u(x)=v(x)$ for every $x$ in $\Omega$. Let $B(0,s)$ be a ball containing $\Omega$, Since $\zeta<\theta$, $\frac{N}{1-\theta+\zeta} >N$. By Morrey's theorem \cite[Corollary 9.13, p.283]{BR}, there are  $C_{3}\in (0,\infty)$ and $\tau\in (0,1)$ such that
\[|v(x)-v(y)|\le C_{3}|x-y|^{\tau}\hh\forall~x,y\in B(0,2s).\]\hk
 Thus we get the proposition.
\end{proof}\hk
\begin{remark} Let $v\in L^{p_{1}}(\Omega)$, $w\in L^{p_{2}}(\Omega)$, $q$ and $s$ be in $[1,\infty)$ such that $\frac{1}{p_{1}}+ \frac{1}{p_{2}}=\frac{1}{q}\le \frac{1}{s}$. By generalized H\"{o}der's inequality  in \cite[Exercise 4.4,p.126]{BR}, $vw$ is in $L^{q}(\Omega)\subset L^{s}(\Omega)$.
\label{r40}
\end{remark}
\begin{remark} Let  $\theta\in (0,\frac{1}{4})$, $\zeta=\frac{\theta}{2}$,  $\epsilon\in (0,\frac{2\theta}{N})$.   We have 
\begin{equation}\frac{N-t'}{t'(N+1)-2}
> \frac{N-t}{t(N+1)-2}\hh\forall~t'<t<2,
 \label{r41}
 \end{equation}
\begin{equation}\frac{N-t}{t(N+1)-2}~~ \mathop  \downarrow \limits_{t\to 2} ~~\frac{N-2}{2N} ,
 \label{r41a}
 \end{equation}
\begin{equation}(\frac{1+\epsilon}{2} - \frac{1}{N})=(\frac{N-2}{2N}+ \frac{\epsilon}{2})~~\mathop  \downarrow \limits_{\epsilon\to 0}~~\frac{N-2}{2N}.
 \label{r41c}
 \end{equation}\hk
We  choose $t(\zeta)$ in Proposition \ref{lem62xz} such that\\\hh
$\displaystyle\frac{N-t(\zeta)}{t(\zeta)(N+1)-2} \le \frac{N-2}{2N} +\frac{\theta}{N}.$
\\\hk
Let $t\in (t(\zeta),2)$. By \eqref{r41} and\eqref{r41a},\\\hh
$\displaystyle\frac{N-t}{t(N+1)-2}\in (\frac{N-2}{2N}, \frac{N-2}{2N} +\frac{\theta}{N}].$\\\hk
Thus, by \eqref{r41c}, we can choose $\epsilon\in (0,\frac{2\theta}{N})$ such that 
\begin{equation}\frac{1}{r}=\frac{N-t}{t(N+1)-2}=\frac{1+\epsilon}{2}-\frac{1}{N}. 
\label{r41d}
 \end{equation}
 \label{rep}
\end{remark}\hk
Now we can apply Proposition \ref{prozz} to a class of $f$.
 \begin{proposition} Let $N\in \{2,\cdots,8\}$, notations and conditions be as in Lemma \ref{lem62xh}, $\theta  \in (0,\frac{1}{4})$,  $\zeta =\frac{1}{4}\theta$, $\epsilon$ and $r$ be in Remark \ref{rep},   $a_{0}$, $a_{1}$, $a_{2}$  be  non-negative measurable functions on $\Omega$, $f$  be  a measurable function on $\Omega$ and $u\in W_{B,A}(\Omega)$.  Assume $b^{-1}\in L^{\frac{1}{\epsilon}}(\Omega)$,  $\displaystyle |f| \le a_{2}|\nabla u| +a_{1}|u| +a_{0}$, $\partial\Omega$ is minimally smooth,
 \begin{equation}
 a_{0}+a_{1}+ a_{2}^{2} \in L^{\frac{N}{2-2\theta}}(\Omega)\subset L^{\frac{N}{2-\theta}}(\Omega),
 \label{pr10zb0}
 \end{equation}
  \begin{equation}\int_{\Omega}\sum_{i,j=1}^{N}b^{ij} \frac{\partial u}{\partial x_{i}}\frac{\partial \varphi}{\partial x_{j}}dx = \hspace{-.05in} \int_{\Omega}\hspace{-.05in} f\varphi dx \hh\forall~\varphi\in W_{B,A}(\Omega), 
 \label{pr10zb}
 \end{equation}\hk
 Then there  are  $C\in (0,\infty)$ and $\tau\in (0,1)$   such that \\\hh
$|u(x)-u(y)|\le C|x-y|^{\tau}\hh\forall~x,y\in\Omega.$
\label{prozzbb}
\end{proposition}
\begin{proof}~~Let $\overline{C}(B,r,\overline{r},\zeta,\Omega)$ and $C_{1}$ be as in \eqref{81bzz} and \eqref{2834bc}. Put $M= N[\overline{C}(B,r,\overline{r},\zeta,\Omega) +C_{1}]$. We have
\begin{equation}
 \int_{\Omega}|\nabla u|^{\frac{2}{1+\epsilon}}dx \leb{H\ddot{o}lder} \{\int_{\Omega}b^{-\frac{1}{\epsilon}}dx\}^{\frac{\epsilon}{1+\epsilon}}\{\int_{\Omega}|\nabla u
 |^{2}bdx\}^{\frac{1}{1+\epsilon}}<\infty,
 \label{gz1}
 \end{equation}
 \begin{equation}0<\frac{1-\theta}{N}+\frac{1+\epsilon}{2}=\frac{1}{N}+\frac{1}{2}-\frac{\theta}{N}+\frac{\epsilon}{2}\lleb{N\ge 2,\epsilon< \frac{2\theta}{N}}1, \label{gz1b}
 \end{equation}\hk
 Thus, by Remark \ref{r40} and \eqref{pr10zb0}, $a_{2}|\nabla u|$ is in $L^{1}(\Omega)$. On other hand 
\begin{equation}0<\frac{2-\theta}{N}+\frac{1}{r}\eqb{\eqref{r41d}} \frac{2-\theta}{N} +\frac{1+\epsilon}{2}-\frac{1}{N}=\frac{1}{N}+\frac{1}{2}-\frac{\theta}{N}+\frac{\epsilon}{2}< 1. 
 \label{gz1c}
 \end{equation}\hk
 Therefore, by Remark \ref{r40} and \eqref{pr10zb0}, $a_{1}| u|$ is in $L^{1}(\Omega)$. Since $a_{0}\in L^{1}(\Omega)$,  $(a_{2}|\nabla u| +a_{1}|u|+a_{0})\in L^{1}(\Omega)$. Thus $f\in L^{1}(\Omega)$ and we can apply Proposition \ref{pro62} and Proposition \ref{nab} to $f$ and get 
\begin{equation}
 |u|\leb{\eqref{81bzz}}M I_{\Omega,2-\zeta}f\leb{\eqref{pr10zb0}} M I_{\Omega,2-\zeta}(a_{2}|\nabla u| +a_{1}|u| +a_{0}), 
 \label{gfz1}
 \end{equation}
 \begin{equation}
 |\nabla u|\leb{{2834bc}}M I_{\Omega,1-\zeta}f\leb{\eqref{pr10zb0}} M I_{\Omega,1-\zeta}(a_{2}|\nabla u| +a_{1}|u| +a_{0}), 
 \label{gfz2}
 \end{equation}\hk
 We have
 \begin{equation}
  (\frac{1-\theta}{N}+ \frac{1+\epsilon}{2})-\frac{2-\zeta}{N}=\frac{1+\epsilon}{2}-\frac{1}{N}-\frac{\theta-\zeta}{N}\eqb{Remark~ \ref{rep}}\frac{1}{r}-\frac{\theta-\zeta}{N} ,
 \label{gfz3}
 \end{equation}
 \begin{equation}
  (\frac{2-\theta}{N}+ \frac{1}{r})-\frac{2-\zeta}{N}=\frac{1}{r}-\frac{\theta-\zeta}{N}  , 
 \label{gfz4}
 \end{equation}
 \begin{equation}
  \frac{2-\theta}{N}-\frac{2-\zeta}{N}=-\frac{\theta-\zeta}{N} <0 . 
 \label{gfz4b}
 \end{equation}\hk
  By Lemma \ref{Riez}, Remark \ref{r40},  \eqref{pr10zb0}, Proposition \ref{pro62} and  \eqref{gfz3}-\eqref{gfz4b}, we have  $I_{\Omega,2-\zeta}(a_{2}|\nabla u|)$, $I_{\Omega,2-\zeta}(a_{1}|u|)$ and $I_{\Omega,2-\zeta}(a_{0})$ are in $L^{(\frac{1}{r}-\frac{\theta-\zeta}{N})^{-1}}(\Omega)$. Thus, by \eqref{gfz1}, $u$ is in $L^{(\frac{1}{r}-\frac{\theta-\zeta}{N})^{-1}}(\Omega)$. On other hand
  \begin{equation}
  (\frac{1-\theta}{N}+ \frac{1+\epsilon}{2})-\frac{1-\zeta}{N}=\frac{1+\epsilon}{2}-\frac{\theta-\zeta}{N} \eqb{Remark~ \ref{rep}}\frac{1}{r}+\frac{1}{N}-\frac{\theta-\zeta}{N},
 \label{gfz5}
 \end{equation}
 \begin{equation}
  [\frac{2-\theta}{N}+ (\frac{1}{r}-\frac{\theta-\zeta}{N})]-\frac{1-\zeta}{N}=\frac{1}{r}+\frac{1}{N}-\frac{\theta-\zeta}{N} , 
 \label{gfz6}
 \end{equation}
  \begin{equation}
  \frac{2-\theta}{N} -\frac{1-\zeta}{N}=\frac{1}{N}-\frac{\theta-\zeta}{N}<\frac{1}{r}+\frac{1}{N}-\frac{\theta-\zeta}{N} . 
 \label{gfz7}
 \end{equation}\hk
  By Lemma \ref{Riez}, Remark \ref{r40},  \eqref{pr10zb0}, Proposition \ref{pro62} and \eqref{gfz5}-\eqref{gfz7},  we have  $I_{\Omega,1-\zeta}(a_{2}|\nabla u|)$, $I_{\Omega,1-\zeta}(a_{1}|u|)$ and $I_{\Omega,1-\zeta}(a_{0})$ are in $L^{(\frac{1}{r}+\frac{1}{N}-\frac{\theta-\zeta}{N})^{-1}}\hspace{-.039in}(\Omega)$.    Thus, by \eqref{gfz2},  $\nabla u\in L^{(\frac{1}{r}+\frac{1}{N}-\frac{\theta-\zeta}{N})^{-1}}\hspace{-.039in}(\hspace{-.01in}\Omega)$.\\\hk
  Replacing $\frac{1+\epsilon}{2}$ and $\frac{1}{r}$ by $\frac{1}{r}+\frac{1}{N}-\frac{\theta-\zeta}{N}$ and $\frac{1}{r}-\frac{\theta-\zeta}{N}$ respectively, we get
   \begin{equation}
  [\frac{1-\theta}{N}+ (\frac{1}{r}+\frac{1}{N}-\frac{\theta-\zeta}{N})]-\frac{2-\zeta}{N}=\frac{1}{r}-2\frac{\theta-\zeta}{N} ,
 \label{gfz32zb}
 \end{equation}
 \begin{equation}
  [\frac{2-\theta}{N}+(\frac{1}{r}-\frac{\theta-\zeta}{N})] -\frac{2-\zeta}{N}=\frac{1}{r}-2\frac{\theta-\zeta}{N}  , 
 \label{gfz42zb}
 \end{equation}
  \begin{equation}
  \frac{2-\theta}{N} -\frac{2-\zeta}{N}=-\frac{\theta-\zeta}{N}<0. 
 \label{gfz42zbb}
 \end{equation}\hk
  By Lemma \ref{Riez}, Remark \ref{r40},  \eqref{pr10zb0}, Proposition \ref{pro62} and  \eqref{gfz32zb}-\eqref{gfz42zbb}, we have $I_{\Omega,2-\zeta}(a_{2}|\nabla u|)$, $I_{\Omega,2-\zeta}(a_{1}|u|)$ and $I_{\Omega,2-\zeta}(a_{0})$ are in $L^{(\frac{1}{r}-2\frac{\theta-\zeta}{N})^{-1}}(\Omega)$. Thus, by \eqref{gfz1}, $u\in L^{(\frac{1}{r}-2\frac{\theta-\zeta}{N})^{-1}}(\Omega)$.
 On other hand
  \begin{equation}
  [\frac{1-\theta}{N}+ (\frac{1}{r}+\frac{1}{N}-\frac{\theta-\zeta}{N})]-\frac{1-\zeta}{N}=\frac{1}{r}+\frac{1}{N}-2\frac{\theta-\zeta}{N},
 \label{gfz5zb}
 \end{equation}
 \begin{equation}
  [\frac{2-\theta}{N}+ (\frac{1}{r}-\frac{\theta-\zeta}{N})]-\frac{1-\zeta}{N}=\frac{1}{r}+\frac{1}{N}-2\frac{\theta-\zeta}{N} , 
 \label{gfz6zb}
 \end{equation}
  \begin{equation}
  \frac{2-\theta}{N} -\frac{1-\zeta}{N}=\frac{1}{N}-\frac{\theta-\zeta}{N} <\frac{1}{r}+\frac{1}{N}-2\frac{\theta-\zeta}{N},\hk~if ~ \frac{1}{r}-\frac{\theta-\zeta}{N}>0. 
 \label{gfz7zb}
 \end{equation}\hk
  By Lemma \ref{Riez}, Remark \ref{r40},  \eqref{pr10zb0}, Proposition \ref{pro62},  \eqref{gfz5zb}-\eqref{gfz7zb}, we have $I_{\Omega,1-\zeta}(a_{2}|\nabla u|)$, $I_{\Omega,1-\zeta}(a_{1}|u|)$ and $I_{\Omega,1-\zeta}(a_{0})$ are in $L^{(\frac{1}{r}+\frac{1}{N}-2\frac{\theta-\zeta}{N})^{-1}}(\Omega)$.    Thus, by \eqref{gfz2},  $\nabla u$ belongs to  $L^{(\frac{1}{r}+\frac{1}{N}-2\frac{\theta-\zeta}{N})^{-1}}(\Omega)$, if $\frac{1}{r}-\frac{\theta-\zeta}{N}>0$.
\\\hk 
  Put $m_{0}= \max\{m: \frac{1}{r}-m\frac{\theta-\zeta}{N}>0\}$. By mathematical induction,  $u\in L^{(\frac{1}{r}-m_{0}\frac{\theta-\zeta}{N})^{-1}}(\Omega)$ and  $\nabla u\in L^{(\frac{1}{r}+\frac{1}{N}-m_{0}\frac{\theta-\zeta}{N})^{-1}}(\Omega)$. Since $\frac{1}{r}-m_{0}\frac{\theta-\zeta}{N}-\frac{\theta-\zeta}{N}= \frac{1}{r}-(m_{0}+1)\frac{\theta-\zeta}{N}\le 0$, we have $\frac{1}{r}-m_{0}\frac{\theta-\zeta}{N}\le \frac{\theta-\zeta}{N}$. Thus $u\in L^{(\frac{\theta-\zeta}{N})^{-1}}(\Omega)$ and  $\nabla u\in L^{(\frac{1}{N}+\frac{\theta-\zeta}{N})^{-1}}(\Omega)=L^{\frac{N}{1+\theta-\zeta}}(\Omega)$. \\\hk
   Let $\overline{\theta} =\frac{1}{2}\theta$. We have 
   \begin{equation}[\frac{1-\theta}{N}+(\frac{1}{N}+\frac{\theta-\zeta}{N})]-\frac{2-\zeta}{N}=0,
   \label{gfz42zbbz}
 \end{equation}
 \begin{equation}(\frac{2-\theta}{N}+\frac{\theta-\zeta}{N})-\frac{2-\zeta}{N}=0,\label{gfz42zbbz1}
 \end{equation}
   \begin{equation}
  \frac{2-\theta}{N} -\frac{2-\zeta}{N}=-\frac{\theta-\zeta}{N}<0. 
 \label{gfz42zbbz2}
 \end{equation}\hk
  By Lemma \ref{Riez},  \eqref{pr10zb0},  Proposition \ref{pro62} and  \eqref{gfz42zbbz}-\eqref{gfz42zbbz2}, $I_{\Omega,2-\zeta}(a_{2}|\nabla u|)$, $I_{\Omega,2-\zeta}(a_{1}|u|)$ and $I_{\Omega,2-\zeta}(a_{0})$ are in $L^{\frac{N}{1-\overline{\theta}+\zeta}}(\Omega)\cap L^{\frac{N}{\theta-\overline{\theta}}}(\Omega)$. Thus, by \eqref{gfz1}, 
 \begin{equation}u\in L^{\frac{N}{1-\frac{1}{2}\theta}}(\Omega)\cap L^{\frac{N}{\frac{1}{4}\theta}}(\Omega).
 \label{gfz42zbbz3b}
 \end{equation}\hk
   On other hand
\begin{equation}\frac{1-\theta}{N}+(\frac{1}{N}+\frac{\theta-\zeta}{N})=\frac{2-\zeta}{N},
\label{gfz42zbbz3}
 \end{equation}
\begin{equation}\frac{2-\theta}{N}+\frac{\frac{1}{4}\theta}{N}=\frac{2-\frac{3}{4}\theta}{N}\lleb{\zeta=\frac{1}{4}\theta}\frac{2-\zeta}{N} ,
\label{gfz42zbbz4}
 \end{equation}
\begin{equation}\frac{2-\theta}{N} \lleb{\zeta=\frac{1}{4}\theta}\frac{2-\zeta}{N}. 
\label{gfz42zbbz5}
 \end{equation}\hk
 Thus $f\in L^{\frac{2-\zeta}{N}}(\Omega)$. Note that\\\hh
$\displaystyle\frac{1-\zeta}{N} -\frac{2-\zeta}{N}=\frac{1}{N}.$\\\hk
 Therefore, by \eqref{gfz2}, $\nabla u$ is in $L^{\frac{1}{N}}(\Omega)$. We have
 \begin{equation}\frac{1-\theta}{N}+\frac{1}{N}=\frac{2-\theta}{N},
\label{gfz42zbbz6}
\end{equation}\hk 
 Combining \eqref{gfz42zbbz6}, \eqref{gfz42zbbz4} and \eqref{gfz42zbbz5}, we get $f\in L^{\frac{2-\frac{3}{4}\theta}{N}}(\Omega)$. Note that
 \[\frac{2-\frac{3}{4}\theta}{N}-\frac{1-\zeta}{N}\eqb{\zeta=\frac{1}{4}\theta}\frac{1-\frac{1}{2}\theta}{N}.\]\hk
 Therefore, by \eqref{gfz42zbbz3b} and\eqref{gfz2}, $u$ and  $\nabla u$ are in $L^{\frac{N}{1-\frac{1}{2}\theta}}(\Omega)$ or $u\in W^{1,\frac{N}{1-\frac{1}{2}\theta}}(\Omega)$. Now arguing as in the proof of Proposition \ref{prozz},  we obtain the desired result.
\end{proof}
\begin{remark} If $b$ and $\overline{b}$ are constant, $b^{ij}$ is Dini-continuous for every $i,j$ in $\{1,\cdots,N\}$ and $A=\partial\Omega$, we get Propositions \ref{prozz} and \ref{prozzbb} with $\zeta=0$ for every $N$ in $\{3,4,\cdots\}$ by using Theorems $(1.1)$ and $(3.3)$  in \cite{GW} and the techniques in the proof of Propositions \ref{prozz} and \ref{prozzbb}.
\label{gw3}
\end{remark}
  \bibliographystyle{amsplain}

\newpage

 \begin{center}{\large\bf Details of the use the Mathematica in the paper}\vspace{.3in}\end{center}
\hk{\bf Proof of (4.33)}\\
In[2]:= $D[t^{1/2}*(t^2 - 10/3*t + 5), t]$\\
Out[2]= $\sqrt{t} (-10/3 + 2 t) + (5 - (10 t)/3 + t^2)/(2 \sqrt{t})$\\
In[3]:= $Expand[\sqrt{t} (-10/3 + 2 t) + (5 - (10 t)/3 + t^2)/(2 \sqrt{t})]$\\
Out[3]= $5/(2 \sqrt{t}) - 5 \sqrt{t} + (5 t^{3/2})/2$\\
Thus\\
$D[t^{1/2}*(t^2 - 10/3*t + 5), t] = t^{-1/2}[(5 t^{2})/2-5t+ 5/2]$\\
\hk{\bf Proof of (3.34)}\\
In[1]:= $D[5/(2 \sqrt{t}) - 5 \sqrt{t} + (5 t^(3/2))/2, t]$\\
Out[1]=$ -5/(4 t^{3/2}) - 5/(2 \sqrt{t}]) + (15 \sqrt{t})/4$\\
Thus\\
 $D[5/(2 \sqrt{t}) - 5 \sqrt{t} + (5 t^{3/2})/2, t]= t^{-3/2}[-5/4 - 5t/2  + (15 t^{2})/4]$\\\hk
 {\bf Proof of (4.42)}\\
In[2]:= $D[t^{s + 1/2} (t^2 - 10/3*t + 5), t]$\\
Out[2]= $t^{
  1/2 + s} (-10/3 + 2 t) + (1/2 + s) t^{-1/2 + 
   s} (5 - (10 t)/3 + t^2)$\\
In[3]:= $Expand[
 t^{1/2 + s} (-10/3 + 2 t) + (1/2 + s) t^{-1/2 +  s} (5 - (10 t)/3 + t^2)]$\\
Out[3]= $5/2 t^{-1/2 + s} + 5 s t^{-1/2 + s} - 5 t^{1/2 + s} - 
 10/3 s t^{1/2 + s} + 5/2 t^{3/2 + s} + s t^{3/2 + s}$\\
In[4]:= $Collect[
 5/2 t^{-1/2 + s} + 5 s t^{-1/2 + s} - 5 t^{1/2 + s} - 
  10/3 s t^{1/2 + s} + 5/2 t^{3/2 + s} + s t^{3/2 + s}, t]$\\
Out[4]=$ -5 t^{1/2 + s} - 10/3 s t^{1/2 + s} + 
 t (5/2 t^{1/2 + s} + s t^{1/2 + s}) + (
 5/2 t^{1/2 + s} + 5 s t^{1/2 + s})/t$\\
Thus\\
$D[t^{s + 1/2} (t^2 - 10/3*t + 5), t]=$\\
$=-5 t^{1/2 + s} - 10/3 s t^{1/2 + s} + 
 t (5/2 t^{1/2 + s} + s t^{1/2 + s}) + (
 5/2 t^{1/2 + s} + 5 s t^{1/2 + s})/t$\\
 $=t^{-1/2 + s}[-(5+\frac{10s}{3})t +(\frac{5}{2}+s)t^{2}+ (\frac{5}{2}+5s)]$\\\hk
{\bf Proof of (4.44)}\\
In[6]:= $D[-5 t^{1/2 + s} - 10/3 s t^{1/2 + s} + 
  t (5/2 t^{1/2 + s} + s t^{1/2 + s}) + (
  5/2 t^{1/2 + s} + 5 s t^{1/2 + s})/t, t]$\\
Out[6]= $-5 (1/2 + s) t^{-1/2 + s} - 10/3 s (1/2 + s) t^{-1/2 + s} + 
 5/2 t^{1/2 + s} + s t^{1/2 + s} + 
 t (5/2 (1/2 + s) t^{-1/2 + s} + s (1/2 + s) t^{-1/2 + s}) + (
 5/2 (1/2 + s) t^{-1/2 + s} + 5 s (1/2 + s) t^{-1/2 + s})/t - (
 5/2 t^{1/2 + s} + 5 s t^{1/2 + s})/t^2$\\
In[7]:= $Expand[-5 (1/2 + s) t^{-1/2 + s} - 
  10/3 s (1/2 + s) t^{-1/2 + s} + 5/2 t^{1/2 + s} + s t^{1/2 + s} + 
  t (5/2 (1/2 + s) t^{-1/2 + s} + s (1/2 + s) t^{-1/2 + s}) + (
  5/2 (1/2 + s) t^{-1/2 + s} + 5 s (1/2 + s) t^{-1/2 + s})/t - (
  5/2 t^{1/2 + s} + 5 s t^{1/2 + s})/t^2]$\\
Out[7]= $-5/4 t^{-3/2 + s} + 5 s^2 t^{-3/2 + s} - 5/2 t^{-1/2 + s} - 
 20/3 s t^{-1/2 + s} - 10/3 s^2 t^{-1/2 + s} + 15/4 t^{1/2 + s} + 
 4 s t^{1/2 + s} + s^2 t^{1/2 + s}$\\
In[8]:= $Collect[-5/4 t^{-3/2 + s} + 5 s^2 t^{-3/2 + s} - 
  5/2 t^{-1/2 + s} - 20/3 s t^{-1/2 + s} - 10/3 s^2 t^{-1/2 + s} + 
  15/4 t^{1/2 + s} + 4 s t^{1/2 + s} + s^2 t^{1/2 + s}, t]$\\
Out[8]= $15/4 t^{1/2 + s} + 4 s t^{1/2 + s} + 
 s^2 t^{1/2 + s} + (-5/2 t^{1/2 + s} - 20/3 s t^{1/2 + s} - 
  10/3 s^2 t^{1/2 + s})/t + (-5/4 t^{1/2 + s} + 5 s^2 t^{1/2 + s})/t^2$\\
  Thus\\
  $D[-5 t^{1/2 + s} - 10/3 s t^{1/2 + s} + 
  t (5/2 t^{1/2 + s} + s t^{1/2 + s}) + (
  5/2 t^{1/2 + s} + 5 s t^{1/2 + s})/t, t]=$\\
$=15/4 t^{1/2 + s} + 4 s t^{1/2 + s} + 
 s^2 t^{1/2 + s} + (-5/2 t^{1/2 + s} - 20/3 s t^{1/2 + s} - 
  10/3 s^2 t^{1/2 + s})/t + (-5/4 t^{1/2 + s} + 5 s^2 t^{1/2 + s})/t^2$\\
$=t^{-3/2 + s}[(15/4 + 4s + s^2)t^{2}-(5/2 + 20s/3  + 10s^2/3 )t -5/4 + 5 s^2]
$\vspace{.1in}\\\hk
{\bf Proof of (4.49)}\\
In[21]:$= 
D[t^{2 s}*(t^2 - 10/3*t + 5)*((5/2 + s)*t^2 - (5 + (10*s)/3)*t + 5/
    2 + 5*s), t]$//
Out[21]= $t^{2 s} (-5 - (10 s)/3 + 2 (5/2 + s) t) (5 - (10 t)/3 + t^2) + 
 t^{2 s} (-10/3 + 2 t) (5/2 + 
    5 s - (5 + (10 s)/3) t + (5/2 + s) t^2) + 
 2 s t^{-1 +   2 s} (5 - (10 t)/3 + t^2) (5/2 + 
    5 s - (5 + (10 s)/3) t + (5/2 + s) t^2)$\\
In[22]:=$ Expand[
 t^{2 s} (-5 - (10 s)/3 + 2 (5/2 + s) t) (5 - (10 t)/3 + t^2) + 
  t^{2 s} (-10/3 + 2 t) (5/2 + 
     5 s - (5 + (10 s)/3) t + (5/2 + s) t^2) + 
  2 s t^{-1 +  2 s} (5 - (10 t)/3 + t^2) (5/2 + 
     5 s - (5 + (10 s)/3) t + (5/2 + s) t^2)]$\\
Out[22]= $-(100 t^{2 s})/3 - 100 s t^{2 s} - 200/3 s^2 t^{2 s} + 
 25 s t^{-1 + 2 s} + 50 s^2 t^(-1 + 2 s) + 190/3 t^(1 + 2 s) + 
 950/9 s t^{1 + 2 s} + 380/9 s^2 t^{1 + 2 s} - 40 t^{2 + 2 s} - 
 140/3 s t^{2 + 2 s} - 40/3 s^2 t^{2 + 2 s} + 10 t^{3+ 2 s} + 
 9 s t^{3+ 2 s} + 2 s^2 t^{3+ 2 s}$\\
In[23]:=$ 
Collect[-(100 t^{2 s})/3 - 100 s t^{2 s} - 200/3 s^2 t^{2 s} + 
  25 s t^(-1 + 2 s) + 50 s^2 t^(-1 + 2 s) + 190/3 t^(1 + 2 s) + 
  950/9 s t^(1 + 2 s) + 380/9 s^2 t^(1 + 2 s) - 40 t^{2 + 2 s} - 
  140/3 s t^{2 + 2 s} - 40/3 s^2 t^{2 + 2 s} + 10 t^{3+ 2 s} + 
  9 s t^{3+ 2 s} + 2 s^2 t^{3+ 2 s}, t]$\\
Out[23]= $-(100 t^{2 s})/3 - 100 s t^{2 s} - 200/3 s^2 t^{2 s} + 
 t^2 (-40 t^{2 s} - 140/3 s t^{2 s} - 40/3 s^2 t^{2 s}) + 
 t^3 (10 t^{2 s} + 9 s t^{2 s} + 2 s^2 t^{2 s}) + 
 t ((190 t^{2 s})/3 + 950/9 s t^{2 s} + 380/9 s^2 t^{2 s}) + (
 25 s t^{2 s} + 50 s^2 t^{2 s})/t
$\\
$=t^{2 s-1}[(-(100 )/3 - 100 s  - 200/3 s^2)t  + 
 t^3 (-40  - 140/3 s  - 40/3 s^2) + 
 t^4 (10  + 9 s  + 2 s^2 ) + 
 t^{2} ((190 )/3 + 950/9 s  + 380/9 s^2 ) + (
 25 s  + 50 s^2 )]$\\
Thus we get $(4.49)$.\\\hk
{\bf Proof of $(4.52)$}\\
In[2]:= $Expand[
 2*(t^2 - 10/3*t + 1)^2 + 16*(t - 80/(3*16))^2 - 16*(5/3)^2 + 48]$\\
Out[2]= $50 - (200 t)/3 + (380 t^2)/9 - (40 t^3)/3 + 2 t^4$\\
In[1]:= $N[-16*(5/3)^2 + 48]$\\
Out[1]= 3.55556\\
Thus we get $(4.52)$.\\\hk
{\bf Proof of $(4.53)$}\\
In[3]:= $Expand[
 25 s + 50 s^2 + (-100/3 - 100 s - (200 s^2)/3) t + (190/3 + (950 s)/
     9 + (380 s^2)/9) t^2 + (-40 - (140 s)/3 - (40 s^2)/
     3) t^3 + (10 + 9 s + 2 s^2) t^4]$\\
Out[3]= $25 s + 50 s^2 - (100 t)/3 - 100 s t - (200 s^2 t)/3 + (
 190 t^2)/3 + (950 s t^2)/9 + (380 s^2 t^2)/9 - 40 t^3 - (
 140 s t^3)/3 - (40 s^2 t^3)/3 + 10 t^4 + 9 s t^4 + 2 s^2 t^4$\\
In[4]:= $Collect[
 25 s + 50 s^2 - (100 t)/3 - 100 s t - (200 s^2 t)/3 + (190 t^2)/3 + (
  950 s t^2)/9 + (380 s^2 t^2)/9 - 40 t^3 - (140 s t^3)/3 - (
  40 s^2 t^3)/3 + 10 t^4 + 9 s t^4 + 2 s^2 t^4, s]$\\
Out[4]= $-(100 t)/3 + (190 t^2)/3 - 40 t^3 + 10 t^4 + 
 s^2 (50 - (200 t)/3 + (380 t^2)/9 - (40 t^3)/3 + 2 t^4) + 
 s (25 - 100 t + (950 t^2)/9 - (140 t^3)/3 + 9 t^4)$\\\hk
  Thus we get $(4.53)$.\\\hk
  {\bf Proof of $(4.54)$}\\
In[1]:= $Expand[11*( t^2  - 30/11 t + 2)^2  + 
  2174 /99 (t - (99*70 )/(3*2174))^2 + 6747/1087]$\\
Out[1]:= $
75 - (500 t)/3 + (1330 t^2)/9 - 60 t^3 + 11 t^4$\\\hk
 Thus we get $(4.54)$.\\\hk   
  {\bf Proof of Step 2 in Lemma  11}\\
By $(4.49)$, we have\\
$h(1,\frac{2}{3},t)= k(\frac{2}{3})= 25\frac{2}{3} + 50 (\frac{2}{3})^2 - (100 t)/3 - 100\frac{2}{3} t - (200 (\frac{2}{3})^2 t)/3 + (
 190 t^2)/3 + (950 \frac{2}{3} t^2)/9 + (380 (\frac{2}{3})^2 t^2)/9 - 40 t^3 - (
 140 \frac{2}{3} t^3)/3 - (40 (\frac{2}{3})^2 t^3)/3 + 10 t^4 + 9 \frac{2}{3} t^4 + 2 (\frac{2}{3})^2 t^4$\\
 In[5]:=$ Expand[
 25*2/3 + 50*(2/
    3)^2 + (-100/3 - 100* 2/3 - (200*(2/3)^2)/3) t + (190/3 + (
     950*2/3)/9 + (380*(2/3)^2)/9) t^2 + (-40 - (140*2/3)/3 - (
     40*(2/3)^2)/3) t^3 + (10 + 9*2/3 + 2*(2/3)^2) t^4]$\\
Out[5]= $350/9 - (3500 t)/27 + (12350 t^2)/81 - (2080 t^3)/27 + ( 152 t^4)/9$\\
Thus\\
$h(1,\frac{2}{3},t)=350/9 - (3500 t)/27 + (12350 t^2)/81 - (2080 t^3)/27 + ( 152 t^4)/9$\\
On other hand\\
In[6]:= $Expand[
 152/9*(t^2 - 1040/456*t + 5/4)^2 + 3830/171*(t - 285/383)^2 - 
  3830/171*(285/383)^2 + 225/18]$\\
Out[6]= $350/9 - (3500 t)/27 + (12350 t^2)/81 - (2080 t^3)/27 + ( 152 t^4)/9.$\\
In[15]:= $N[- 3830/171*(285/383)^2 + 225/18]$\\
Out[15]= $0.0979112$\\
It implies\\
$h(1,\frac{2}{3},t)= 152/9*(t^2 - 1040/456*t + 5/4)^2 + 3830/171*(t - 285/383)^2 - 
  3830/171*(285/383)^2 + 225/18]\ge0.0979112>10^{-2}. $\\\hk
  {\bf Proof of  Lemma  12}\\
In[6]:=$ Expand[
 5*(t^2 - 10/3*t + 5) - ((5/2 + s) t^2 - (5 + (10 s)/3)*t + 5/2 + 
    5*s)]$\\
Out[6]= $45/2 - 5 s - (35 t)/3 + (10 s t)/3 + (5 t^2)/2 - s t^2$\\
In[7]:= $Collect[
 45/2 - 5 s - (35 t)/3 + (10 s t)/3 + (5 t^2)/2 - s t^2, t]$\\
Out[7]=$ 45/2 - 5 s + (-35/3 + (10 s)/3) t + (5/2 - s) t^2$.

\end{document}